\documentclass[11pt]{amsart} 
\usepackage[lmargin=1in,rmargin=1in,tmargin=1in,bmargin=1in]{geometry}

\usepackage[utf8]{inputenc}
\usepackage{amsfonts}
\usepackage{amsmath}
\usepackage{amsthm}
\usepackage{amssymb}
\usepackage{mathtools}
\usepackage{multicol}
\usepackage{comment}
\usepackage{fullpage}
\usepackage{tikz}
\usepackage{xcolor}
\usepackage{hyperref}
\numberwithin{equation}{section}
\numberwithin{figure}{section}
\numberwithin{table}{section}
\usetikzlibrary{patterns}
\usepackage{tikz-3dplot}
\usepackage{subcaption}
\usepackage{enumerate}
\usepackage{dsfont}

\newtheorem{theorem}{Theorem}[section]
\newtheorem{lemma}[theorem]{Lemma}
\newtheorem{proposition}[theorem]{Proposition}

\newtheorem{corollary}[theorem]{Corollary}
\theoremstyle{definition}
\newtheorem{definition}[theorem]{Definition}

\newcommand{\R}{\mathbb{R}}
\newcommand{\C}{\mathbb{C}}

\newcommand{\Z}{\mathbb{Z}}
\newcommand{\N}{\mathbb{N}}
\newcommand{\one}{\mathds{1}}
\newcommand{\Li}{\operatorname{Li}}
\newcommand{\card}{\operatorname{card}}
\newcommand{\modu}{\operatorname{mod}}
\renewcommand{\Re}{\operatorname{Re}}

\newcommand{\partition}{\mathfrak{p}}
\newcommand{\Pri}{\mathbb{P}}
\DeclareMathOperator{\res}{res}
\newcommand{\ee}{\operatorname{e}}

\usepackage[ps,all,arc,rotate]{xy}
\usepackage{graphicx, float, epstopdf}
\usepackage{hyperref, color}
\usepackage{centernot}
\usepackage{fancyhdr}
\usepackage[utf8]{inputenc}
\usepackage{amsfonts,amssymb,amsmath,amsthm,mathrsfs}
\usepackage{graphics, setspace}
\usepackage{braket}
\usepackage{mathtools}
\usepackage{tikz}
\usepackage{esint}
\newcommand{\real}{\operatorname{Re}}
\newcommand{\imag}{\operatorname{Im}}
\newcommand{\e}{\operatorname{e}}

\newmuskip\pFqmuskip

\newcommand*\pFq[6][8]{%
\begingroup % only local assignments
\pFqmuskip=#1mu\relax
\mathchardef\normalcomma=\mathcode`,
\mathcode`\,=\string"8000
\begingroup\lccode`\~=`\,
\lowercase{\endgroup\let~}\pFqcomma
{}_{#2}F_{#3}{\left[\genfrac..{0pt}{}{#4}{#5};#6\right]}%
\endgroup
}
\newcommand{\pFqcomma}{{\normalcomma}\mskip\pFqmuskip}

\allowdisplaybreaks[4]         %allows breaking multiline environments in amsmath commands
 
\begin{document}

\title{Strange and pseudo-differentiable functions with applications to prime partitions}
%\author{Nicolas Robles}
%\author{Dirk Zeindler}

\author[A.~Dong]{Anji Dong}
\address{Department of Mathematics, University of Illinois, 1409 West Green Street, Urbana, IL 61801, USA}
\email{anjid2@illinois.edu}
\author[N.~Robles]{Nicolas Robles}
\address{RAND Corporation, Engineering and Applied Sciences, 1200 S Hayes St, Arlington, VA, 22202, USA}
\email{nrobles@rand.org}
\author[A.~Zaharescu]{Alexandru Zaharescu}
\address{Department of Mathematics, University of Illinois, 1409 West Green
Street, Urbana, IL 61801, USA; \textnormal{and} Institute of Mathematics of the Romanian Academy, P.O. BOX 1-764, Bucharest, Ro-70700, Romania}
\email{zaharesc@illinois.edu}
\author[D.~Zeindler]{Dirk Zeindler}
\address{Department of Mathematics and Statistics, Lancaster University, Fylde College, Bailrigg, Lancaster LA1 4YF, United Kingdom}
\email{d.zeindler@lancaster.ac.uk}

\begin{abstract}
Let $\mathfrak{p}_{\mathbb{P}_r}(n)$ denote the number of partitions of $n$ into $r$-full primes. We use the Hardy-Littlewood circle method to find the asymptotic of $\mathfrak{p}_{\mathbb{P}_r}(n)$ as $n \to \infty$. This extends previous results in the literature of partitions into primes. We also show an analogue result involving convolutions of von Mangoldt functions and the zeros of the Riemann zeta-function. To handle the resulting non-principal major arcs we introduce the definition of strange functions and pseudo-differentiability.
\end{abstract}

\subjclass[2020]{Primary: 11P55, 11P82, 26A24. Secondary: 11L07, 11L20, 11M26. \\ \indent \textit{Keywords and phrases}: weights associated to partitions, pseudo-differentiable functions, strange functions, inclusion-exclusion, Hardy-Littlewood circle method, exponential sums, von Mangoldt function, zeros of the zeta function.}

\maketitle

\textcolor{black}{
\tableofcontents}

%\newpage

\section{Introduction and results} \label{sec:introduction}
\subsection{Previous and current results}
A partition of a positive integer $n$ is a non-decreasing sequence of positive integers whose sum is equal to $n$. Let $\mathfrak{p}_{\N}(n)$ denote the restricted partition function that counts partitions of $n$ lying within $\N$. A classical result of Hardy and Ramanujan \cite{HardyRamanujan} states that 
\[
\mathfrak{p}_{\N}(n) \sim \frac{1}{4n\sqrt{3}}\exp\bigg(\pi\sqrt{\frac{2n}{3}}\bigg) \quad \textnormal{as} \quad n \to \infty.
\]
Let $\Pri$ denote the set of primes. Partition functions that count partitions of $n$ lying within $\Pri$ have also been studied by several authors such as Bateman and Erd\"{o}s \cite{BatemanErdos1, BatemanErdos2}, Browkin \cite{Browkin}, Kerawala \cite{Kerawala}, Roth and Szekeres \cite{RothSzekeres}, as well as Yang \cite{Yang}. In 2008 Vaughan \cite{VaughanPrimes} was able to simplify and improve most of the literature on prime partitions. Moreover, in 2021, Gafni \cite{gafniPrimePowers} proved an asymptotic formula for partitions with respect to \textcolor{black}{$\Pri^k = \{p^k\,;\, p \in \Pri\}$ for fixed $k \in \N$}.
Explicitly, Gafni established
\[
\mathfrak{p}_{\Pri^k}(n) \sim \mathfrak{c}_1 \exp \bigg(\mathfrak{c}_2 \frac{n^{\frac{1}{k+1}}}{(\log n)^{\frac{k}{k+1}}}(1+o(1)) \bigg)n^{-\frac{2k+1}{2k+2}}(\log n)^{-\frac{k}{2k+2}} \quad \textnormal{as} \quad n \to \infty,
\]
where $\mathfrak{c}_1$ and $\mathfrak{c}_2$ are positive constants that can be made explicit and that depend only on $k$. Vaughan's result is the special case $k=1$. For $r\in\N_{>0}$, let \textcolor{black}{$\Pri_r = \{p_1 p_2 \cdots p_r\,;\, p_i \in \Pri \text{ for all } 1\leq i \leq r\}$} denote the set of $r$-full primes. In this note we propose to find the asymptotic for $\mathfrak{p}_{\Pri_r}(n)$ as $n \to \infty$. Due to the close proximity of the characteristic function of the primes $\one_\Pri$ to the von Mangoldt function $\Lambda$, we also show how to adapt the asymptotic of $\mathfrak{p}_{\Pri_r}(n)$ to that of $\mathfrak{p}_{\Lambda^{*r}}(n)$ where $\Lambda^{*r} = \Lambda * \Lambda * \cdots * \Lambda$ exactly $r$ times. Our results are as follows.

\begin{theorem}
\label{thm:main_Pri_r}
We have
\begin{align}
	\log \partition_{\Pri_r}(n)
	= 
	2n^{\frac{1}{2}}\big(Q(n)\big)^{-1}\bigg(1+ O\bigg(\frac{1}{\log n}\bigg)\bigg),
\end{align}
where the term $Q$ is given by
\begin{align}
	Q(n)
	:=
	\left(\frac{\log n + \log \log n - \log 2 - \log P_r(\log\log n-\log2) -\log\zeta(2)}{2\zeta(2) P_r\left(\log \log n - \log 2 \right) }\right)^{\frac{1}{2}},
\end{align}	
and $P_r$ is a polynomial of degree $r-1$ with leading coefficient $r$.
\end{theorem}
The polynomial $P_r$ can be explicitly determined and its formula is given in \eqref{eq:P_r_formula} with the special cases $r=1,2,3,4$ stated in Corollary~\ref{cor:P_r_spcial_cases}.
\begin{theorem}
\label{thm:main_Lambda_r}
There exists a polynomial $\widetilde{P}_r$ of degree $r-1$ and leading coefficient $1$ such that 
\begin{align}
	\log \partition_{\Lambda^{*r}}(n)
	= 
	2n^{\frac{1}{2}}\left(\zeta(2)\widetilde{P}_r\left(\frac{\log n}{2}\right)\right)^{\frac{1}{2}}
	\left(1+O\left(\frac{(\log\log n)^{r-1}}{\log n}\right)\right).
\end{align}
\end{theorem}
Likewise, the polynomial $\widetilde{P}_r$ can also be explicitly determined and its formula is given in \eqref{eq:Phi_Lambda-precise} with the special cases $r=1,2,3$ stated in Corollary~\ref{cor:tilde_P_r_explicit}.

It is worth remarking some recent results in the literature of weighted partitions following the work of Vaughan in \cite{VaughaSquares}. The partitions for $\mathfrak{p}_{\mathcal{N}_k}(n)$ where 
\textcolor{black}{$\mathcal{N}_k = \{x^k\,;\, x\in \N\}$ for fixed $k \in \N_{\ge 2}$} were studied by Gafni in \cite{gafniPowers}. 
This was later generalized in \cite{amitaPartitions} by Berndt, Malik and Zaharescu to \textcolor{black}{$\mathcal{N}_{k,a,b} = \{x^k \equiv a \modu b\,;\, x\in \N\}$ for fixed $k \in \N_{\ge 2}$ and $a,b\in\N$ with $(a,b)=1$}. 
Let $\mathcal{N}_f = \{f(x)\,;\, x \in \N\}$ where $f$ is a polynomial of degree $d \ge 2$. 
The associated partitions were studied in \cite{dunnRobles} when $\gcd(\mathcal{N}_f)=1$. Chromatic partitions associated to the generalized divisor function were established in \cite{BRZZ23}. Furthermore, partitions \textit{signed} and \textit{weighted} by arithmetic functions such as the M\"{o}bius and Liouville functions were studied in \cite{taylorMoebius} and \cite{BRZ23}, respectively. The case $r=2$ of Theorem \ref{thm:main_Pri_r} had been studied in \cite{semiprimes}, albeit with weaker arithmetic technologies. Related recent work can also be found in \cite{Bridges}.

The technique to extract the asymptotic formula from the generating series is that of the Hardy-Littlewood circle method \cite{HLM}. As we will elaborate, the idea is to divide a circle of radius $\rho<1$ into three distinct types of arcs: the \textcolor{black}{principal} major arc, the non-principal major arcs, and the minor arcs. Each of these arcs requires its own set of techniques. The main term in the asymptotic for $\mathfrak{p}_{\Pri_r}(n)$ will come from the principal major arcs whereas the other two arcs will give rise to the error terms. 

The technology to deal with the minor arcs requires precise bounds on exponential sums twisted by suitable arithmetic functions such as 
\begin{align} \label{eq:generic_exp_sum}
\sum_{n \le x} f(n) \ee(\alpha n) \quad \textnormal{where} \quad \ee(x)=\exp(2\pi i x)
\end{align}
and $\alpha \in \R$. In our case, these functions will be $f=\one_{\Pri_r}$, the characteristic function of $r$-full primes, and $f=\Lambda^{*r}$. 
\begin{color}{black}
The bounds must hold in the `minor arc regime'. This regime consists of all $\alpha$ that cannot be well approximated by a rational number $\frac{a}{q}$ with $q$ small, 
see \eqref{eq:def_delta_q} -- \eqref{eq:def_majyor_M_and_minor_m} for the exact requirement.
The given bounds are indeed valid on the whole circle, but in some cases they are worse than trivial.
\end{color}

The handling of the minor arcs is for the most part arithmetic in nature.

On the other hand, the principal arcs are mostly handled analytically, usually by the method of contour integration or the method of moments. In addition, one also heavily employs the method of steepest descent. 

Interestingly, the non-principal major arcs are a mixture and necessitate arithmetico-analytic techniques such as exponential sums in arithmetic progressions, Siegel-Walfisz estimates, and, in our case, pseudo-differentiability and strange functions.

By leveraging our previous results on exponential sums \cite{DoRoZaZe24} twisted by $\one_{\Pri_r}$ and $\Lambda^{*r}$ we are able to bound the minor arcs associated to $\mathfrak{p}_{\Pri_r}$ and $\mathfrak{p}_{\Lambda^{*r}}$. Furthermore, by using a Hankel contour over the zero-free region of the Riemann zeta-function we will be able to extract the main term of the asymptotic for $\mathfrak{p}_{\Pri_r}$. In the case of $\mathfrak{p}_{\Lambda^{*r}}$ we will employ a contour that explicitly shows the connection to the zeros of the Riemann zeta-function. By far, the longest and most difficult part of the paper will deal with the non-principal major arcs. We elaborate on this point in the next section.

\subsection{Technology on pseudo-differentiable functions}

To show that the non-principal major arcs do not contribute to the main term, 
we require a result similar to the Siegel-Walfisz theorem for primes in an arithmetic progression.     
Explicitly, for $q,\ell\in\N$ with $(q,\ell)=1$ we have to study the sums 
\begin{align}
A_{r}(t;q,\ell)
&:=
\mathop{\sum\sum}_{\substack{p_1,\ldots,p_r\\p_1 \cdots p_r\leq t\\  p_1\ldots p_r\equiv\ell \bmod q }} 1
\quad \text{ and } \quad 
\widetilde{A}_{r}(t;q,\ell)
:=
\mathop{\sum\sum}_{\substack{n_1,\ldots,n_r\in\N\\n_1 \cdots n_r\leq t\\  n_1\ldots n_r\equiv\ell \bmod q }} \prod_{i=1}^{r}\Lambda(n_i),
%\label{eq:def_A_r}	
\end{align}
and our main task is to show that we can write 
\begin{align}
A_{r}(t;q,\ell) = D_r(t) + O\left(t(\log t)^{-C}\right)
\ \text{ and }\ 
\widetilde{A}_{r}(t;q,\ell) = \widetilde{D}_r(t) + O\left(t(\log t)^{-C}\right),
%\label{eq:Ar_what_we would_like_to_have}
\end{align}
where $C\geq 1$ is arbitrarily large, $D_r(t)$ and $\widetilde{D}_r(t)$ are differentiable with respect to $t$ and $D_r'(t)$ and $\widetilde{D}'_r(t)$ can be bounded suitably. Observe that 
\begin{align}
A_{1}(t;q,\ell)
=
\pi(t;q,\ell) 
= 
\card\,\{p\leq t \,;\, p \equiv \ell \bmod q \}
\ \text{ and }  \
\widetilde{A}_{1}(t;q,\ell)
=
\psi(t;q,\ell),
\end{align}
and thus the case $r=1$ corresponds to the classical the Siegel-Walfisz theorem.
The proof of this theorem can be found for example in \cite[$\mathsection$20-22]{Da74} and requires the analytic continuation of the Dirichlet $L$-function. 
To use the same argument for $r>1$ as for the Siegel-Walfisz theorem,  an analytic continuation of the function 
\[
\mathfrak{F}_r(s) := \sum_{n=1}^\infty \chi(n)\frac{\one_{\Pri}^{*r}(n)}{n^s}
\] 
is needed for $\real(s) > s_0$ for some $s_0$. 
However, as far as the authors are aware, such an analytical continuation is not available and would also require a substantial amount of work to be established. 
Thus an alternative argument is needed. The most natural one is to combine induction over $r$, the hyperbola method and Abel summation.
In particular, we obtain 
\begin{align*}
A_{2}(t;q,\ell)
=&
\frac{\big(\Li(\sqrt{t})\big)^2}{\varphi(q)}
-
2\sum_{p|q} \frac{\Li(t/p) }{\varphi(q)}
+
\frac{2}{\varphi(q)}\int_{2}^{\sqrt{t}}  \frac{t \pi(u)}{u^2\log(t/u)}\,du
+
O\left(\frac{t}{(\log t)^{C}}\right).
\end{align*}
Unfortunately, two problems occur in this induction argument.
First, explicit expressions for $r\geq 3$  are not easy to achieve and are also quite complicated.
Second, for $r\geq2$ the function $A_{r}(t;q,\ell)$ is only once continuously differentiable, but we require for the induction argument higher derivatives.
Our solution to these problems is to introduce pseudo-differentiable functions, see Section~\ref{sec:strange-functions}.
This simplifies the resulting expressions substantially and allows to carry over the desirable strong error term arising from the Siegel-Walfisz theorem.

\subsection{Organization of the paper}
The paper is organized as follows. In Section \ref{sec_st_up_arcs} we introduce the generating functions of our partitions of interest and the decomposition of the arcs. In Section \ref{sec:useful-lemmas} we present some preliminary lemmas that will be needed throughout the paper. The results on exponential sums will be presented in Section \ref{sec:exponential-sums-and-the-minor-arcs} and we will show they can be used to satisfactorily bound the minor arcs. Pseudo-differentiability will be introduced in Section \ref{sec:pseudo-differentiable-functions-and-non-principal-arcs} along with the definition of strange functions. These tools will enable us to eventually bound the non-principal major arcs. The main terms and the method of steepest descent will be discussed in Section \ref{sec:main_terms_and_principal_arcs}. Lastly, we conclude with some ideas for future work in Section \ref{sec:conclusion}. 

\subsection{Notation}
Throughout the paper, the expressions $f(X)=O(g(X))$, $f(X) \ll g(X)$, and $g(X) \gg f(X)$ are equivalent to the statement that $|f(X)| \le (\ge) C|g(X)|$ for all sufficiently large $X$, where $C>0$ is an absolute constant. A subscript of the form $\ll_{\alpha}$ means the implied constant may depend on the parameter $\alpha$. The notation $f = o(g)$ as $x\to a$  means that $\lim_{x\to a} f(x)/g(x) = 0$ and $f\sim g$ as $x\to a$ denotes $\lim_{x\to a} f(x)/g(x)= 1$. %Dyadic sums are represented by $\sum_{n \sim N} f(n) = \sum_{N < n \leqslant 2N} f(n)$. The divisor function is denoted by $\tau(n)$. The $k$-fold divisor function $\tau_k(n)$ is defined by the coefficients of the Dirichlet series $\zeta^k(s) = \sum_{n=1}^{\infty}\frac{\tau_k(n)}{n^s}$ for $\real(s)>1$. 
The notation $f^{*r}$ indicates that the arithmetic function $f$ is convolved with itself $r$ times. The Meissel-Mertens constant is denoted by $M \approx 0.26149721$. The Greek character $\rho$ is reserved for the radius of the circle method, whereas the non-trivial zeros of zeta will be denoted by $\varpi$. The Latin character $p$ will always denote a prime, whereas $\mathfrak{p}_A(n)$ denotes the number of partitions of $n$ with respect to some weight $A$. The logarithmic integral is $\Li(x) = \int_2^x \frac{dt}{\log t}$. Finally, throughout the paper, we shall use the convention that $\varepsilon$ denotes an arbitrarily small positive quantity that may not be the same at each occurrence.

\section{Generating functions and decomposition of arcs}
\label{sec_st_up_arcs}

In this section, we specify the generating functions we are working with as well as the associated arcs that we shall use throughout the paper.

\subsection{Generating functions}
\label{sec:Gen_function_used}

For partitions into $r$-full primes and partitions weighted by $\Lambda^{*r}$ we will work with the following generating functions
\begin{align}
\Psi_{\Pri_r}(z)
&:=
\sum_{n=0}^{\infty}
\partition_{\Pri_r}(n) z^n
:=
\prod_{p_1,\ldots,p_r\in\Pri}^{} \frac{1}{1-z^{p_1\cdot\ldots\cdot p_r}},
\label{eq:def_gen_Pr}
\\
\Psi_{\Lambda^{*r}}(z)
&:=
\sum_{n=0}^{\infty}
\partition_{\Lambda^{*r}}(n) z^n
:=
\prod_{n\in\N}^{} \left(\frac{1}{1-z^{n}}\right)^{\Lambda^{*r}(n)},
\label{eq:def_gen_Lambda_r}
\end{align}
respectively. We now can rewrite those functions as
\begin{align*}
\Psi_{\Pri_r}(z)
=
\exp(\Phi_{\Pri_r} (z))
\ \text{ and } \
\Psi_{\Lambda^{*r}}(z)
=
\exp(\Phi_{\Lambda^{*r}} (z))
%\label{eq:def_generating_fkt_partition}
\end{align*}
with $\Phi_{\Pri_r}$ and $\Phi_{\Lambda^{*r}}$ given by
\begin{align}
\Phi_{\Pri_r} (z)
&:= 
\sum_{j=1}^\infty \frac{1}{j} \mathop{\sum\sum}_{\substack{p_1,\ldots,p_r\in\Pri^r}} z^{p_1\cdot\ldots\cdot p_rj}
\quad \text{and} \quad
\label{eq:def_Phi_prime}\\
\Phi_{\Lambda^{*r}} (z)
&:= 
\sum_{j=1}^\infty \frac{1}{j} \sum_{n=1}^\infty\Lambda^{*r}(n)z^{jn}
=
\sum_{j=1}^\infty \frac{1}{j} \mathop{\sum\sum}_{\substack{n_1,\ldots,n_r\in\N}} \bigg(\prod_{j=1}^{r}\Lambda(n_j)\bigg)z^{n_1\cdot\ldots\cdot n_rj}.
\label{eq:def_Phi_Lambda}
\end{align}

\subsection{Set up of the arcs}
Let $\ee(x)=\exp(2\pi i x)$. Applying Cauchy's theorem to the generating function in
\eqref{eq:def_gen_Pr}, 
we can write $\partition_{\Pri_r}(n)$ as
\begin{align} 
\partition_{\Pri_r}(n)
&= 
\rho^{-n} \int_{-\frac{1}{2}}^{\frac{1}{2}} \Psi_{\Pri_r}(\rho \ee(\alpha)) \ee(-n \alpha) d\alpha %\nonumber\\
= 
\rho^{-n} \int_{-\frac{1}{2}}^{\frac{1}{2}} \exp\left(\Phi_{\Pri_r}(\rho \ee(\alpha))\right) \ee(-n \alpha) d\alpha,
\label{eq:integralarcs}
\end{align}
where the radius $0<\rho<1$ is at our disposal (similarly for $\partition_{\Lambda^{*r}}(n)$).
We evaluate this integral with the Hardy-Littlewood circle method and we will choose $\rho$ in Section~\ref{sec:the-asymptotic-orders-of-magnitude} as the solution of the corresponding saddle point equation.
In particular, we have $\rho\to 1$ as $n\to\infty$.

In order to prove our main theorems, we have to study the behavior of $\Psi_{\Pri_r} (z)$ 
and $\Psi_{\Lambda^{*r}} (z)$ near the boundary of the unit disc.
We can see in Figures \ref{fig:contourplots1}, \ref{fig:contourplots2} and \ref{fig:contourplots3} the largest values of $\Psi_{\Pri_r} (z)$ are near the point $1$, 
but also that $\Psi_{\Pri_r} (z)$ is large near $\ee(a/q)$ with $a\in\Z$, $q\in\N$ and $q$ small.

\begin{figure}[h!] %scale=0.374 scale=0.36 scale=0.374 scale=0.36
\includegraphics[scale=0.308]{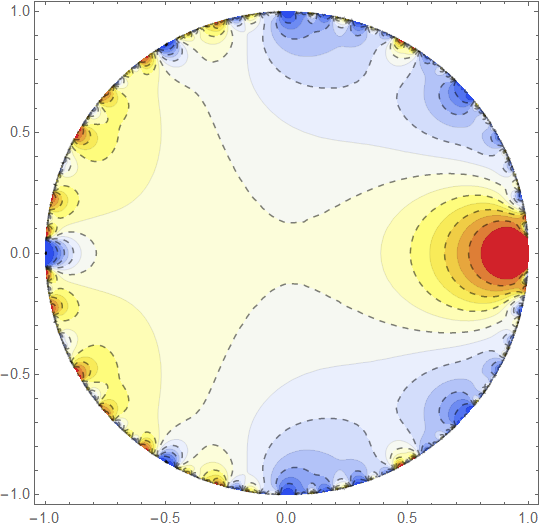} 
\includegraphics[scale=0.308]{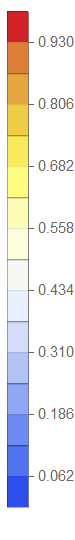} 
\includegraphics[scale=0.308]{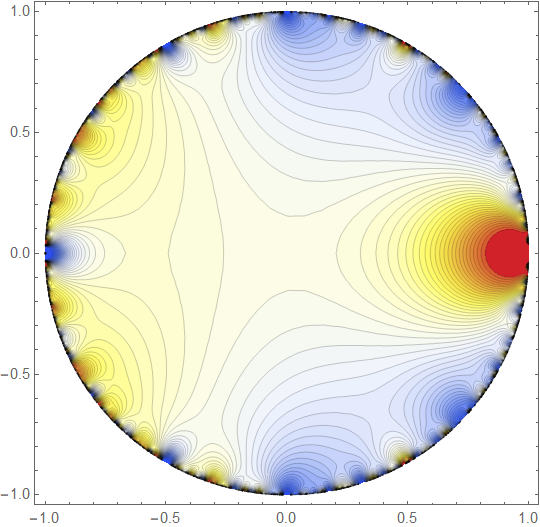} 
\includegraphics[scale=0.308]{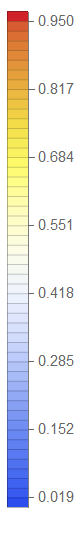} 
\caption{Domain plot of $\Phi_{\Pri,M}(z)= \prod_{p \in \mathbb{P} \cap [1,M]} (\frac{1}{1-z^p})$  with $M=50$ (left) and $M=100$ (right) and absolute value and argument hue.}
\label{fig:contourplots1}
\includegraphics[scale=0.308]{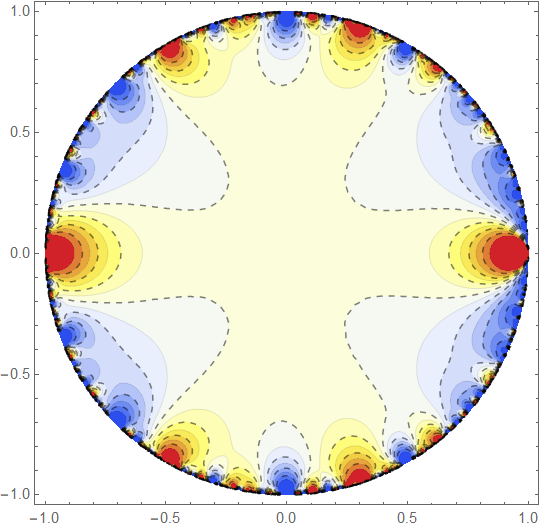} 
\includegraphics[scale=0.308]{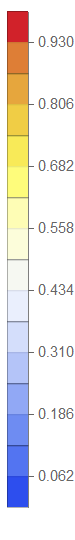} 
\includegraphics[scale=0.308]{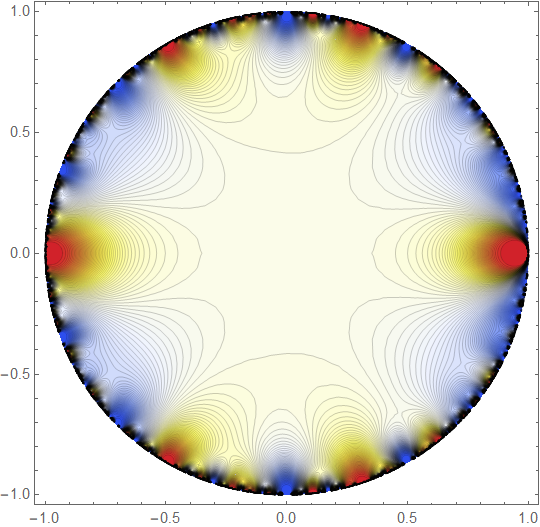} 
\includegraphics[scale=0.308]{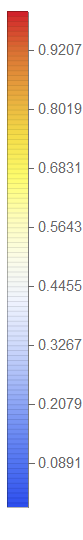} 
\caption{Domain plot of $\Phi_{\Pri_2,M}(z)= \prod_{p_1, p_2 \in \mathbb{P} \cap [1,M]} (\frac{1}{1-z^{p_1p_2}})$  with $M=50$ (left) and $M=100$ (right) and absolute value and argument hue.}
\label{fig:contourplots2}
\includegraphics[scale=0.308]{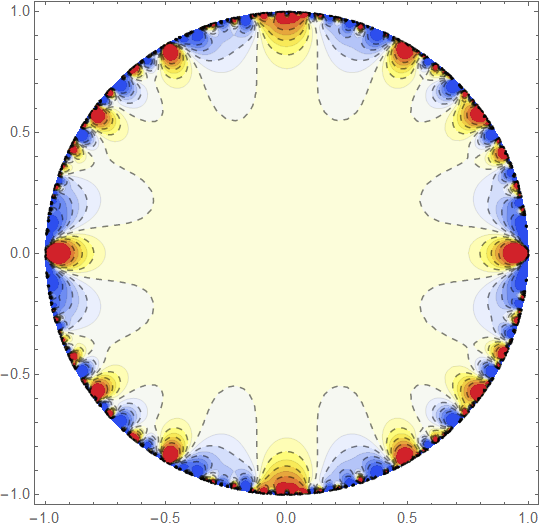} 
\includegraphics[scale=0.308]{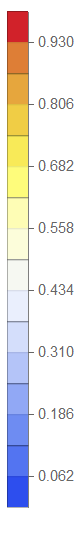} 
\includegraphics[scale=0.308]{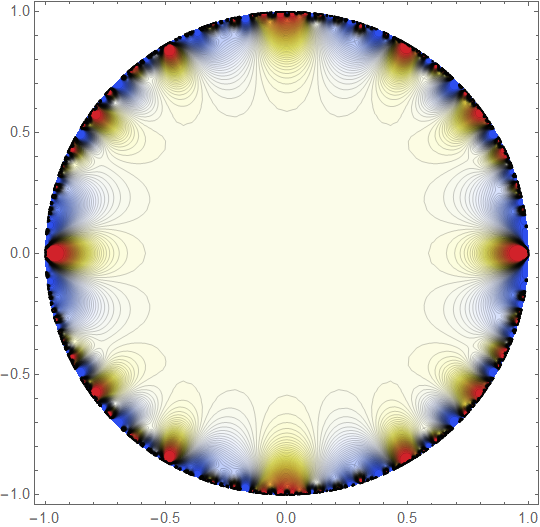} 
\includegraphics[scale=0.308]{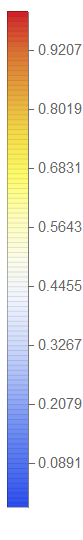} 
\caption{Domain plot of $\Phi_{\Pri_3,M}(z)= \prod_{p_1, p_2, p_3 \in \mathbb{P} \cap [1,M]} (\frac{1}{1-z^{p_1p_2p_3}})$  with $M=50$ (left) and $M=100$ (right) and absolute value and argument hue.}
\label{fig:contourplots3}
\end{figure}

Thus we have to carefully  split the integral in \eqref{eq:integralarcs}.
First, we choose an $A>8r$.
Then, we define the major arcs $\mathfrak{M}$ and minor arcs $\mathfrak{m}$ as follows. 
For $q\in\N$, we set
\begin{align} \label{eq:def_delta_q}
\delta_q := q^{-1}X^{-1}(\log X)^A \quad \textnormal{and} \quad Q := (\log X)^A.
\end{align}

Moreover, for $1 \leq a \le q \le Q$ with $(a,q)=1$ we define $\mathfrak{M}(q,a) $ to be the open interval
\begin{align}
\mathfrak{M}(q,a) := \bigg(\frac{a}{q} - \delta_q, \frac{a}{q} + \delta_q\bigg).
\label{eq:def_M(q,a)}
\end{align}
The major arcs $\mathfrak{M}$ and the minor arcs $\mathfrak{m}$  are defined by 
\begin{align}
\mathfrak{M} 
:= 
\bigcup_{1 \le a \le q \le Q}\bigcup_{\substack{(a,q)=1}} \mathfrak{M}(q,a) 
\quad \textnormal{and} \quad 
\mathfrak{m} 
= 
[-\tfrac{1}{2},\tfrac{1}{2}) \backslash \mathfrak{M}.
\label{eq:def_majyor_M_and_minor_m}
\end{align}
Next, in accordance with general strategies, 
we divide the integral in \eqref{eq:integralarcs} into three pieces:
\begin{itemize}
\item the principal major arc $\mathfrak{M}(1,0)$,
\item the non-principal major arcs $\mathfrak{M}(q,a)$ with $q>1$ and
\item and the minor arcs $\mathfrak{m}$.
\end{itemize}

The main term will be entirely dictated by the principal arc $\mathfrak{M}(1,0)$. 
The asymptotic behaviour of $\Phi_{\Pri_r} (z)$ and $\Phi_{\Lambda^{*r}}(z)$ on
the principal arc $\mathfrak{M}(1,0)$ is determined in Sections~\ref{sec:principalarcs}.
Furthermore, the bound for the non-principal major arcs  $\mathfrak{M}(q,a)$ with $q>1$ is deduced in Section~\ref{sec:pseudo-differentiable-functions-and-non-principal-arcs} using pseudo-differentiable functions, see Lemma~\ref{lem:culminationnonprincipal}. 
Finally, the bound for the minor arcs $\mathfrak{m}$ will be deduced in Section~\ref{sec:exponential-sums-and-the-minor-arcs}
using bounds for exponential sums from \cite{DoRoZaZe24}.
We have chosen the exponent of $\log X$ to be greater than $A>8r$ in order to yield a satisfactory bound on the minor arcs, see Lemma~\ref{lem:minorarclemmabound}.
We shall show the details of the derivation in Section~\ref{sec:exponential-sums-and-the-minor-arcs}.

\section{Useful lemmas}
\label{sec:useful-lemmas}

We state in this section some auxiliary lemmas we will need later.
\begin{lemma}
\label{lem:dirk'slemma2}
Let $a\geq0$, $b\in \R$ and $\lambda\geq 0$ be given. 
Further let $\gamma=\gamma_1+i\gamma_2\in\C$ with $\gamma_1>0$ and $\gamma_2 \ll \gamma_1 (\log(1/\gamma_1))^A$ for some $A>0$.
We then have as $\gamma_1\to 0$
\begin{align*}
\int_{2}^{\infty}  \frac{(\log\log t)^a}{(\log t)^b} t^\lambda\exp(-\gamma t) dt  
=
\frac{(\log\log(1/\gamma_1))^a}{(\log(1/\gamma_1))^b}\frac{\Gamma(\lambda+1)}{\gamma^{\lambda+1}}+ O\bigg(\frac{\log\log(1/\gamma_1)^{a+1}}{\gamma_1^{\lambda +1}(\log(1/\gamma_1))^{b+1}}\bigg).    
\end{align*}
\end{lemma}

\begin{proof}
We split the interval $[2,\infty]$ into intervals $[2,d], [d,u]$ and $[u,\infty]$ with
\begin{color}{black}
\begin{align*}
d=\frac{1}{\gamma_1(\log(1/\gamma_1))^{A+1+|b|}} \quad \text{and} \quad u=\frac{(\log (1/\gamma_1))^{2}}{\gamma_1} .
\end{align*}
\end{color}
\begin{color}{black}
We first look at the integral over $[2,d]$. 
We bound each term in the integral trivially and get
\begin{align*}
	\left|\int_{2}^{d}  \frac{(\log\log t)^a}{(\log t)^b} t^{\lambda}\exp(-\gamma t)  dt\right|
	&\leq 
	(\log\log d)^{a}(\log d)^{|b|} d^\lambda \int_{2}^{d} 1\,dt 
	\ll
	(\log\log d)^{a}(\log d)^{|b|} d^{\lambda+1}   \\
	&\ll 
	\frac{(\log\log(1/\gamma_1))^a (\log(1/\gamma_1))^{|b|}}{\gamma_1^{\lambda+1}(\log(1/\gamma_1))^{(A+1+|b|)(\lambda+1)}}
	\ll
	\frac{(\log\log(1/\gamma_1))^a}{\gamma^{\lambda+1}(\log(1/\gamma_1))^{\lambda+1}}.
\end{align*}
\end{color}
\begin{color}{black}
Next we look at the integral over $[u,\infty]$. 
Observe that 
$
(\log\log t)^a(\log t)^{|b|} t^\lambda \exp(-\frac{1}{2}\gamma_1t)
$
is monotonically decaying in the interval $[u,\infty]$ for $\gamma_1$ small enough. 
Thus 
\begin{align*}
	\left|\int_{u}^{\infty} \frac{(\log\log t)^a}{(\log t)^b} t^\lambda \exp(-\gamma t) dt\right|
	&\leq 
	\exp\left(-\frac{1}{2}\gamma_1u\right)
	(\log\log u)^a (\log u)^{|b|} u^\lambda
	\int_{u}^{\infty} \exp\left(-\frac{1}{2}\gamma_1t\right)dt \\
	&\ll
	(\log\log u)^a (\log u)^{|b|} u^\lambda \frac{\exp(-\gamma_1 u)}{\gamma_1}  \\
	&=
	\frac{	(\log\log (1/\gamma_1))^a (\log (1/\gamma_1)	)^{|b|}}{\gamma_1^{\lambda+1}(\log(1/\gamma_1))^{2\lambda}}
	 \exp(-\gamma_1 u)\\
	%\nonumber \\
	%&=
	%\frac{C\log\log u}{\log u}\frac{4}{\gamma^2_1} \exp\left(-\frac{\gamma_1}{2} u\right) 
	%&
	&\ll 
	\frac{(\log\log (1/\gamma_1))^a (\log (1/\gamma_1)	)^{|b|}}{\gamma_1^{\lambda+1}(\log(1/\gamma_1))^{2\lambda}} \frac{1}{\gamma_1^C},
\end{align*}
where $C>0$ can be chosen arbitrarily.
We have used on the last line that $\gamma_1 u = \log^2(1/\gamma_1)$
and thus $\exp(-\gamma_1 t) \leq \exp(- C \log(1/\gamma_1))$ for any $C>0$ for $\gamma_1$ small enough.
\end{color}

The remaining task is to look at the integral over $[d,u]$.
Notice that $t\in[d,u]$ implies $\log(t\gamma_1)= O(\log\log(1/\gamma_1))$ and so
\begin{align*}
\log t
&=
\log(1/\gamma_1)\left(1+\frac{\log (\gamma_1  t)}{\log(1/\gamma_1)}\right)
=
\log(1/\gamma_1)\left(1+O\left(\frac{\log\log(1/\gamma_1)}{\log(1/\gamma_1)}\right)\right),\\
\log\log t
&=
\log\big(\log(1/\gamma_1)+\log(\gamma_1 t)\big)
=
\log\log(1/\gamma_1)+\log\left(1+\frac{\log (\gamma_1  t)}{\log(1/\gamma_1)}\right)\nonumber\\
&=
\log\log(1/\gamma_1)+O\left(\frac{\log\log(1/\gamma_1)}{\log(1/\gamma_1)}\right).
\end{align*}
We therefore have for $t\in[d,u]$ that
\begin{align}
\frac{(\log\log t)^a}{(\log t)^b} 
&= 
\frac{(\log\log(1/\gamma_1))^a}{(\log(1/\gamma_1))^b}
+
O\left(\frac{\log\log(1/\gamma_1)^{a+1}}{(\log(1/\gamma_1))^{b+1}}\right).
\label{eq:loglog(tgamma)/log(tgamma)}
\end{align}
Furthermore, similar estimates as the ones above as well as the variable substitution $z=\gamma t$ yield
\begin{align}
\int_{d}^{u} t^\lambda e^{-\gamma t} dt
&=
\int_{0}^{\infty} t^\lambda e^{-\gamma t} dt
+
O\left(d^{\lambda+1} + u^\lambda \frac{\exp(-\gamma_1 u)}{\gamma_1}\right)\nonumber\\
&=
\frac{1}{\gamma^{\lambda+1}}
\varointclockwise_{\mathcal{L}} z^\lambda e^{-z} dz
+
O\left(\frac{1}{\gamma^{\lambda+1}(\log(1/\gamma_1))^{\lambda+1}}\right),
\label{eq:it:is_Gamma}
\end{align}
where $\mathcal{L}$ is the curve $\mathcal{L}(x) = x e^{i\beta}$ with $x\in[0,\infty]$ and $\beta$ is the argument of $\gamma$.
Since $|\beta|<\frac{\pi}{2}$, the last integral exists and is equal to $\Gamma(\lambda+1)$.
Inserting \eqref{eq:loglog(tgamma)/log(tgamma)} and \eqref{eq:it:is_Gamma} into the integral over $[d,u]$ gives
\begin{align*}
\int_{d}^{u} \frac{(\log\log t)^a}{(\log t)^b} t^\lambda e^{-\gamma t} dt 
&=
\frac{(\log\log(1/\gamma_1))^a}{(\log(1/\gamma_1))^b}\int_{d}^{u} t^\lambda e^{-\gamma t} dt  + O\left(\frac{\log\log(1/\gamma_1)^{a+1}}{(\log(1/\gamma_1))^{b+1}}\int_{0}^{\infty} t^\lambda e^{-\gamma_1 t} dt\right)\\
&=
\frac{(\log\log(1/\gamma_1))^a}{(\log(1/\gamma_1))^b}\frac{\Gamma(\lambda+1)}{\gamma^{\lambda+1}}+ O\bigg(\frac{\log\log(1/\gamma_1)^{a+1}}{\gamma_1^{\lambda +1}(\log(1/\gamma_1))^{b+1}}\bigg).         
\end{align*}
This completes the proof.
\end{proof}

\begin{lemma}
\label{lem:int_loglog_a_log_b}
Let $a\in\N_0$ and $b\in\Z$ be given.
We have for $d\in\N_0$  as $t\to \infty$ 	
\begin{align}
\int_{2}^{\sqrt{t}}
\frac{(\log u)^d(\log\log(t/u))^a}{u(\log(t/u))^b} \,du
&=
c_{a,b,d}\frac{(\log\log t)^{a} + \widetilde{P}(\log\log t)}{(d+1) 2^{d+1}(\log t)^{b-d-1}} 
+O\left(\frac{(\log\log t)^{a}}{(\log t)^{b-d}}\right)
\label{eq:int_loglog_a_log_b2}
\end{align}
with $\widetilde{P}$ a polynomial of degree $a-1$ and $c_{a,b,d}={}_2F_1(b, 1 + d; 2 + d, \frac{1}{2})$,
where  ${}_2F_1$ is the ordinary hypergeometric function.
Further, 
\begin{align}
\int_{2}^{\sqrt{t}}
\frac{(\log\log(t/u))^a}{(u \log u) (\log(t/u))^b} \,du
&=
\frac{(\log\log t)^{a+1} +P(\log\log t)}{(\log t)^b} 
+O\left(\frac{(\log\log t)^{a}}{(\log t)^{b+1}}\right),
\label{eq:int_loglog_a_log_b}
\end{align}
where $P$ is a polynomial of degree at most $a$.
Finally, if $d\in\N$ and $d\geq 2$ then 
\begin{align}
\int_{2}^{\sqrt{t}}
\frac{(\log\log(t/u))^a}{u(\log u)^{d} (\log(t/u))^b} \,du
&=
\frac{(\log\log t)^{a}}{(d-1)(\log 2)^{d-1}} +O\left(\frac{(\log\log t)^{a}}{\log t}\right) .
\label{eq:int_loglog_a_log_b_c<-1}
\end{align}
\end{lemma}	
Recall, the hypergeometric function is defined for $|z| < 1$ by the power series
\begin{align}
{}_{2}F_{1}(a_1,a_2;b_1;z)
=
\sum _{n=0}^{\infty }
\frac {(a_1)_{n}(a_2)_{n}}{(b_1)_{n}}\frac {z^{n}}{n!},
\end{align}
where $(q)_n = q(q+1)\cdots (q+n-1)$ is the rising Pochhammer symbol.
We require for the proof of Theorem~\ref{thm:Ar_is_strange_lambda} the two cases
\begin{align}
{}_2F_1(0, 1 + d; 2 + d;\tfrac{1}{2})= 1 
\ \text{ and } \
{}_2F_1(b, 1 ; 2 ; \tfrac{1}{2})= \frac{2}{1-b}(1-2^{b-1}).
\label{eq:special_cases_2F1}
\end{align}
The proof of Lemma~\ref{lem:int_loglog_a_log_b}
below gives more than \eqref{eq:int_loglog_a_log_b2}, \eqref{eq:int_loglog_a_log_b} and \eqref{eq:int_loglog_a_log_b_c<-1}.
Indeed, we can give an explicit expressions for the $O(\cdot)$ term as a convergent power series involving 
$\log\log t$ and $\log t$ (for $t\geq e^2$). 
However, the expression for the $O(\cdot)$ term is a little bit involved and \eqref{eq:int_loglog_a_log_b2} and \eqref{eq:int_loglog_a_log_b} are sufficient for our purposes.
\begin{proof}[Proof of Lemma~\textnormal{\ref{lem:int_loglog_a_log_b}}]
We begin with some preparations.
Observe that we have for all $m\in\Z\setminus\{-1\}$ that
\begin{align}
\int_{2}^{\sqrt{t}}
\frac{(\log u)^m}{u} \,du
&=
\frac{(\log t)^{m+1}}{(m+1)2^{m+1}}- \frac{(\log 2)^{m+1}}{m+1}
\ \text{ and }
\label{eq:int_ulog^c(u)}\\
\int_{2}^{\sqrt{t}}
\frac{du}{u\log u}
&=
\log\log\sqrt{t} - \log\log 2
=
\log\log t - \log\log 4.
\label{eq:int_ulog(u)}
\end{align}
Further, let $f$ be a meromorphic function such that
$f(x)= \sum_{j=-\ell}^\infty c_jx^j$ is convergent for $0<|x|<1$ for some $\ell\geq 0$ and $c_j\in\R$. 
Since $0<|\frac{\log u}{\log t}|\leq \frac{1}{2}$, we can use dominated convergence and get with \eqref{eq:int_ulog^c(u)} and \eqref{eq:int_ulog(u)} that for all $c\in\Z$
\begin{align}
J_f(t)
:=&
\int_{2}^{\sqrt{t}}
\frac{(\log u)^c}{u}  f\left(\frac{\log u}{\log t}\right)\,du
=
\sum_{j=-\ell}^\infty c_j
\int_{2}^{\sqrt{t}} 
\frac{1}{u} \frac{(\log u)^{c+j}}{(\log t)^j} \,du\nonumber\\
=\,&
c_{-c-1}(\log t)^{c+1}\left(\log\log t-\log\log 4\right) \nonumber\\
&+
\sum_{j\neq -c-1}^\infty c_j\left( \frac{(\log t)^{c+1}}{2^{c+1}(c+j+1)2^{j}} - \frac{(\log 2)^{j+c+1}}{(j+c+1)(\log t)^{j}}\right).
\label{eq:int_I_f2}
\end{align}
Further, observe
\begin{align}
\frac{(\log\log(t/u))^a}{ (\log(t/u))^b}
=
\frac{(\log\big(\log t-\log u\big))^a}{(\log t-\log u)^b}
&=
\frac{\big(\log\log t-\log(1-\frac{\log u}{\log t})\big)^a}{(\log t)^b(1-\frac{\log u}{\log t})^b}\nonumber\\
&=
\frac{1}{(\log t)^b} \sum_{k=0}^a \binom{a}{k} \big(\log\log t\big)^{a-k} (-1)^{k} f_{k}\left(\frac{\log u}{\log t}\right)
\label{eq:expansion_loglog(t/u)}	
\end{align}
with $f_k(x)
=
\frac{\left(\log(1-x)\right)^k}{(1-x)^b}$.
Each $f_k$ is an analytic function for $|x|<1$ and thus
\begin{align}
\int_{2}^{\sqrt{t}}
\frac{(\log u)^c(\log\log(t/u))^a}{u (\log(t/u))^b} \,du
&=
\frac{1}{(\log t)^b} \sum_{k=0}^a \binom{a}{k} (\log\log t)^{a-k} (-1)^{k} J_{f_{k}}(t).
\label{eq:int_sum_J_fk}
\end{align}
We are now able to prove the lemma.
We first prove \eqref{eq:int_loglog_a_log_b2} and we thus choose $c=d\geq 0$ in \eqref{eq:int_I_f2}.
Since each $f_k$ is holomorphic and $c\geq 0$, we get that $c_{-c-1} =0$ for each $k$. 
Thus the order of magnitude of each $J_{f_{k}}(t)$ in \eqref{eq:int_sum_J_fk} is $(\log t)^{c+1}$.
Also, inserting \eqref{eq:int_I_f2} into \eqref{eq:int_sum_J_fk} shows that the integral in \eqref{eq:int_loglog_a_log_b2} has the required form and that 
the leading term comes from $J_{f_{0}}(t)$.
It only remains to determine the leading coefficient. 
Using the Taylor series of $f_0(x) = (1-x)^{-b}$ gives
\begin{align}
\frac{1}{2^{c+1}}
\sum_{j=0}^\infty \binom{-b}{j}\frac{(-1)^j}{(c+j+1)2^{j}}
=
\frac{1}{2^{d+1}(d+1)}
\sum_{j=0}^\infty \frac{(b)_j}{j!} \frac{(d+1)_j}{(d+2)_j}  \frac{1}{2^{j}}.
\end{align}
This completes the proof of \eqref{eq:int_loglog_a_log_b2}.
Next we prove \eqref{eq:int_loglog_a_log_b}.
Here choose $c=-1$ and observe that $f_0(0)=1$ and  $f_k(0)=0$ for $1\leq k\leq b$.
This implies 
\begin{align*}
J_{f_0}(t) = \log\log t + O(1) \ \text{ and } \ I_{f_k}(t) \ll 1 \ \text{ for }1\leq k\leq a.
\end{align*}
Inserting this into \eqref{eq:int_sum_J_fk} completes the prove of \eqref{eq:int_loglog_a_log_b}.
Finally, to prove \eqref{eq:int_loglog_a_log_b_c<-1}, we choose $c=-d$ and thus $c\leq -2$.
Therefore
\begin{align}
J_f(t)
=\,&
%c_{-c-1}(\log t)^{c+1}\left(\log\log(t)-\log\log 4\right) \nonumber\\
%&+
-\sum_{j\neq -c-1}^\infty c_j\left( \frac{(\log 2)^{j+c+1}}{(j+c+1)(\log t)^{j}}\right)
+O\left(\frac{\log\log t}{\log t}\right)
%	\label{eq:int_I_f2}
\end{align}
and 
\begin{align*}
J_{f_0}(t) 
&= 
- \frac{(\log 2)^{c+1}}{(c+1)} +O\left(\frac{\log\log t}{\log t}\right) 
=
\frac{1}{(d-1)(\log 2)^{d-1}} +O\left(\frac{\log\log t}{\log t}\right) 
\ \text{ and } \\
J_{f_k}(t) &\ll \frac{\log\log t}{\log t} \ \text{ for }1\leq k\leq a.
\end{align*}
This completes the proof.
\end{proof}

\section{Exponential sums and the minor arcs}
\label{sec:exponential-sums-and-the-minor-arcs}

In this section, we establish an upper bound for $\Phi_{\Pri_r}$ and $\Phi_{\Lambda^{*r}}$ on the minor arcs $\mathfrak{m}$, 
see Lemma~\ref{lem:cul_minor_arcs}. 
%We will use this bound in Section~\ref{label} to show that the contribution of the integral over the minor arcs $\mathfrak{m}$ is of lower order.
For this, we require for each $r\in\N$ an upper bound for the exponential sums
\begin{align}
S_r(\alpha,X)
&:=
\sum_{\substack{p_1,\ldots,p_r\in\Pri\\p_1\cdots p_r\leq X}}\ee(\alpha p_1\cdots p_r)
\ \text{ and } \
\widetilde{S}_r(\alpha,X)
:=
\sum_{\substack{n_1,\ldots,n_r\in\N\\n_1\cdots n_r\leq X}}\ee(\alpha n_1\cdots n_r)\prod_{i=1}^r\Lambda(n_i).
\label{eq:def_S_r}
\end{align}
In the sum $S_r(\alpha,X)$,  each $p_j$ runs over all primes. 
Thus if a natural number $n$ has the form $n=p^r$ then $n$ occurs once in this sum.
On the other hand, if $n = p_1\cdots p_r$ with all $p_j$ distinct then $n$ occurs $r!$ times.
To simplify the notation of sums as in \eqref{eq:def_S_r}, we will work with a vector notation. 
Specifically, for $1\leq s\leq r$ we employ the vectorized notation
\begin{align}
\textbf{p}_s=(p_1,\ldots,p_{s}), \quad
\textbf{n}_s=(n_1,\ldots,n_{s})
\ \text{ and } \
\mathcal{J}(\textbf{p}_s):=\prod_{i=1}^{s} p_i.
\label{eq:vector_notation_bold_p_one}
\end{align}
Thus we can write $S_r(\alpha,X)$ and $\widetilde{S}_r(\alpha,X)$ in \eqref{eq:def_S_r} as
\begin{align}
S_r(\alpha,X)
&=
\sum_{\substack{\textbf{p}_{r}\in\Pri^{r}\\\mathcal{J}(\textbf{p}_r)\leq t}}\ee(\mathcal{J}(\textbf{p}_r))
\ \text{ and }\
\widetilde{S}_r(\alpha,X)
=
\mathop{\sum}_{\substack{\textbf{n}_{r}\in\N^{r}\\\mathcal{J}(\textbf{n}_r)\leq t}} 
\Lambda(\textbf{n}_r) \ee(\mathcal{J}(\textbf{n}_r)).
\label{eq:Sr_with_vector}
\end{align}
Further, in sums \eqref{eq:Sr_with_vector} we drop the notation $\textbf{p}_r\in\Pri^r$ and $\textbf{n}_r\in\N^r$ and just write 
\begin{align}
S_r(\alpha,X)
&=
\sum_{\substack{\mathcal{J}(\textbf{p}_r)\leq t}} \ee(\mathcal{J}(\textbf{p}_r))
\ \text{ and }\
\widetilde{A}_{r}(t;q,\ell)
=
\mathop{\sum}_{\substack{\mathcal{J}(\textbf{n}_r)\leq t}} \Lambda(\textbf{n}_r) \ee(\mathcal{J}(\textbf{n}_r)).
\label{eq:Sr_with_vector2}
\end{align}
The index of the vector $\textbf{p}_r$ and $\textbf{n}_r$ shows over which set we take the sum.
If a symbol with index is not bold, for instance $p_r$ and $n_r$, then we sum always just over $\Pri$ and $\N$.

Exponential sums have been studied by many authors and bounds are well known for $S_1(\alpha,X)$ and $\widetilde{S}_1(\alpha,X)$. 
The  bound for $S_1(\alpha,X)$ was established by Vinogradov \cite{Vi42} and the proof was later simplified by Vaughan \cite{Va77}.
A much more powerful version was established by the authors in \cite{DoRoZaZe24}.
Explicitly we have the following.
\begin{theorem}[Generalized version of Vinogradov's lemma, {\cite[Theorem~1.1]{DoRoZaZe24}}]
\label{thm:general_vinogradov}
Let $\alpha \in \R$,  $a\in\Z$ and $q\in\N$ as well as $\Upsilon>0$ such that 
\begin{align*}
\left|\alpha-\frac{a}{q}\right|\leq \frac{\Upsilon}{q^2}, \quad (a,q)=1.
\end{align*}
Then for any $X\geq 2$, one has 
\begin{align*}
S_1(\alpha,X)
&
\ll 
\bigg(\frac{X\max\{1,\sqrt{\Upsilon}\}}{\sqrt{q}}+X^{\frac{4}{5}}+\sqrt{X}\sqrt{q}\bigg)(\log X)^3,\\
\widetilde{S}_1(\alpha,X)
&\ll 
\bigg(\frac{X\max\{1,\sqrt{\Upsilon}\}}{\sqrt{q}}+X^{\frac{4}{5}}+\sqrt{X}\sqrt{q}\bigg)(\log X)^4.
%\ll (x\log^3 x)\left(\frac{1}{\sqrt{q}}+\frac{\sqrt{q}}{\sqrt{x}}\right)
\end{align*}
\end{theorem}		
Note that for $\Upsilon=1$, we recover Vinogradov's lemma.
Theorem~\ref{thm:general_vinogradov} turns out to be very powerful for induction arguments.
In particular, the authors were able to establish in \cite{DoRoZaZe24} to following bound
for $S_r(\alpha,X)$ and $\widetilde{S}_r(\alpha,X)$.
%
%We have for $r=2$ the following bound for $S_2(\alpha, X) $.
%%
%\begin{proposition}[{\cite[Proposition~5.1]{semiprimes}}]
%\label{prop:doublevinogradov}
%Let $\alpha \in \R$. If $a \in \Z$ and $q \in \N$ are such that
%\begin{align*}
%\bigg|\alpha - \frac{a}{q}\bigg| \le \frac{1}{q^2} \quad \textnormal{with} \quad (a,q)=1,
%\end{align*}
%then one has that
%\begin{align*}
%S_2(\alpha, X) 
%= 
%\sum_{p_1 p_2 \le X} e(\alpha p_1 p_2) 
%\ll 
%Xq^{-\frac{1}{6}} (\log X)^{\frac{7}{3}} + X^{\frac{16}{17}} (\log X)^{\frac{39}{17}} + X^{\frac{7}{8}}q^{\frac{1}{8}} (\log X)^{\frac{9}{4}},
%\end{align*}
%where the sum is taken over primes $p_1$ and $p_2$.
%\end{proposition}
%%
%For $\widetilde{S}_r(\alpha,X)$ only the case $r=1$ is known.
%%
%\begin{lemma}[e.g. {\cite[Chapter~25]{Da74}}]
%\label{lem:general_vinogradov_lambda}
%Let $\alpha \in \R$,  $a\in\Z$  and $q\in\N$ such that 
%\begin{align*}
%	\left|\alpha-\frac{a}{q}\right|\leq \frac{1}{q^2}, \quad (a,q)=1.
%\end{align*}
%Then one has as $X\to\infty$ that
%\begin{align*}
%	\widetilde{S}_1(\alpha,X)
%	\ll 
%\left(\frac{X}{\sqrt{q}}+X^{\frac{4}{5}}+\sqrt{X}\sqrt{q}\right)(\log X)^4,
%\end{align*}
%where the sum runs over primes.
%\end{lemma}	
%%
%
\begin{theorem}
\label{thm:minorarclemmaprimorial}
Let $\alpha\in\R$, $a\in\Z$, $q\in\N$ and $\Upsilon>0$ such that 
\begin{align}
\left|\alpha-\frac{a}{q}\right|\leq \frac{\Upsilon}{q^2}, \quad (a,q)=1.  
\label{eq:approx_alpha_for_main}
\end{align}
We then have for $X\geq 2$ and  $q \leq X$
\begin{align}
S_r(\alpha,X)
&\ll 
(\log X)^3 \left(X q^{-\frac{1}{2r}} \max\{1,\Upsilon^{\frac{1}{2r}}\}+ X^{\frac{2+2r}{3+2r}} + X^{\frac{2r-1}{2r}}q^{\frac{1}{2r}}\right),
\label{eq:s3finalprimorial}\\
\widetilde{S}_r(\alpha,X)
&\ll 
(\log X)^{3+r} \left(X q^{-\frac{1}{2r}} \max\{1,\Upsilon^{\frac{1}{2r}}\}+ X^{\frac{2+2r}{3+2r}} + X^{\frac{2r-1}{2r}}q^{\frac{1}{2r}}\right).
\label{eq:s3finalprimorial_lambda}
\end{align}
\end{theorem}

We now can use Theorem~\ref{thm:minorarclemmaprimorial} to establish an upper bound for  $\Phi_{\Pri_r}$ and  $\Phi_{\Lambda^{*r}}$ 
on the minor arcs $\mathfrak{m}$ \eqref{eq:def_majyor_M_and_minor_m}.
Using the definition of $\Phi_{\Pri_r}$ in \eqref{eq:def_Phi_prime} and that $\rho =e^{-1/X}$, we can write
\begin{align*}
\Phi_{\Pri_r}(\rho \ee (\alpha)) 
= 
\sum_{j=1}^\infty \frac{1}{j} \sum_{\textbf{p}_r\in\Pri^r} e^{-j\mathcal{J}(\textbf{p}_r)/X} \ee(j \mathcal{J}(\textbf{p}_r) \alpha).
\end{align*}
We have for all $\beta>0$ that
\begin{align*}
e^{-\beta j /X} = \int_{\beta}^\infty jX^{-1}e^{-yj/X}dy.
\end{align*}
Thus we can rewrite the sum infinite over primes as the following combined sum
\begin{align*}
\sum_{\textbf{p}_r\in\Pri^r} e^{-\mathcal{J}(\textbf{p}) j /X} \ee(j \mathcal{J}(\textbf{p}_r)\alpha) 
&=
\int_{2^r}^\infty jX^{-1}e^{-yj/X} \sum_{\mathcal{J}(\textbf{p}_r) \le y} \ee(j \mathcal{J}(\textbf{p}_r) \alpha) dy\\
&=
\int_{2^r}^\infty jX^{-1}e^{-yj/X} 	S_r(j\alpha,y) dy,
\end{align*}
with $S_r(\alpha,X)$ is as in \eqref{eq:Sr_with_vector}.
Since $|S_r(j\alpha,X)|\leq S_r(0,X) \ll X (\log\log X)^{r-1}$ and
using Lemma~\ref{lem:dirk'slemma2} with $\gamma=j/X$, we get for $j\leq X$ for the  coarse but useful bound
\begin{align*}
\int_{2^r}^\infty jX^{-1}e^{-yj/X} S_r(j\alpha,X)\, dy
&\ll 
\int_{2^r}^\infty y (\log\log y)^{r-1}j X^{-1}e^{-yj/X}dy
\ll
(\log\log X)^{r-1}
\bigg(\frac{X}{j}\bigg)^{}.
\end{align*}
%Upon integrating by parts we see that for any $\lambda > 0$ we have
%\begin{align} \label{ibpminor}
%\int_{2^r}^\infty (\log\log y)^{r-1} y^\lambda j X^{-1} e^{-y j /X} dy \ll \bigg(\frac{X}{j}\bigg)^{\lambda+\epsilon},
%\end{align}
%
Let $J\leq X$ be a parameter at our disposal to be chosen later. Then
\begin{align*}
\sum_{j=J+1}^\infty \frac{1}{j} \int_{2^r}^\infty j X^{-1}e^{-y j /X} S_r(j\alpha,X)\,dy 
\ll 
(\log\log X)^{r-1}
\sum_{j=J+1}^\infty \frac{1}{j} \frac{X^{}}{j^{}} 
\ll \frac{X^{1+\epsilon}}{J}.
\end{align*}
We can summarize this by saying that for any $ J \leq X$ we have
\begin{align} 
\label{eq:Phionminor}
\Phi_{\Pri_r}(\rho \ee (\alpha)) 
%&= 
%\sum_{j=1}^J \frac{1}{j} \sum_{p_1}\sum_{p_2} e^{-p_1 p_2 j /X} e(j p_1 p_2 \alpha)dy + O \bigg(\frac{X^{1+\epsilon}}{J}\bigg) \nonumber \\
&= 
\sum_{j=1}^J \frac{1}{j} \int_{2^r}^\infty j X^{-1}e^{-y j /X }S_r(j\alpha,X)\,dy + O \bigg(\frac{X^{1+\epsilon}}{J}\bigg).
\end{align}
Using the definition of \eqref{eq:def_Phi_Lambda}, we get, with a similar computation, that
\begin{align} 
\label{eq:Phionminor_Lambda}
\Phi_{\Lambda^{*r}}(\rho \ee (\alpha)) 
%&= 
%\sum_{j=1}^J \frac{1}{j} \sum_{p_1}\sum_{p_2} e^{-p_1 p_2 j /X} e(j p_1 p_2 \alpha)dy + O \bigg(\frac{X^{1+\epsilon}}{J}\bigg) \nonumber \\
&= 
\sum_{j=1}^J \frac{1}{j} \int_{2^r}^\infty j X^{-1}e^{-y j /X }\widetilde{S}_r(\alpha,X)\,dy + O \bigg(\frac{X^{1+\epsilon}}{J}\bigg)
\end{align}
with $\widetilde{S}_r(\alpha,X)$ as in \eqref{eq:def_S_r}.
Note that \eqref{eq:Phionminor} and \eqref{eq:Phionminor_Lambda} hold for all $\alpha$ and not just those in the minor arcs.
Combining these two inequalities with Theorem~\ref{thm:minorarclemmaprimorial}, we can obtain the following bounds.
\begin{lemma} 
\label{lem:minorarclemmabound}
Let $\mathfrak{m}$ be as in \eqref{eq:def_majyor_M_and_minor_m} and $A>0$. 
We then have for $\alpha\in\mathfrak{m}$ that
\begin{align}
\Phi_{\Pri_r}(\rho \ee (\alpha)) 
&\ll 
X(\log X)^{3-\frac{A}{2r}}
%\label{eq:lem:minorarclemmabound_primes}\\
\ \text{ and } \
\Phi_{\Lambda^{*r}}(\rho \ee (\alpha)) 
\ll 
X(\log X)^{3+r-\frac{A}{2r}}.
\label{eq:lem:minorarclemmabound_Lambda}
\end{align}
\end{lemma}
%
%The inequalities in this lemma hold for all $A\geq0$. 
%However, they are only meaningful if the upper bounds are lower than $\Phi_{\Pri_r}(\rho)$ and $\Phi_{\Lambda^{*r}}(\rho)$, see Theorem~\ref{thm:Phi_asympt_principal} or Corollary~\ref{cor:lemma5.1}.
%Thus we require for $\Phi_{\Pri_r}$ that $3-\frac{A}{3^{r}-1}<-1$,
%and for $\Phi_{\Lambda^{*r}}$ that   $3+r-\frac{A}{3^{r}-1}<r-1$.
%Both inequalities are equivalent to $A>4\cdot (3^{r}-1)$.
%We immediately get
%\begin{lemma} 
%\label{lem:cul_minor_arcs}
%Let $\mathfrak{m}$ and $A$ be as in \eqref{eq:def_majyor_M_and_minor_m}. 
%We then have for all $\alpha\in\mathfrak{m}$ that
%\begin{align*}
%\Re(\Phi_{\Pri_r}(\rho \ee(\alpha))) 
%\leq  
%\frac{3}{4}\Phi_{\Pri_r}(\rho),
%\qquad 
%\Re(\Phi_{\Lambda^{*r}}(\rho \ee(\alpha))) 
%\leq  
%\frac{3}{4}\Phi_{\Lambda^{*r}}(\rho),
%\end{align*}	
%\end{lemma}

\begin{proof}%[Proof of Lemma~\ref{lem:minorarclemmabound}]
The proofs for $\Phi_{\Pri_r}$ and  $\Phi_{\Lambda^{*r}}$ are (almost) identical and we thus give it only for $\Phi_{\Pri_r}$.	
Fix $J$ be a parameter of our choice define the $y$-integral in \eqref{eq:Phionminor} by
\begin{align} 
\label{expressionboundminor}
\mathfrak{I}(X,j)
:=
\int_{2^r}^\infty j X^{-1}e^{-y j /X } S_r(j\alpha ,y) dy. %\sum_{p_1 p_2 \le x}  \e(j p_1 p_2 \alpha)
\end{align}
For each $j\leq J$, we employ Dirichlet's theorem (see \cite[Lemma 2.1]{Va77}) to choose $a \in \Z$ and $q \in \N$ with $(a,q)=1$, both maybe dependent on $j$, such that 
\begin{align}
\bigg| j \alpha - \frac{a}{q} \bigg| 
\le 
q^{-1} X^{-1} (\log X)^A \quad \textnormal{and} \quad q < X^{}(\log X)^{-A}.
\label{eq:lem_5_3_approx_dirichlet}
\end{align}
Since $q^{-1} X^{-1} (\log X)^A\leq q^{-2}$, Theorem~\ref{thm:minorarclemmaprimorial} implies
\begin{align} 
\label{eq:auxminorarcs01}
S_r(j\alpha, y)
\ll
(\log y)^{3} \left(y q^{-\frac{1}{2r}} + y^{\frac{2+2r}{3+2r}} + y^{\frac{2r-1}{2r}}q^{\frac{1}{2r}}\right).
\end{align}
Inserting this into $\mathfrak{I}(X,j)$ and using Lemma~\ref{lem:dirk'slemma2} with $\gamma=j/X$, we obtain
\begin{align}
\mathfrak{I}(X,j)
\ll
\left( \log \frac{X}{j} \right)^{3}
\bigg( \frac{X}{jq^{\frac{1}{2r}}}  + \left(\frac{X}{j}\right)^{\frac{2+2r}{3+2r}} +\left(\frac{X}{j}\right)^{\frac{2r-1}{2r}} q^{\frac{1}{2r}}  \bigg)   , 
\end{align}
We now have to bound the terms involving $q$.  
We use for $q^{\frac{1}{2r}}$ that $q < X^{}(\log X)^{-A}$.
For the term $q^{-\frac{1}{2r}}$, we use that $\alpha \in \mathfrak{m}$.
First, we divide \eqref{eq:lem_5_3_approx_dirichlet} by $j$ and then define $a_j:= a/(a,j)$ and $q_j:=jq/(a,j)$.
The definition of $\delta_q$ in \eqref{eq:def_delta_q} then implies
\begin{align}
\bigg|  \alpha - \frac{a_j}{q_j} \bigg| \le \delta_{q_j}.
\label{eq:lem_5_3_approx_dirichlet_2}
\end{align}
Since $\alpha \in \mathfrak{m}$, the definition of $\mathfrak{m}$ implies that $jq\geq q_j > Q$, where $Q=(\log X)^A$. 
Thus 
\begin{align}
\mathfrak{I}(X,j)
\ll
\left( \log X \right)^{3}
\bigg( \frac{X}{j^{\frac{2r-1}{2r}} (\log X)^{\frac{A}{2r}}}  + \left(\frac{X}{j}\right)^{\frac{2+2r}{3+2r}} +\frac{X}{j^{\frac{2r-1}{2r}}(\log X)^{\frac{A}{2r}}}  \bigg). 
\end{align}
Inserting this into \eqref{eq:Phionminor}, we get
\begin{align*}
\sum_{j=1}^J 
\frac{\mathfrak{I}(X,j)}{j}
\ll\,&
(\log X)^3 
\bigg(
\frac{X}{(\log X)^{\frac{A}{2r}}} \sum_{j = 1}^J   \frac{1}{j^{\frac{4r-1}{2r}}}
+
X^{\frac{2+2r}{3+2r}}\sum_{j = 1}^J \frac{1	}{j^{\frac{5+4r}{3+2r}}}
\bigg).
\end{align*}
Since the power of $j$ in both sums is strictly larger than one we obtain 
\begin{align*}
\Phi_{\Pri_r}(\rho \ee (\alpha)) 
&= 
\sum_{j=1}^J 
\frac{\mathfrak{I}(X,j)}{j} + O \bigg(\frac{X^{1+\epsilon}}{J}\bigg)
\ll
X(\log X)^{3-\frac{A}{2r}}
+ 
X^{\frac{2+2r}{3+2r}}(\log X)^3 
+ O \bigg(\frac{X^{1+\epsilon}}{J}\bigg).
\end{align*}
The second term is lower than the first term since $\frac{2+2r}{3+2r}<1$.
%Further, Lemma~\ref{lem:properties_of_beta} implies that $1-\beta_2(r) = \frac{1}{3^r-1}$. 
The result now follows by choosing $J=\sqrt{X}$.
\end{proof}

The inequalities in Lemma~\ref{lem:minorarclemmabound} hold for all $A>0$. 
However, they are only meaningful if these upper bounds are lower than $\Phi_{\Pri_r}(\rho)$ and $\Phi_{\Lambda^{*r}}(\rho)$. Theorem~\ref{thm:Phi_asympt_principal} (or Corollary~\ref{cor:lemma5.1}) implies
%\begin{align*}
%	\Phi_{\Pri_r}(\rho \ee(\alpha))
%	&=
%	\frac{\zeta(2)rX(\log\log X)^{r-1}}{(1-2\pi i \alpha X)\log X}+O\left(\frac{X(\log\log X)^{r-2}}{\log X}
%	\right),\\
%	\Phi_{\Lambda^{*r}}(\rho \ee(\alpha))
%	&=
%	\frac{\zeta(2)X(\log X)^{r-1}}{(r-1)!(1-2\pi i \alpha X)}+O\left(X(\log\log X)^{r-2}\right).
%\end{align*}
\begin{align*}
\Phi_{\Pri_r}(\rho)
\sim
\frac{\zeta(2)rX(\log\log X)^{r-1}}{\log X}
\ \text{ and } \
\Phi_{\Lambda^{*r}}(\rho)
\sim 
\frac{\zeta(2)X(\log X)^{r-1}}{(r-1)!}.
\end{align*}
Thus we require for $\Phi_{\Pri_r}$ that $3-\frac{A}{2r}<-1$,
and for $\Phi_{\Lambda^{*r}}$ that   $3+r-\frac{A}{2r}<r-1$.
Both inequalities are equivalent to $A>8r$, which is exactly the condition in \eqref{eq:def_majyor_M_and_minor_m}.
We immediately get
\begin{lemma} 
\label{lem:cul_minor_arcs}
Let $\mathfrak{m}$ and $A$ be as in \eqref{eq:def_majyor_M_and_minor_m}.
Then there exists a $\rho_0$ such that for all  $\rho\geq \rho_0$ and  all $\alpha\in\mathfrak{m}$ we have
\begin{align*}
\Re(\Phi_{\Pri_r}(\rho \ee(\alpha))) 
\leq  
\frac{3}{4}\Phi_{\Pri_r}(\rho),
\ \text{ and } \
\Re(\Phi_{\Lambda^{*r}}(\rho \ee(\alpha))) 
\leq  
\frac{3}{4}\Phi_{\Lambda^{*r}}(\rho).
\end{align*}	
\end{lemma}

\section{Pseudo-differentiable functions}
\label{sec:pseudo-differentiable-functions-and-non-principal-arcs}
The aim of this section is to provide the necessary framework to establish upper bounds for the real parts of $\Phi_{\Pri^r}$ and of $\Phi_{\Lambda^{*r}}$ on the non-principal major arcs.

\subsection{Motivation}
\label{sec:motivation}

To establish an upper bound for the real part of $\Phi_{\Pri^r}$ and of $\Phi_{\Lambda^{*r}}$ on the non-principal major arcs $\mathfrak{M}(q,a)$ with $2\leq q\leq Q$, we have to study for $q,\ell\in\N$ with $(q,\ell)=1$ the expression $A_{r}(t;q,\ell)$ and $\widetilde{A}_{r}(t;q,\ell)$ with
\begin{align}
A_{r}(t;q,\ell)
&:=
\mathop{\sum\sum}_{\substack{p_1,\ldots,p_r\\p_1 \cdots p_r\leq t\\  p_1\ldots p_r\equiv\ell \bmod q }} 1
\quad \text{ and } \quad 
\widetilde{A}_{r}(t;q,\ell)
:=
\mathop{\sum\sum}_{\substack{n_1,\ldots,n_r\in\N\\n_1 \cdots n_r\leq t\\  n_1\ldots n_r\equiv\ell \bmod q }} \prod_{i=1}^{r}\Lambda(n_i)
\label{eq:def_A_r}	
\end{align}
and $A_{r}(t):= A_{r}(t;1,0)$ and $\widetilde{A}_{r}(t):= A_{r}(t;1,0)$.
More precisely, we have to show that we can write 
\begin{align}
A_{r}(t;q,\ell) = D_r(t) + O\left(t(\log t)^{-C}\right)
\ \text{ and }\ 
\widetilde{A}_{r}(t;q,\ell) = \widetilde{D}_r(t) + O\left(t(\log t)^{-C}\right),
\label{eq:Ar_what_we would_like_to_have}
\end{align}
where $C\geq 1$ is arbitrarily large, $D_r(t)$ and $\widetilde{D}_r(t)$ are differentiable with respect to $t$ and $D_r'(t)$ and $\widetilde{D}'_r(t)$ can be bounded suitable.
Obverse that we have 
%Recall, $\pi(t;q,a)$ denotes the number of primes less than or equal to $t$ which are congruent to $a$ mod $q$. 
%In formulae
\begin{align}
A_{1}(t;q,\ell)
=
\pi(t;q,\ell) 
= 
\card\,\{p\leq t \,;\, p \equiv \ell \bmod q \}
\ \text{ and }  \
\widetilde{A}_{1}(t;q,\ell)
=
\psi(t;q,\ell),
\end{align}
where $\pi(t;q,\ell)$ denotes the number of primes less than or equal to $t$ which are congruent to $\ell$ mod $q$
and where $\psi(t;q,\ell)$ is the second Chebyshev function.
The Siegel-Walfisz theorem then states that for every natural number $A\in\N$, there exists a positive constant $C_A>0$ such that as $t\to\infty$ and $\ell,q\in\N$ satisfies $(\ell,q)=1$ and $q\leq (\log t)^A$, we have
\begin{align}
\pi(t;q,\ell)
&=
\frac {\Li(t)}{\varphi(q)} + O\bigg(t\exp \bigg(-\frac {C_{A}}{2}(\log t)^{\frac {1}{2}}\bigg)\bigg)
\ \text{ and }  \\
\psi(t;q,\ell)
&=
\frac {t}{\varphi(q)} + O\bigg(t\exp \bigg(-\frac {C_{A}}{2}(\log t)^{\frac {1}{2}}\bigg)\bigg),
\end{align}
where $\Li(t)$ is the logarithmic integral and $\varphi(q)$ is Euler's totient function.
We do not require the Siegel-Walfisz theorem in full strength.
We will use it only in the form 
\begin{align}
\pi(t;q,\ell)
=
\frac {\Li(t)}{\varphi(q)} +  O\big(t (\log t)^{-C} \big)
\ \text{ and } \
\psi(t;q,\ell)
=
\frac {t}{\varphi(q)} +  O\big(t (\log t)^{-C} \big)
\label{eq:Siegel-Wal}
\end{align}
as $t\to\infty$ with $q\leq (\log t)^A$ and $C\geq 1$ arbitrary.
In particular, the implicit constant in the $O(\cdot)$ term of \eqref{eq:Siegel-Wal} only depends on $A$, but not on $q$ or $\ell$.
Thus \eqref{eq:Ar_what_we would_like_to_have} is fulfilled in the case $r=1$.
Further, we have for $r=2$ the following result.
\begin{theorem}[{\cite[Theorem~6.1]{semiprimes}}]
\label{thm:Ugly_A_mod_q}
We have as $t\to\infty$ and $q,\ell\in\N$ with $(\ell,q)=1$ and  $q\leq (\log t)^A$ with $A>0$ that
\begin{align}
A_2(t;q,\ell)
=\, &
\frac{2}{\varphi(q)}
\int_2^{\sqrt t }
(\log \log u + M +O((\log u)^{-C})\left(\operatorname{Li}(t/u) -\frac{t}{u \log (t/u) }\right) du \nonumber\\
&-\frac{\operatorname{Li}^2(\sqrt{t})}{\varphi(q)}+\frac{2\sqrt{t}}{\varphi(q)}\operatorname{Li}(\sqrt{t})  \left( {\log \log t - \log 2 + M } \right)
+ O\left(t (\log t)^{-C} \right),
\label{eq:A_2_explicit}
\end{align}
where $C>0$ is arbitrarily large.
\end{theorem}
Theorem~\ref{thm:Ugly_A_mod_q} implies, with a small computation, that \eqref{eq:Ar_what_we would_like_to_have} is fulfilled for $r=2$, see proof of Lemma~6.2 in \cite{semiprimes}. 

\begin{color}{black}
Unfortunately, explicit expressions like \eqref{eq:Siegel-Wal} or \eqref{eq:A_2_explicit} are not easy to obtain for $r\geq 3$ and are also quite complicated.
Also, the argument to obtain the Siegel-Walfisz theorem does not seem to work in this situation. We are thus using induction over $r$.
So we need an approach that is compatible with an induction argument and that can also be implemented.
If we analyze the proof carefully in the cases $r=1$ and $r=2$, we see that we only need \eqref{eq:Ar_what_we would_like_to_have} and the leading terms of $D_r(t)$ and $D_r'(t)$, but explicit expressions for $D_r(t)$ and $D_r'(t)$ are not needed. This observation motivates the introduction of the terminology of pseudo-differentiable functions.
\end{color}

\subsection{Definition of pseudo-differentiable functions}
\label{sec:strange-functions}

We have to consider  functions, where the leading term has the form $\frac{t P_{a}(\log\log t)}{(\log t)^b}$ with $P_a$ a polynomial, but the function itself is not a rational function in $\log t$ and $\log\log t$.
An example is the logarithmic integral $\Li$.
Furthermore, we are dealing with functions that are approximately smooth, but all derivatives have a small, non-differentiable error term. For this type of functions we introduce the following class.
\begin{definition}
\label{def:strange_function}
Let $g$ be a smooth function on $[2,\infty]$ and $N\in\N_0$.
Then a function $f$ on $[2,\infty]$ is called $N$ times pseudo-differentiable with respect to $g$ if there exists a sequence $\{f^{(n)}\}_{n=0}^{N}$ of $C^1$-functions such that as $t\to\infty$
\begin{align}
f(t) 
&= 
f^{(0)}(t) +O\left(g(t)\right),\\
\intertext{and for all $0\leq n\leq N-1$}
(f^{(n)})'(t) 
&= 
f^{(n+1)}(t)+O(g^{(n+1)}(t)),
\label{eq:def_nth_quasi_diff}
\end{align}
where $(f^{(n)})'(t)$ is the derivative of $f^{(n)}(t) $ and $g^{(n)}(t)$ is the $n$th derivative of $g$. 
We will call $f^{(n)}$ the $n$th pseudo-derivative of $f$ with respect to $g$.
If a function $f$ is infinitely often pseudo-differentiable with respect to $g$ 
then we call the function strange with respect to $g$.
\end{definition}
The implicit constant in $O(\cdot)$ in \eqref{eq:def_nth_quasi_diff} may depend on $n$. 
To simplify the notation, we will  omit `\emph{with respect to $g$}' when the function $g$ is clear from the context.
Note that the sequence $\{f^{(n)}\}_{n=0}^{\infty}$ is in general not unique.
However, if $f$ is a smooth function, then $f$ is automatically strange with respect to any function $g$ and we always choose the $n$-th derivative of $f$ for $f^{(n)}$ in this case.
Furthermore, we will only work with functions on $[2,\infty]$. 
Therefore, we will no longer state this and assume that the reader is aware of that.

In this paper, we are mainly interested in strange functions with respect to 
$t(\log t)^{-C}$ with $C\geq 1$.
%$\frac{t}{(\log t)^C}$ with $C\geq 1$.
%
Lemma~\ref{lem:derivatives_loglog^a_log^b} (or a direct computation) shows that the derivatives of $t(\log t)^{-C}$ are
\begin{align}
\left(\frac{t}{(\log t)^C}\right)' 
&\asymp
\frac{1}{(\log t)^C},
\ \text{ and } \
\left(\frac{t}{(\log t)^C}\right)^{(n)} 
\asymp
\frac{1}{t^{n-1}(\log t)^{C+1}}
\ \text{ for }n\geq 2,
\label{eq:derivatives_tlogt_C}
\end{align}
where $f\asymp g$ denotes that there exist constants $M\geq m>0$ such that 
\begin{align*}
mf(t)\leq g(t)\leq Mf(t). 
\end{align*}

Examples of strange functions with respect to $t(\log t)^{-C}$  are 
\begin{align}
\pi(t;q,a) 
\ \text{ and } \
\int_{2}^{\sqrt{t}} \pi(u;q,a) \frac{t\Li(t/u)}{u^2}\, du.
\label{eq:examples_strange_functions}
\end{align}
The Siegel-Walfisz theorem immediately gives that $\pi(t;q,a)$ is strange.
The second one requires more work, see Lemma~\ref{lem:int_pi_D_is_strange}.

The main purpose of introducing pseudo-differential functions is to establish an asymptotic expression for $A_r(t)$ similar to the expressions for $A_1(t)$ and $A_2(t)$. 
A natural question at this point is: how many pseudo-derivatives do we need? 
%The expression for $A_2(t)$ is a combination of the two terms in \eqref{eq:examples_strange_functions}. 
To determine the behavior of $A_{r}(t)$ for $r\geq 3$, we use induction over $r$.
In particular, we use the Abel summation formula to give in Lemma~\ref{lem:explcit_Ar_as_A{r-1}} an expression for $A_{r}(t)$, which requires the first pseudo-derivative of $A_{r-1}(t)$.
This implies that we need $r-2$ pseudo-derivatives of $A_2(t)$ to be able to determine the asymptotic behavior of $A_{r}(t)$. 
Since $r\in\N$ is arbitrary, we require that $A_2(t)$ be strange.
Explicitly, we will show
\begin{theorem}
\label{thm:Ar_is_strange}
Let $r\geq 1$ be given. 
Further, as $t\to\infty$, assume that $\ell$, $q\in\N$ with $q\leq (\log t)^A$ for some $A>0$ and $(q,\ell) =1$.
Then the function $A_{r}(t;q,\ell)$ is strange with respect to $t(\log t)^{-C}$ for all $C\geq C_{0}$ 
for some $C_0=C_{0}(r)>0$.
Further,
\begin{align}
	A_{r}(t;q,\ell)
	=
	A_{r}^{(0)}(t;q) +O(t(\log t)^{-C}),
	\label{eq:As_aympt_leading_only_q_and_O:uniform}
\end{align}
where $A_{r}^{(0)}(t,q) $ is independent of $\ell$ and the implicit constant in $O(\cdot)$ can be chosen independent of $q$ and $\ell$.	
Furthermore, we have for $n=0$ and $n=1$
\begin{align}
	A_{r}^{(n)}(t;q) 
	&= 
	\frac{r}{\varphi(q)}\frac{t^{1-n}(\log\log t)^{r-1}}{\log t } + O\left(\frac{t^{1-n}(\log\log q)(\log\log t)^{r-2}}{\varphi(q) \log t}\right),
	\label{eq:find_a_cool_name_0_1}\\
	\intertext{and for $n\geq 2$}
	A^{(n)}_r(t;q)  
	&=
	(-1)^{n+1}\frac{r(n-2)!}{\varphi(q)} 
	\frac{(\log\log t)^{r-1}}{t^{n-1}(\log t)^2} 
	+  
	O\left(\frac{(\log\log q)(\log\log t )^{r-2}}{\varphi(q)t^{n-1}(\log t)^2}\right).
	\label{eq:find_a_cool_name2}
\end{align}
\end{theorem}
\begin{color}{black}
We have in Theorem~\ref{thm:Ar_is_strange} that $A_{r}(t;q,\ell)= A_{r}(t;q(t),\ell(t))$ since  $\ell$ and $q$ can depend in general on $t$. For clarity, the (pseudo-)derivatives in Theorem~\ref{thm:Ar_is_strange} are the partial (pseudo-)derivatives with respect to $t$ and not the total (pseudo-)derivative of $A_{r}(t;q,\ell)$.
\end{color}

The expressions in \eqref{eq:find_a_cool_name_0_1} and \eqref{eq:find_a_cool_name2} can be improved. 
We expect for $n=0$ and $n=1$ that
\begin{align}
	A_{r}^{(n)}(t)
	&= 
	\frac{t^{1-n}}{\varphi(q)}
	\sum_{k=0}^K \frac{P_{r,n,k}(\log\log t)}{(\log t)^{1+k}} 
	+ O\left(t^{1-n} \frac{(\log\log t)^{a-1}}{(\log t)^{2+K}}\right)
\intertext{and for $n\geq 2$}
	A_{r}^{(n)}(t)
	&= 
	\frac{t^{1-n}}{\varphi(q)}
	\sum_{k=0}^K \frac{P_{r,n,k}(\log\log t)}{(\log t)^{2+k}} 
	+ O\left( t^{1-n}\frac{(\log\log t)^{a-1}}{(\log t)^{3+K}}\right),
\end{align}
where $P_{r,n,k}$ are polynomials of degree $a-1$. 
However, the proof of Theorem~\ref{thm:Ar_is_strange} is little bit involved and the stated version is sufficient for our purposes. 
Thus we do not try to establish a stronger version.
Furthermore, we have the following result.
\begin{theorem}
	\label{thm:Ar_is_strange_lambda}
	Let $r\geq 1$ be given. 
	Further, as $t\to\infty$, assume that $\ell$, $q\in\N$ with $q\leq (\log t)^A$ for some $A>0$ and $(q,\ell) =1$.
	Then function $\widetilde{A}_{r}(t;q,\ell)$ is strange with respect to $t(\log t)^{-C}$ for all $C\geq C_{0}$ 
	for some $C_0=C_{0}(r)>0$.
	Further,
	\begin{align}
		\widetilde{A}_{r}(t;q,\ell)
		=
		\widetilde{A}_{r}^{(0)}(t;q) +O(t(\log t)^{-C}),
		\label{eq:As_aympt_leading_only_q_and_O:uniform_lambda}
	\end{align}
	where $A_{r}^{(0)}(t,q) $ is independent of $\ell$ and the implicit constant in $O(\cdot)$ can be chosen independently of $q$ and $\ell$.	
	Furthermore, we have for $n=0$ and $n=1$ that
	\begin{align}
		\widetilde{A}_{r}^{(n)}(t;q) 
		&= 
		\frac{t^{1-n}(\log t)^{r-1}}{(r-1)!\varphi(q)} + O\left(\frac{t^{1-n}(\log\log q)(\log t)^{r-2}}{\varphi(q)}\right),
		\label{eq:find_a_cool_name_0_1_lambda}\\
		\intertext{and for $n\geq 2$}
		\widetilde{A}^{(n)}_r(t;q)  
		&=
		(-1)^{n}\frac{(n-2)!}{(r-2)!\varphi(q)} 
		\frac{(\log t)^{r-2}}{t^{n-1}} 
		+  
		O\left(\frac{(\log\log q)(\log t )^{r-3}}{\varphi(q)t^{n-1}}\right).
		\label{eq:find_a_cool_name2_lambda}
	\end{align}
\end{theorem}
\begin{color}{black}
The proofs of the Theorems~\ref{thm:Ar_is_strange} and~\ref{thm:Ar_is_strange_lambda} are unfortunately a little bit involved.
Therefore, we complete this section by establishing some  properties of pseudo-differentiable functions in the general setting
and then give the proofs of these theorems in Section~\ref{sec:proof-of-the-theoremsrefthmarisstrange-andrefthmarisstrangelambda}.
After that, we use the Theorems~\ref{thm:Ar_is_strange} and~\ref{thm:Ar_is_strange_lambda} in Section~\ref{sec:upper-bound-for-the-real-part-of-phimathbbpr-and-philambdar} to establish upper bounds for the real parts of $\Phi_{\Pri^r}$ and of $\Phi_{\Lambda^{*r}}$ on the non-principal major arcs.
\end{color}

\subsection{Basic properties of pseudo-differential functions}
\label{sec:basic-properties-of-pseudo-differential-functions}

Pseudo-differential functions share many properties with smooth functions, but require usually some additional assumptions.
Examples are the chain rule and product rule.
\begin{lemma}[Chain rule for strange functions]
\label{lem:chain_rule_strange}
Let $\varphi:[2,\infty]\to  [2,\infty]$ a smooth function
with $\varphi(t)\to\infty$ as $t\to\infty$ and let $f$ be strange with respect to $g$.
Additionally, suppose that for any $n\geq k\geq 1$ and any sequence $\{m_j\}_{j=1}^{n-k+1}$ of non-negative integers satisfying
\begin{align}
\sum_{j=1}^{n-k+1} m_{j} =k \ \text{ and } \ \sum_{j=1}^{n-k+1} jm_{j} =n,
\label{eq:sum_k_sum_n_chain_rule}
\end{align}
the following condition holds:
\begin{align}
g^{(k)}(\varphi(t)) \prod_{j=1}^{n-k+1}(\varphi^{(j)}(t))^{m_j}
=
O(g(\varphi(t))^{(n)}).
\label{eq:chain_rule_assumption_on_g}
\end{align}
Then $f(\varphi)$ is strange with respect to $g(\varphi)$.
Furthermore,
\begin{align}
\big(f(\varphi(t))\big)^{(1)}
=
f^{(1)}(\varphi(t)) \varphi'(t)
\label{eq:chain_rule_starnge_1th_derivative}
\end{align}
and for $n\in\N$
\begin{align}
\big(f(\varphi(t))\big)^{(n)}
=
\sum_{k=1}^{n}
f^{(k)}(\varphi(t)) B_{n,k}\big(\varphi^{(1)}(t),\varphi^{(2)}(t),\ldots,\varphi^{(n-k+1)}(t)\big),
\label{eq:chain_rule_starnge_nth_derivative}
\end{align}
where $B_{n,k}$ are the Bell polynomials.
\end{lemma}
Thus we see that the chain rule(s) for strange functions in \eqref{eq:chain_rule_starnge_1th_derivative} and \eqref{eq:chain_rule_starnge_nth_derivative} have the same form as for smooth functions. 
Recall that the Bell polynomials are given by
\begin{align*}
B_{n,k}(x_1,\ldots ,x_{n-k+1})
=
\sum \frac{n!}{m_{1}!m_{2}!\cdots m_{n-k+1}!} 
\prod_{j=1}^{n-k+1}\left(\frac{x_{j}}{j!}\right)^{m_{j}},
\end{align*}
where the sum is taken over all sequences $\{m_j\}_{j=1}^{n-k+1}$ fulfilling \eqref{eq:sum_k_sum_n_chain_rule}.
In particular, we have
\begin{align}
B_{n,1}(x_1,\ldots,x_{n}) = x_{n}
\ \text{ and } \
B_{n,n}(x_1,\ldots,x_{n}) = B_{n,n}(x_1) =x_{1}^n.
\label{eq:Bell_poly_special_cases}
\end{align}
\begin{proof}
We prove the lemma by induction. We start with $n=1$.
By assumption, we have
\begin{align*}
f(t) = f^{(0)}(t) + O\left(g(t)\right).
\end{align*} 
Since $\varphi(t)\to\infty$ as $t\to\infty$, we immediately get 
\begin{align*}
f(\varphi(t)) = f^{(0)}(\varphi(t)) + O\left(g(\varphi(t))\right).
\end{align*} 
Since $f^{(0)}$ is $C^1$, the chain rule and that $f$ is strange imply
\begin{align}
\big(f^{(0)}(\varphi(t))\big)' 
=
(f^{(0)})'(\varphi(t))\varphi'(t)
&=
\big(f^{(1)}(\varphi(t)) +O(g^{(1)}(\varphi(t)))\big)\varphi'(t)\nonumber\\
&=
f^{(1)}(\varphi(t))\varphi'(t) +O\big(\big(g(\varphi(t))\big)'\big).
\label{eq:proof_chain_n=1}
\end{align}
This implies that \eqref{eq:chain_rule_starnge_nth_derivative} holds for $n=1$.
Suppose now \eqref{eq:chain_rule_starnge_nth_derivative} holds for some $n\in\N$.
Since $f$ is strange, the chain rule gives for each $k\in\N$ with the same computation as in \eqref{eq:proof_chain_n=1} that
\begin{align}
\big(f^{(k)}(\varphi(t))\big)'
%&=
%(f^{(k)})'(\varphi(t))\varphi'(t)
%=	
%\Big(f^{(k+1)}(\varphi(t)) +O\left(g^{(k+1)}(\varphi(t))\right)\Big)\varphi'(t)\\
&=
f^{(k+1)}(\varphi(t))\varphi'(t) + O\big(g^{(k+1)}(\varphi(t))\varphi'(t)\big).
\label{eq:proof_chain_n=k}
\end{align}
Thus applying the product to \eqref{eq:chain_rule_starnge_nth_derivative} and inserting \eqref{eq:proof_chain_n=k} gives
\begin{align}
\big(\big(f(\varphi(t))\big)^{(n)}\big)'
=\,&
\sum_{k=1}^{n}
f^{(k+1)}(\varphi(t))\varphi'(t) B_{n,k}\big(\varphi^{(1)}(t),\varphi^{(2)}(t),\ldots,\varphi^{(n-k+1)}(t)\big)\nonumber\\
&+
\sum_{k=1}^{n}
f^{(k)}(\varphi(t)) \big(B_{n,k}\big(\varphi^{(1)}(t),\varphi^{(2)}(t),\ldots,\varphi^{(n-k+1)}(t)\big)\big)'\nonumber\\
&+
\sum_{k=1}^{n}
O(g^{(k+1)}(\varphi(t))\varphi'(t)) B_{n,k}\big(\varphi^{(1)}(t),\varphi^{(2)}(t),\ldots,\varphi^{(n-k+1)}(t)\big).
\label{eq:chain_rule_proof_nth_step_ugly_derivative}
\end{align}

The first two lines have exactly the same form as for the ordinary chain rule.
By leveraging the combinatorial properties of the Bell polynomials, we are able to combine these two lines and obtain equation \eqref{eq:chain_rule_starnge_nth_derivative} with $n+1$.
Finally, assumption \eqref{eq:chain_rule_assumption_on_g} and the definition of the Bell polynomials show that 
\begin{align*}
g^{(k+1)}(\varphi(t))\varphi'(t) B_{n,k}\big(\varphi^{(1)}(t),\varphi^{(2)}(t),\ldots,\varphi^{(n-k+1)}(t)\big)
\ll
g(\varphi(t))^{(n+1)}.
\end{align*}
Inserting this into \eqref{eq:chain_rule_proof_nth_step_ugly_derivative} completes the proof. 
\end{proof}

\begin{lemma}[Product rule for strange functions]
\label{lem:product_rule_strange}
Let $f$ be a strange function with respect to $g$. Further, let $h$ and $m$ be a smooth functions such that for all $k,\ell\in\N_0$
\begin{align}
g^{(k)}(t)h^{(\ell)}(t) \ll m^{(k+\ell)}(t).%O\left( \big(g(t)h(t)\big)^{(k+\ell)}\right).
\label{eq:lem:product_rule_strange_assumption_omn_derivatives}
\end{align}
Then $fh$ is strange with respect to the function $m$ and
\begin{align}
(fh)^{(n)}(t) 
=
\sum_{k=0}^{n} \binom{n}{k} f^{(k)}(t) h^{(n-k)}(t).
\label{eq:strange_product_rule}
\end{align}
\end{lemma}
%
%In many cases it is possible to choose $m=gh$, but in general this is wrong.
%An example, which will appear in our calculations below, 
%is $f(t)=g(t)=h(t) = \sqrt{t}\log^{-1}(t)$.
%In this case, $f$ is trivially strange with to $g$.
%On the other hand
%\begin{align}
%f^{(1)}(t)h^{(1)}(t)
%&= 
%%\left(\frac{1}{\sqrt{t} \log(t)}-\frac{2}{\sqrt{t} \log^2(t)}\right)^2
%%=
%\frac{1}{4t(\log t)^2} + O\left(\frac{1}{t(\log t)^3}\right),\quad
%(g(t)h(t))^{(2)}
%=
%-\frac{2}{t(\log t)^3} + \frac{24}{t\log^4(t)}.
%\end{align}
%Thus $g^{(1)}(t)h^{(1)}(t) \neq O\big((g(t)h(t))^{(2)}\big)$.
%
\begin{proof}
We prove the statement by induction. 
We have
\begin{align}
f(t)h(t)
=
\big(f^{(0)}(t)+ O(g(t))\big)h(t) 
=
f^{(0)}(t)h(t) + O(g(t)h(t)).
\end{align}
Since $f^{(0)}(t)$ is differentiable by assumption, $f^{(0)}(t)h(t)$ is also differentiable.
Thus $f(t)h(t)$ is $0$ times pseudo-differentiable and \eqref{eq:strange_product_rule} hold for $n=0$.

Suppose that $f(t)h(t)$ is $n$ times pseudo-differentiable and 
\eqref{eq:strange_product_rule} holds $n\in\N_0$.
%We now show that $f(t)h(t)$ is $n+1$ times pseudo-differentiable and \eqref{eq:strange_product_rule} holds for $n+1$.
Then $(fh)^{(n)}(t)$ is differentiable and the induction assumption implies
\begin{align*}
\big((fh)^{(n)}\big)'
&=
\bigg(\sum_{k=0}^{n} \binom{n}{k} f^{(k)} h^{(n-k)}\bigg)'\\
&=
\sum_{k=0}^{n} \binom{n}{k} \left( (f^{(k)})' h^{(n-k)} + f^{(k)} h^{(n-k+1)} \right)\\
&=
\sum_{k=0}^{n} \binom{n}{k} \left( f^{(k+1)} h^{(n-k)} + f^{(k)} h^{(n-k+1)}  + O(g^{(k+1)}h^{(n-k)})\right)\\
%	+
%	\sum_{k=0}^{n} \binom{n}{k} O\left(g^{(k+1)}h^{(n-k)}\right)\\
&=
\sum_{k=0}^{n+1} \binom{n+1}{k} f^{(k)} h^{(n+1-k)} 
+
O(m^{(n+1)}),
\end{align*}
where we used on the last line assumption \eqref{eq:lem:product_rule_strange_assumption_omn_derivatives}.
This completes the proof.
\end{proof}
In many cases it is possible to choose $m=gh$ in Lemma~\ref{lem:product_rule_strange}. 
However, there are examples where we cannot choose $m=gh$.  
One example, which will appear in our calculations, is $g(t)=h(t) = \sqrt{t}(\log t)^{-1}$.
We have in this case 
\begin{align}
g^{(1)}(t)h^{(1)}(t)
&= 
\frac{1}{4t(\log t)^2} + O\bigg(\frac{1}{t(\log t)^3}\bigg),\nonumber\\%\quad
\big(g(t)h(t)\big)^{(2)}
&=
-\frac{2}{t(\log t)^3} + \frac{24}{t(\log t)^4}.
\end{align}
Thus $g^{(1)}(t)h^{(1)}(t) \neq O\big((g(t)h(t))^{(2)}\big)$.
In this situation, we can choose $m(t) = t(\log t)^{-1}$.

\section{{Proof of the Theorems~\ref{thm:Ar_is_strange} and~\ref{thm:Ar_is_strange_lambda}}}
\label{sec:proof-of-the-theoremsrefthmarisstrange-andrefthmarisstrangelambda}
%
%In this section we prove that $A_{r}(t;q,\ell)$ and $\widetilde{A}_{r}(t;q,\ell)$ are both strange functions.
%

Theorems~\ref{thm:Ar_is_strange} and~\ref{thm:Ar_is_strange_lambda} are proved by induction over $r$, and the proofs of both are very similar. 
We therefore only give the full details for the proof of the Theorem~\ref{thm:Ar_is_strange} and highlight the differences for the proof of the Theorem~\ref{thm:Ar_is_strange_lambda}.

\begin{color}{black}
The proof is a bit long and thus we will break it down into several smaller parts.
In the first step, in Section~\ref{sec:sum_r_primes}, we derive an explicit expression for $A_r(t)$ involving $A_{r-1}(t)$ and an explicit expression for $\widetilde{A}_{r}(t)$ involving $\widetilde{A}_{r-1}(t)$, see Lemma~\ref{lem:explcit_Ar_as_A{r-1}} and Lemma~\ref{lem:explcit_Ar_as_A{r-1}_lambda}. 
In the second step, in Section~\ref{sec:preparations}, we determine the derivatives of some smooth functions needed in the rest of the proof. 
In the third step, we show in Section~\ref{sec:some-required-strange-functions} that the expressions appearing in the Lemmas~\ref{lem:explcit_Ar_as_A{r-1}} and~\ref{lem:explcit_Ar_as_A{r-1}_lambda} are strange and determine their asymptotic behavior.
Finally, in Section~\ref{sec:proof_Ar_complete} we complete the proof of the Theorems~\ref{thm:Ar_is_strange} and~\ref{thm:Ar_is_strange_lambda} by combining the results of the previous steps.
\end{color}

\subsection{Recurrence expression for $A_{r}(t;q,\ell)$ and $\widetilde{A}_{r}(t;q,\ell)$}
\label{sec:sum_r_primes}
%\subsection{The sum over $r$-primes}
%\label{sec:sum_r_primes}
%We state in this section the main theorems for $A_{r}(t;q,\ell)$ and $\widetilde{A}_{r}(t;q,\ell)$.
%
To simplify the notation of sums as in \eqref{eq:def_A_r},
we will work with a vectorized notation. 
We write for $1\leq s\leq r$ 
\begin{align}
\textbf{p}_s=(p_1,\ldots,p_{s}), \quad
\textbf{n}_s=(n_1,\ldots,n_{s}), \quad
\mathcal{J}(\textbf{p}_s):=\prod_{i=1}^{s} p_i,
\ \text{ and } \
\Lambda(\textbf{n}_s):=\prod_{i=1}^{s} \Lambda(n_i).
\label{eq:vector_notation_bold_p}
\end{align}
Thus we can write $A_{r}(t;q,\ell)$ and $\widetilde{A}_{r}(t;q,\ell)$ in \eqref{eq:def_A_r} as
\begin{align}
A_{r}(t;q,\ell)
&=
\sum_{\substack{\textbf{p}_{r}\in\Pri^{r}\\\mathcal{J}(\textbf{p}_r)\leq t}} \one_{\{\mathcal{J}(\textbf{p}_r) \equiv\ell \bmod q\}}
\ \text{ and }\
\widetilde{A}_{r}(t;q,\ell)
=
\mathop{\sum}_{\substack{\textbf{n}_{r}\in\N^{r}\\\mathcal{J}(\textbf{n}_r)\leq t}} \Lambda(\textbf{n}_r) \one_{\{\mathcal{J}(\textbf{n}_r) \equiv\ell \bmod q\}},
\label{eq:Ar_with_vector}
\end{align}
where $\one_{\{\cdot\}}$ is the indicator function.
Further, we drop the $\textbf{p}_r\in\Pri^r$ and $\textbf{n}_r\in\N^r$ terms in sums appearing in \eqref{eq:Ar_with_vector} and just write 
\begin{align}
A_{r}(t;q,\ell)
&=
\sum_{\substack{\mathcal{J}(\textbf{p}_r)\leq t}} \one_{\{\mathcal{J} \equiv\ell \bmod q\}}
\ \text{ and }\
\widetilde{A}_{r}(t;q,\ell)
=
\mathop{\sum}_{\substack{\mathcal{J}(\textbf{n}_r)\leq t}} \Lambda(\textbf{n}_r) \one_{\{\mathcal{J}(\textbf{n}_r) \equiv\ell \bmod q\}}.
\label{eq:Sa_with_vector2}
\end{align}
The index of the vector $\textbf{p}_r$ and $\textbf{n}_r$ shows over which set we take the sum.
If a symbol with index is not bold, for instance $p_r$ and $n_r$, then we sum always just over $\Pri$ and $\N$.

We begin with the following two lemmas.
\begin{lemma}
\label{lem:explcit_Ar_as_A{r-1}}
Let $r\geq 2$ and suppose that Theorem~\textnormal{\ref{thm:Ar_is_strange}} holds for all $1\leq s\leq r-1$.
Then
\begin{align*}
A_{r}(t;q,\ell)
=&
\frac{\Li(\sqrt{t})A^{(0)}_{r-1}(\sqrt{t})}{\varphi(q)}
-
\sum_{p|q} A_{r-1}^{(0)}(t/p;q)
-
\frac{1}{\varphi(q)}
\sum_{\substack{\mathcal{J}(\textbf{p}_{r-1}) \leq\sqrt{t}\\ (\mathcal{J}(\textbf{p}_{r-1}),q)>1}}
\Li\left(\frac{t}{\mathcal{J}(\textbf{p}_{r-1})}\right)   \nonumber\\
&+
\int_{2}^{\sqrt{t}} \pi(u) \frac{tA_{r-1}^{(1)}(t/u;q)}{u^2}\,du
+
\frac{1}{\varphi(q)}
\int_{2}^{\sqrt{t}} \frac{t A_{r-1}(u)}{u^2 \log(t/u)} \,du
+
O\left(\frac{t}{(\log t)^{C}}\right), 
\end{align*}
where the implicit constant in $O(\cdot)$ can be chosen independently of $q$ and $\ell$.
\end{lemma}

\begin{lemma}
\label{lem:explcit_Ar_as_A{r-1}_lambda}
Let $r\geq 2$ and suppose that Theorem~\textnormal{\ref{thm:Ar_is_strange_lambda}} holds for all $1\leq s\leq r-1$.
Then
\begin{align*}
\widetilde{A}_{r}(t;q,\ell)
=&
\frac{\sqrt{t}\widetilde{A}^{(0)}_{r-1}(\sqrt{t})}{\varphi(q)}
-
\sum_{\substack{n_{r}\leq \sqrt{t}\\(n_{r},q)>1}} \Lambda(n_{r}) \widetilde{A}_{r-1}^{(0)}(t/n_{r};q)
-
\frac{t}{\varphi(q)}
\sum_{\substack{\mathcal{J}(\textbf{n}_{r-1}) \leq\sqrt{t}\\ (\mathcal{J}(\textbf{n}_{r-1}),q)>1}}
\frac{\Lambda(\textbf{n}_{r-1})}{\mathcal{J}(\textbf{n}_{r-1})}   \nonumber\\
&+
\int_{2}^{\sqrt{t}} \psi(u) \frac{t\widetilde{A}_{r-1}^{(1)}(t/u;q)}{u^2}\,du
+
\frac{1}{\varphi(q)}
\int_{2}^{\sqrt{t}} \frac{t \widetilde{A}_{r-1}(u)}{u^2} \,du
+
O\left(\frac{t}{(\log t)^{C}}\right), 
\end{align*}
where the implicit constant in $O(\cdot)$ can be chosen independently of $q$ and $\ell$.
\end{lemma}
Looking at the explicit expression for $A_{r}(t;q,\ell)$, one might be tempted to use
the prime number theorem and \eqref{eq:As_aympt_leading_only_q_and_O:uniform} to replace $\pi(u)$ and $A_{r-1}(u)$
by $\Li(u)$ and $A_{r-1}^{(0)}(u)$ in the integrals.
Unfortunately, this would weaken the error term to $O\left(t(\log t)^{-1}\right)$, which is not strong enough for our purposes, see Lemma~\ref{lem:main_sum_exp_gamma}. Also, inserting the case $r=2$ into \eqref{lem:explcit_Ar_as_A{r-1}} gives 
\begin{align*}
A_{2}(t;q,\ell)
=&
\frac{\big(\Li(\sqrt{t})\big)^2}{\varphi(q)}
-
2\sum_{p|q} \frac{\Li(t/p) }{\varphi(q)}
+
\frac{2}{\varphi(q)}\int_{2}^{\sqrt{t}}  \frac{t \pi(u)}{u^2\log(t/u)}\,du
+
O\left(\frac{t}{(\log t)^{C}}\right).
\end{align*}
Comparing this expression with the expression for $A_2(t)$ in Theorem~\ref{thm:Ugly_A_mod_q} shows that they do not agree, but are of course equivalent. 
The main reason is that Theorem~\ref{thm:Ugly_A_mod_q} was deduced with Mertens theorem while we will use the prime number theorem here.

We give only the proof of Lemma~\ref{lem:explcit_Ar_as_A{r-1}} since the proof of Lemma~\ref{lem:explcit_Ar_as_A{r-1}_lambda} is almost identical and one mainly has to replace $\pi(t;q,\ell)$ by $\psi(t;q,\ell)$.
\begin{proof}[Proof of Lemma~\textnormal{\ref{lem:explcit_Ar_as_A{r-1}}}.]
When using the prime number theorem, the Siegel-Walfisz theorem or \eqref{eq:As_aympt_leading_only_q_and_O:uniform} in this proof, 
we will choose the implicit constant in the $O(\cdot)$ term always independently of $q$ and $\ell$.
Thus the implicit constant in all $O(\cdot)$ terms in this proof are automatically independent of $q$ and $\ell$.   
%We thus will not state this anymore. 
Further, we occasionally choose $C+1$ instead of $C$ in the $O(\cdot)$ term. 
The main purpose is to ensure that the resulting error term has the correct order.
Observe that 
\begin{align}
p_1\cdots p_{r}\equiv\ell\bmod{q} \ \Longrightarrow (p_1\cdots p_{r}, q) =1.
\label{eq_p1..pr_eqi_ell_mod_q_implies:coprime}
\end{align}
Indeed, if some $p_j$ would divide $q$ for some $j$ then this $p_j$ would also divide $\ell$ since $\ell = p_1\cdots p_r +kq$ for some $k\in\Z$. 
This is a contradiction to the assumption $(\ell,q)=1$.

Using the notation $\textbf{p}_r = (\textbf{p}_{r-1},p_r)\in\Pri^r$, 
we can write $A_{r}(t;q,\ell)$ as  
\begin{align}
A_{r}(t;q,\ell)
=
A_{r,1}(t)
+
A_{r,2}(t)
-
A_{r,3}(t)
\end{align}
with
\begin{align}
A_{r,1}(t)
&=
\sum_{\substack{\mathcal{J}(\textbf{p}_r)\leq t\\ p_r\leq \sqrt{t}	}} \one_{\{\mathcal{J}(\textbf{p}_r) \equiv\ell \bmod q\}},
\qquad
A_{r,2}(t)
=
\sum_{\substack{\mathcal{J}(\textbf{p}_r)\leq t\\ \mathcal{J}(\textbf{p}_{r-1}) \leq\sqrt{t}	}}
\one_{\{\mathcal{J}(\textbf{p}_r) \equiv\ell \bmod q\}},\\
A_{r,3}(t)
&=
\sum_{\substack{\mathcal{J}(\textbf{p}_r)\leq t\\ p_r\leq \sqrt{t},\mathcal{J}(\textbf{p}_{r-1}) \leq\sqrt{t}	}}
\one_{\{\mathcal{J}(\textbf{p}_r) \equiv\ell \bmod q\}}.
\end{align}
We first look at $A_{r,3}(t)$.
Equation \eqref{eq_p1..pr_eqi_ell_mod_q_implies:coprime} implies that only those $p_r$ can occur in $A_{r,3}(t)$ which are coprime to $q$.
Denote by $\overline{p_{r}}$ the multiplicative inverse of $p_r$ modulo $q$.
Then
\begin{align}
A_{r,3}(t)
&=
\sum_{\substack{p_{r}\leq \sqrt{t}\\(p_{r},q)=1}}
\sum_{\substack{\mathcal{J}(\textbf{p}_{r-1}) \leq\sqrt{t}}} \one_{\{\mathcal{J}(\textbf{p}_{r-1}) \equiv\overline{p_{r}}\ell \bmod q\}}
=
\sum_{\substack{p_{r}\leq \sqrt{t}\\(p_{r},q)=1}}
A_{r-1}(\sqrt{t};q, \overline{p_{r}}\ell).
\end{align}

By assumption, $A_{r-1}(\sqrt{t};q, \overline{p_{r}}\ell)$ is strange.
Further, we can write
\begin{align*}
A_{r-1}(t;q, \overline{p_{r}}\ell)
=
A_{r-1}^{(0)}(t;q) + O(\sqrt{t}(\log\sqrt{t})^{-C-1}),
\end{align*}
where $A_{r-1}^{(0)}(t;q)$ is differentiable and independent of $\overline{p_{r}}\ell$.
The prime number theorem, \eqref{eq:find_a_cool_name_0_1} and $q \leq (\log t)^A$ then give
\begin{align}
A_{r,3}(t)
&=
\left(A_{r-1}^{(0)}(\sqrt{t},q) +O\left(\sqrt{t}(\log\sqrt{t})^{-C}\right)\right)\Big(\sum_{\substack{p_{r}\leq \sqrt{t}\\(p_{r},q)=1}} 1\Big)
\nonumber\\
%&=
%\left(A_{r-1}^{(0)}(\sqrt{t}) +O\left(\sqrt{t}(\log\sqrt{t})^{-C-1}\right)\right)
%\left(\Li(\sqrt{t})+O\left(\sqrt{t}(\log\sqrt{t})^{-C-1}\right) +O(\card\{p;\,p|q\}) \right)
%\nonumber\\
&=
\left(A_{r-1}^{(0)}(\sqrt{t},q) +O\left(\sqrt{t}(\log\sqrt{t})^{-C}\right)\right)
\left(\Li(\sqrt{t})+O(\sqrt{t}(\log\sqrt{t})^{-C})  \right)
\nonumber\\
&=
\Li(\sqrt{t}) A_{r-1}^{(0)}(\sqrt{t},q) +O\left(\frac{t}{(\log\sqrt{t})^{C+1}}\right).
\label{eq:proof_Ar_3_with_O}
\end{align}
%	
%In particular, the implicit constant in $O(\cdot)$ in \eqref{eq:proof_Ar_3_with_O} 
%can be chosen independent of $q$, $\ell$. 

Next, we look at $A_{r,1}(t)$.
Using the definition and assumptions on $A_{r-1}(t;q,\ell)$, we obtain
\begin{align*}
A_{r,1}(t)
&=
\sum_{\substack{p_{r}\leq \sqrt{t}\\(p_{r},q)=1}}
\sum_{\substack{\mathcal{J}(\textbf{p}_{r-1}) \leq t/p_r}}  \one_{\{\mathcal{J}(\textbf{p}_{r-1}) \equiv\overline{p_{r}}\ell \bmod q\}}
=
\sum_{\substack{p_{r}\leq \sqrt{t}\\(p_{r},q)=1}} A_{r-1}(t/p_{r};q,\overline{p_{r}}\ell)\\
&=
\sum_{\substack{p_{r}\leq \sqrt{t}\\(p_{r},q)=1}} \left(A_{r-1}^{(0)}(t/p_{r};q) 
+ O\left(\frac{(t/p_{r})}{(\log (t/p_{r}))^{C+1}}\right)\right).
\end{align*}
As above $A_{r-1}^{(0)}(t/p_{r};q)$ is independent of $\overline{p_{r}}\ell$.
Since $\log(\sqrt{t})\leq \log(t/p_r)\leq \log t$, we see that 
\begin{align*}
\sum_{\substack{p_{r}\leq \sqrt{t}\\(p_{r},q)=1}}  \frac{t/p_{r}}{(\log (t/p_{r}))^{C+1}}
\ll
\frac{t}{(\log \sqrt{t})^{C+1}}\sum_{\substack{p\leq \sqrt{t}}} \frac{1}{p}
\ll
\frac{t\log\log t}{(\log t)^{C+1}}
\ll
\frac{t}{(\log t)^{C}}.
\end{align*}
Inserting this estimate in the above expression for $A_{r,1}(t)$, we are led to
\begin{align}
A_{r,1}(t)
&=
\sum_{\substack{p_{r}\leq \sqrt{t}\\(p_{r},q)=1}} A_{r-1}^{(0)}(t/p_{r};q) + O\left(\frac{t}{(\log t)^{C}}\right)
\nonumber\\
&=
\sum_{p\leq \sqrt{t}} A_{r-1}^{(0)}(t/p;q) 
-
\sum_{p|q} A_{r-1}^{(0)}(t/p;q)
+
O\left(\frac{t}{(\log t)^{C}}\right).
\label{eq:A_r+1,1_after_induc_before_Abel}
\end{align}
Only the first sum in \eqref{eq:A_r+1,1_after_induc_before_Abel} requires further attention.
Since $A_{r-1}^{(0)}(t/p;q)$ is a differentiable function, we apply Abel's summation formula.
We use the sequence $\{b_n\}_{n=1}^\infty$ with $b_n =1$ if $n$ is prime and $b_n=0$ otherwise. 
Then $\sum_{n\leq t} b_n =\pi(t)$ and
\begin{align}
\sum_{p\leq \sqrt{t}} A_{r-1}^{(0)}(t/p;q) 
&=
\sum_{n\leq \sqrt{t}} b_n  A_{r-1}^{(0)}(t/n;q)\nonumber\\
&=
\pi(\sqrt{t}) A_{r-1}^{(0)}(\sqrt{t};q)
+
\int_{2}^{\sqrt{t}} \pi(u) \frac{t(A_{r-1}^{(0)})'(t/u;q)}{u^2}\,du.
\label{eq:proof_non:principal_A_r+1,1}
\end{align}
Equation \eqref{eq:find_a_cool_name_0_1} and the prime number theorem imply
\begin{align}
\pi(\sqrt{t}) A_{r-1}^{(0)}(\sqrt{t};q)
=
\Li(\sqrt{t})A_{r-1}^{(0)}(\sqrt{t};q)+ O\left(\frac{t}{(\log t)^{C}}\right).
\label{eq:A1:1_pi_part}
\end{align}
Inserting that $A_{r-1}(t;q)$ is strange into the integral in \eqref{eq:proof_non:principal_A_r+1,1} gives
\begin{align}
\int_{2}^{\sqrt{t}} \pi(u) \frac{t(A_{r-1}^{(0)})'(t/u;q)}{u^2}\,du
&=
\int_{2}^{\sqrt{t}} \pi(u) \frac{tA_{r-1}^{(1)}(t/u;q)}{u^2}\,du
+
O\bigg(\int_{2}^{\sqrt{t}}  \frac{t \pi(u)}{u^2 (\log(t/u))^{C+1}}\,du \bigg)\nonumber\\
&=
\int_{2}^{\sqrt{t}} \pi(u) \frac{tA_{r-1}^{(1)}(t/u;q)}{u^2}\,du
+
O\left(\frac{t}{(\log t)^{C}} \right).
\label{eq:A1:1_int_part}
\end{align}
Combining everything, i.e. inserting  \eqref{eq:A1:1_pi_part} and \eqref{eq:A1:1_int_part} into \eqref{eq:A_r+1,1_after_induc_before_Abel}, we get
\begin{align}
A_{r,1}(t)
=\,&
\Li(\sqrt{t})A_{r-1}^{(0)}(\sqrt{t};q)
-
\sum_{p|q} A_{r-1}^{(0)}(\tfrac{t}{q};q) +
\int_{2}^{\sqrt{t}} \pi(u) \frac{tA_{r-1}^{(1)}(\tfrac{t}{u};q)}{u^2}\,du
+
O\left(\frac{t}{(\log t)^{C}}\right).
\label{eq:Ar_1_final}
\end{align}

It remains to look at $A_{r,2}(t)$. 
We denote by $\overline{\mathcal{J}(\textbf{p}_{r-1})}$  the multiplicative inverse of $\mathcal{J}(\textbf{p}_{r-1})$ modulo $q$.
Using the definition of $\pi(t;q,\ell)$ and the Siegel-Walfisz theorem, we get
\begin{align*} 
A_{r,2}(t)
&=
\sum_{\substack{\mathcal{J}(\textbf{p}_{r-1}) \leq\sqrt{t}\\ (\mathcal{J}(\textbf{p}_{r-1}),q)=1}}
\bigg(\sum_{p_r\leq \sqrt{t}}
\one_{\{p_r \equiv\overline{\mathcal{J}(\textbf{p}_{r-1})}\ell \bmod q\}}
\bigg)
=
\sum_{\substack{\mathcal{J}(\textbf{p}_{r-1}) \leq\sqrt{t}\\ (\mathcal{J}(\textbf{p}_{r-1}),q)=1}}
\pi(\sqrt{t};q,\overline{\mathcal{J}(\textbf{p}_{r-1})}\ell)
\\
&=
\sum_{\substack{\mathcal{J}(\textbf{p}_{r-1}) \leq\sqrt{t}\\ (\mathcal{J}(\textbf{p}_{r-1}),q)=1}}
\bigg(\frac{1}{\varphi(q)} \Li\left(\frac{t}{\mathcal{J}(\textbf{p}_{r-1})}\right) + O\bigg(\frac{t}{\mathcal{J}(\textbf{p}_{r-1}) \big(\log(t/\mathcal{J}(\textbf{p}_{r-1}))\big)^{C+1}}\bigg)\bigg).
\end{align*}
Since $\mathcal{J}(\textbf{p}_{r-1})\leq \sqrt{t}$, we have $\log(t/\mathcal{J}(\textbf{p}_{r-1}))\geq \frac{1}{2}\log t$ and thus
\begin{align*}
\sum_{\substack{\mathcal{J}(\textbf{p}_{r-1}) \leq\sqrt{t}\\ (\mathcal{J}(\textbf{p}_{r-1}),q)=1}}
\frac{t}{\mathcal{J}(\textbf{p}_{r-1}) (\log(t/\mathcal{J}(\textbf{p}_{r-1})))^{C+1}}
&\ll
\frac{t}{(\log t)^{C+1}}
\sum_{\substack{p_1,\ldots,p_{r-1}\leq \sqrt{t}}} 
\frac{1}{p_1\cdots p_{r-1}}\\
&\ll
\frac{t}{(\log t)^{C+1}}
\bigg(\sum_{p\leq\sqrt{t}} \frac{1}{p}\bigg)^r
\ll
\frac{t}{(\log t)^{C}}.
\end{align*}
Thus
\begin{align}
A_{r,2}(t)
&=
\sum_{\substack{\mathcal{J}(\textbf{p}_{r-1}) \leq\sqrt{t}\\ (\mathcal{J}(\textbf{p}_{r-1}),q)=1}}
\frac{1}{\varphi(q)}\Li\left(\frac{t}{\mathcal{J}(\textbf{p}_{r-1})}\right)  
+ O\left(\frac{t}{(\log t)^{C}}\right)\nonumber\\
&=
\frac{1}{\varphi(q)}
\sum_{\substack{\mathcal{J}(\textbf{p}_{r-1}) \leq\sqrt{t}}}
\Li\left(\frac{t}{\mathcal{J}(\textbf{p}_{r-1})}\right)  
-
\frac{1}{\varphi(q)}
\sum_{\substack{\mathcal{J}(\textbf{p}_{r-1}) \leq\sqrt{t}\\ (\mathcal{J}(\textbf{p}_{r-1}),q)>1}}
\Li\left(\frac{t}{\mathcal{J}(\textbf{p}_{r-1})}\right)  
+ 
O\left(\frac{t}{(\log t)^{C}}\right).
\end{align}
Thus the remaining task is to examine the first sum.
For this we define, for $n\in\N$,
\begin{align}
b_n := \card\{(p_1,\ldots,p_{r-1}) \,;\, \mathcal{J}(\textbf{p}_{r-1})=n\}.
\end{align}
Then $\sum_{n\leq x} b_n
=
\sum_{\substack{\mathcal{J}(\textbf{p}_{r-1}) \leq x}} 
1
=
A_{r-1}(x)$.
Abel summation then implies 
\begin{align}
\frac{1}{\varphi(q)}\sum_{\substack{\mathcal{J}(\textbf{p}_{r-1}) \leq\sqrt{t}}}
\Li\left(\frac{t}{\mathcal{J}(\textbf{p}_{r-1})}\right)
&=
\frac{1}{\varphi(q)}\sum_{n\leq \sqrt{t}} b_n \Li(t/n)\nonumber\\
&=
\frac{\Li(\sqrt{t})A_{r-1}(\sqrt{t})}{\varphi(q)}
+
\frac{1}{\varphi(q)}
\int_{2}^{\sqrt{t}} \frac{t A_{r-1}(u)}{u^2 \log(t/u)} \,du.
\label{eq:almost_done}
\end{align}
Inserting \eqref{eq:almost_done} into the above expression for $A_{r,2}(t)$ and using that $A_{r-1}(t)$ is strange gives
\begin{align}
A_{r,2}(t)
=\,&
\frac{\Li(\sqrt{t})A^{(0)}_{r-1}(\sqrt{t})}{\varphi(q)}
+
\frac{1}{\varphi(q)}
\int_{2}^{\sqrt{t}} \frac{t A_{r-1}(u)}{u^2 \log(t/u)} \,du\nonumber\\
&-
\frac{1}{\varphi(q)}
\sum_{\substack{\mathcal{J}(\textbf{p}_{r-1}) \leq\sqrt{t}\\ (\mathcal{J}(\textbf{p}_{r-1}),q)>1}}
\Li\left(\frac{t}{\mathcal{J}(\textbf{p}_{r-1})}\right)  
+ 
O\left(\frac{t}{(\log t)^{C}}\right).
\end{align}

Combining the expressions for $A_{r,1}(t)$, $A_{r,2}(t)$ and $A_{r,3}(t)$ in \eqref{eq:Ar_1_final}, \eqref{eq:almost_done} and \eqref{eq:proof_Ar_3_with_O} completes the last step of the proof.
\end{proof}

\subsection{Derivatives of some smooth functions}
\label{sec:preparations}
\begin{color}{black}
In this section we study the properties of functions involving  $t^{r} (\log\log t)^a (\log t)^{-b}$ and their derivatives.
\end{color}
\begin{lemma}
\label{lem:derivatives_loglog^a_log^b}
Let  $a\in\N_0$ and $b\in\Z$  be given and $P_{a}(x)$ be a polynomial of degree $a$.
Further, let
\begin{align}
D(t) 
&= 
\frac{t P_{a}(\log\log t)}{(\log t)^b}.
\end{align}
Then, as $t\to\infty$, 
\begin{align}
D'(t) 
&= 
\frac{P_{a}(\log\log t)}{(\log t)^b} +  O\left(\frac{(\log\log t)^{a}}{(\log t)^{b+1}}\right),
\label{eq_diff_D_first}
\end{align}
and for $n\geq 2$ 
\begin{align}
D^{(n)}(t) 
&= 
\frac{(n-2)!(-1)^{n}(-bP_{a}+P_a')(\log\log t)}{t^{n-1}(\log t)^{b+1}} 
+  O\left(\frac{(\log\log t)^{a}}{t^{n-1}(\log t)^{b+2}}\right).
\label{eq_diff_D_third}
\end{align}
\end{lemma}
\begin{proof}
If $a=b=0$ then $D_t(t)=t$ and the lemma clearly holds in this case.
We thus can assume that $(a,b)\neq (0,0)$.
We have for $r\in\R$
\begin{align}
\frac{d}{d t} \frac{t^{r} (\log\log t)^a}{(\log t)^{b}}
=
\frac{t^{r-1} (\log\log t)^a}{(\log t)^{b}} \big(r  -b(\log t)^{-1}+ a(\log(t)\log\log t)^{-1}  \big).
\label{eq:derivatives_loglog^a_log^b0}
\end{align}
Therefore the leading term in \eqref{eq:derivatives_loglog^a_log^b0} is different in the cases $r=0$ and $r\neq 0$.
Inserting $r =1$ into \eqref{eq:derivatives_loglog^a_log^b0} implies
\begin{align*}
\frac{d}{d t} \frac{t (\log\log t)^a}{(\log t)^{b}}
=
\frac{ (\log\log t)^a}{(\log t)^{b}} \big( 1  -b(\log t)^{-1}+ a(\log (t)\log\log t)^{-1} \big).
\end{align*}
This completes the proof of \eqref{eq_diff_D_first}.
We next look at \eqref{eq_diff_D_third} in the case $n=2$. 
Inserting $r =0$ in \eqref{eq:derivatives_loglog^a_log^b0} gives
\begin{align}
\frac{d}{d t} 	\frac{ (\log\log t)^a}{(\log t)^{b}}
=
t^{-1} 	\frac{ (\log\log t)^a}{(\log t)^{b+1}} \big(-b+ a(\log\log t)^{-1}  \big).
\end{align}
The cases $n\geq 3$ then follow by induction.
\end{proof}
Furthermore, we need
\begin{lemma}
\label{lem:prod_Li_D_sqrt_growth}
Let  $D(t)$ be as in Lemma~\textnormal{\ref{lem:derivatives_loglog^a_log^b}}.
We then have for $n\in\N_0$ that
\begin{align}
\big(D(\sqrt{t})\big)^{(n)}
&=
\frac{\Gamma(\frac{3}{2})}{\Gamma(\frac{3}{2}-n)} \frac{P_{a}(\log\log \sqrt{t} )}{t^{n-1/2}(\log\sqrt{t})^b}\left(1+ O\left((\log t)^{-1}\right)\right),
\label{eq_diff_D_sqrt}\\
(\Li(\sqrt{t}))^{(n)}
&=
\frac{\Gamma(\frac{3}{2})}{\Gamma(\frac{3}{2}-n)} \frac{1}{t^{n-1/2}\log\sqrt{t}} \left(1+ O\left((\log t)^{-1}\right)\right),
\label{eq:diff_t_Li_sqrt}\\
\left(\frac{\sqrt{t}}{(\log \sqrt{t})^C}\right)^{(n)}
&=
\frac{\Gamma(\frac{3}{2})}{\Gamma(\frac{3}{2}-n)}\frac{1}{t^{n-1/2} (\log \sqrt{t})^C} \left(1+ O\left((\log t)^{-1}\right)\right).
\label{eq:diff_t_log_t_C_sqrt}
\end{align} 
\end{lemma}
\begin{proof}
For the proof of \eqref{eq_diff_D_sqrt}, \eqref{eq:diff_t_Li_sqrt} and \eqref{eq:diff_t_log_t_C_sqrt} we require
\begin{align}
	\frac{d}{d t} \frac{t^{r+\frac{1}{2}} (\log\log \sqrt{t})^a}{(\log \sqrt{t} )^{b}}
	=
	\frac{t^{r-\frac{1}{2}} (\log\log \sqrt{t})^a}{(\log \sqrt{t} )^{b}} \Big(r+\frac{1}{2} +O\left( (\log t)^{-1}\right)  \Big)
\end{align}
and the observation 
\begin{align*}
	\Gamma\left(\frac{3}{2}\right)
	=
	\frac{1}{2}\cdot\Gamma\left(\frac{1}{2}\right)
	=
	\frac{1}{2}\cdot\left(-\frac{1}{2}\right)\cdot \Gamma\left(-\frac{1}{2}\right)
	=
	\frac{1}{2}\cdot\left(-\frac{1}{2}\right)\cdot\left(-\frac{3}{2}\right)\cdot\Gamma\left(-\frac{1}{2}\right)=\ldots.
\end{align*}
The proof of the remaining formulas follows by a combination of induction with the equations \eqref{eq_diff_D_sqrt}, \eqref{eq:diff_t_Li_sqrt} and \eqref{eq:diff_t_log_t_C_sqrt}.
This computation is straightforward and we thus omit the details.
\end{proof}
An immediate consequence of this lemma is the identity 
\begin{align}
\sum_{k=0}^n \binom{n}{k}\frac{\Gamma\left(\frac{3}{2}\right)}{\Gamma\left(\frac{3}{2}-k\right)}\frac{\Gamma\left(\frac{3}{2}\right)}{\Gamma\left(\frac{3}{2}-(n-k)\right)}
=
\begin{cases}
	1 &\text{if }n=0 \text{ or }n=1,\\
	0 &n\geq 2.
\end{cases}
\label{eq:identity_from_sqrt}
\end{align}
Indeed, we have 
\begin{align}
	\frac{4t}{(\log t)^2}
	=
	\frac{\sqrt{t}}{\log \sqrt{t} }\frac{\sqrt{t}}{\log \sqrt{t} }.
\end{align}
Computing the $n$th derivative and inserting \eqref{eq:diff_t_log_t_C_sqrt} immediately gives \eqref{eq:identity_from_sqrt}.

\subsection{Some required strange functions}
\label{sec:some-required-strange-functions}
The purpose of this subsection is to study the functions occurring in Lemma~\ref{lem:explcit_Ar_as_A{r-1}} and in Lemma~\ref{lem:explcit_Ar_as_A{r-1}_lambda}.
Our objective is twofold: first, to demonstrate that all of these functions are strange, and second, to establish necessary asymptotics for them.
All strange functions in this subsection originate from $\pi(t;q,a)$ or $\psi(t;q,a)$ and have a leading term as in Lemma~\ref{lem:derivatives_loglog^a_log^b}.
%Instead to work directly with $A_r(t)$, we work here with a slight generalisation. 
We therefore assume for the rest of this section that $D$ is a strange function with respect to $t(\log t )^{-C}$ 
for all $C\geq C_0$ for some $C_0\geq 0$.
Further, we assume there exists $a\in\N$, $	b\in\Z$ and for each $n\in\N_0$ a polynomial $P_{a,n}$ of degree $a-1$ such that, as $t\to\infty$, 
\begin{align}
D(t) 
&= 
\frac{t P_{a,0}(\log\log t)}{(\log t)^b} + O\left(\frac{t(\log\log t)^{a-1}}{(\log t)^{b+1}}\right),
\label{eq:assumption_diff_D_zero}\\
D^{(1)}(t) 
&= 
\frac{P_{a,1}(\log\log t)}{(\log t)^b} +  O\left(\frac{(\log\log t)^{a-1}}{(\log t)^{b+1}}\right),
\label{eq:assumption_diff_D_first}
\intertext{and for $n\geq 2$ }
D^{(n)}(t) 
&= 
\frac{P_{a,n}(\log\log t)}{t^{n-1}(\log t)^{b+1}} +  O\left(\frac{(\log\log t)^{a-1}}{t^{n-1}(\log t)^{b+2}}\right).
\label{eq:assumption_diff_D_third}
\end{align}
The assumptions above on $D$ are more general than in Theorem~\ref{thm:Ar_is_strange} and in Theorem~\ref{thm:Ar_is_strange_lambda}.
The reason is that almost all of the results in this subsection also hold for a leading term as in Lemma~\ref{lem:derivatives_loglog^a_log^b}.
\begin{lemma}
\label{lem:D_sqrt_strange}
\textcolor{black}{Let $D$ satisfy the assumptions at the beginning of Section~\textnormal{\ref{sec:some-required-strange-functions}}}. Then $D(\sqrt{t})$ is strange with respect to $\sqrt{t}(\log \sqrt{t})^{-C}$.
Further, we have for all $n\geq 1$
\begin{align}
\big(D(\sqrt{t})\big)^{(n)}
=
\frac{\Gamma(\frac{3}{2})}{\Gamma(\frac{3}{2}-n)}\frac{P_{a,1}(\log\log \sqrt{t} )}{(\log\sqrt{t})^b}   t^{\frac{1}{2}-n}
+
O\bigg(\frac{(\log\log t)^{a-1}}{t^{\frac{n-1}{2}}(\log\sqrt{t})^{b+1}}\bigg).
\label{eq:lem:D_sqrt_strange_asympt}
\end{align}
\end{lemma}

\begin{proof}
To show that $D(\sqrt{t})$ is strange, we apply the chain rule with $g(t)=t(\log t )^{-C}$ and $\varphi(t) = \sqrt{t}$.
Let $n\geq k\geq 1$ and $(m_j)_{j=1}^{n-k+1}$ be given, which satisfy the condition \eqref{eq:sum_k_sum_n_chain_rule}.
Observe that for $j\in\N_0$
\begin{align}
( t^{\frac{1}{2}})^{(j)}
=
\frac{\Gamma(\frac{3}{2})}{\Gamma(\frac{3}{2}-j)} t^{\frac{1}{2}-j},	
\ \text{ and } \ 
g^{(k)}(\varphi(t))
\ll
\begin{cases}
(\log t)^{-C} &\text{if }k=1,\\
t^{\frac{1-k}{2}}(\log t)^{-C-1} &\text{if }k\geq 2.
\end{cases}
\end{align}
Since $k=\sum_{j=1}^{n-k+1} m_j$ and $n=\sum_{j=1}^{n-k+1} jm_j$, we get
\begin{align*}
\prod_{j=1}^{n-k+1}(\varphi^{(j)}(t))^{m_j}
&\ll
\prod_{j=1}^{n-k+1} (t^{\frac{1}{2}-j})^{m_j}
=
t^{\frac{k}{2}-n}.
\end{align*}
Combining these computations gives	
\begin{align}
g^{(k)}(\varphi(t)) \prod_{j=1}^{n-k+1}(\varphi^{(j)}(t))^{m_j}
\ll
\begin{cases}
t^{\frac{1}{2}-n}(\log t)^{-C} &\text{if }k=1,\\
t^{\frac{1}{2}-n}(\log t)^{-C-1} &\text{if }k\geq 2.
\end{cases}
\end{align}
Furthermore, \eqref{eq:diff_t_log_t_C_sqrt} gives 
\begin{align*}
\big(g(\varphi(t))\big)^{(n)}
=
\left(\frac{\sqrt{t}}{(\log \sqrt{t})^C}\right)^{(n)}
&\ll
t^{1/2-n} (\log t)^{-C}.
%	\frac{\Gamma(\frac{3}{2})}{\Gamma(\frac{3}{2}-n)}\frac{1}{t^{n-1/2} (\log \sqrt{t})^C}.
\end{align*}
Thus condition \eqref{eq:chain_rule_assumption_on_g} is fulfilled and 
$D(\sqrt{t})$ is strange with respect to $\sqrt{t}(\log \sqrt{t})^{-C}$.
It remains to show \eqref{eq:lem:D_sqrt_strange_asympt}.
We have 
\begin{align*}
B_{n,1}\big(\varphi^{(1)}(t),\varphi^{(2)}(t),\ldots,\varphi^{(n)}(t)\big) 
&= 
\varphi^{(n)}(t)
\ \text{ and }\\
B_{n,k}\big(\varphi^{(1)}(t),\varphi^{(2)}(t),\ldots,\varphi^{(n-k+1)}(t)\big)
&\ll
\prod_{j=1}^{n-k+1}(\varphi^{(j)}(t))^{m_j}
\ll
t^{\frac{k}{2}-n}.
\end{align*}
Thus combining these two identities with the chain rule \eqref{eq:chain_rule_starnge_nth_derivative} and the assumptions on $D$ in \eqref{eq:assumption_diff_D_third} yields the following expression
\begin{align*}
\big(D(\sqrt{t})\big)^{(n)}
&=
\sum_{k=1}^{n}
D^{(k)}(\sqrt{t}) B_{n,k}\Big(\varphi^{(1)}(t),\varphi^{(2)}(t),\ldots,\varphi^{(n-k+1)}(t)\Big)\\
&=
D^{(1)}(\sqrt{t}) \varphi^{(n)}(t)
+
\sum_{k=2}^{n} O\bigg(\frac{(\log\log \sqrt{t} )^{a-1}}{t^{\frac{k-1}{2	}}(\log \sqrt{t})^{b+1}}\bigg)O(t^{\frac{k}{2}-n})\\
&=
\frac{\Gamma(\frac{3}{2})}{\Gamma(\frac{3}{2}-n)}\frac{P_{a,1}(\log\log \sqrt{t} )}{(\log\sqrt{t})^b}   t^{\frac{1}{2}-n}
+
O\bigg(\frac{(\log\log t)^{a-1}}{t^{n-\frac{1}{2}}(\log \sqrt{t})^{b+1}}\bigg),
\end{align*}
which was the last step of the proof.
\end{proof}
\begin{lemma}
\label{lem:D_Li_sqrt_strange}
\textcolor{black}{Let $D$ satisfy the assumptions at the beginning of Section~\textnormal{\ref{sec:some-required-strange-functions}}}.
Then $\Li(\sqrt{t})D(\sqrt{t})$ is strange with respect to $\frac{t}{(\log t)^{C}}$
and
$\sqrt{t}D(\sqrt{t})$ is strange with respect to $\frac{t}{(\log t)^{C-1}}$. 
Further, for $n\in\N_0$
\begin{align}
\big(\Li(\sqrt{t})D(\sqrt{t})\big)^{(n)}
=&\,
\frac{\Gamma(\frac{3}{2})}{\Gamma(\frac{3}{2}-n))} \frac{P_{a,0}(\log\log \sqrt{t} )- P_{a,1}(\log\log \sqrt{t} )}{t^{n-1}(\log \sqrt{t})^{b+1}}\nonumber\\
&+
\one_{\{0,1\}}(n) \frac{P_{a,1}(\log\log \sqrt{t} )}{t^{n-1}(\log \sqrt{t})^{b+1}}
+
O\left(\frac{(\log\log t)^{a-1}}{t^{n-1}(\log t)^{b+2}}\right),
\label{eq:D_Li_sqrt_strange_derivatives}\\
\big(\sqrt{t}D(\sqrt{t})\big)^{(n)}
=&\,
\frac{\Gamma(\frac{3}{2})}{\Gamma(\frac{3}{2}-n))} \frac{P_{a,0}(\log\log \sqrt{t} )- P_{a,1}(\log\log \sqrt{t} )}{t^{n-1}(\log\sqrt{t})^{b}}\nonumber\\
&+
\one_{\{0,1\}}(n) \frac{P_{a,1}(\log\log \sqrt{t} )}{t^{n-1}(\log\sqrt{t})^{b}}
+
O\left(\frac{(\log\log t)^{a-1}}{t^{n-1}(\log t)^{b+1}}\right),
\label{eq:D_Li_sqrt_strange_derivatives_lambda}
\end{align}
where $\one_{\{0,1\}}(n)$ is the indicator function, which is equal to $1$ if $n=0$ or $n=1$ and equal to $0$ otherwise.
\end{lemma}
\begin{proof}
Lemma~\ref{lem:D_sqrt_strange} shows that $D(\sqrt{t})$ is strange with respect to $\frac{\sqrt{t}}{(\log \sqrt{t})^C}$. 
Using the derivatives of $\Li(\sqrt{t})$ and $\sqrt{t}(\log \sqrt{t})^{-C}$ in \eqref{eq:diff_t_Li_sqrt} and \eqref{eq:diff_t_log_t_C_sqrt}, we get for $k,\ell\in\N_0$ that
\begin{align*}
(\Li(\sqrt{t}))^{(k)}\bigg(\frac{\sqrt{t}}{(\log \sqrt{t})^C}\bigg)^{(\ell)}
&\ll
\frac{1}{t^{k+\ell-1}(\log t)^{C+1}},
\\
(\sqrt{t})^{(k)}\bigg(\frac{\sqrt{t}}{(\log \sqrt{t})^C}\bigg)^{(\ell)}
&\ll
\frac{1}{t^{k+\ell-1}(\log t)^{C}}.
\end{align*}
Thus the product rule and the derivatives of $\frac{t}{(\log t)^{C}}$ in \eqref{eq:derivatives_tlogt_C} 
imply that $\Li(\sqrt{t})D(\sqrt{t})$ is strange with respect to $\frac{t}{(\log t)^{C}}$
and $\sqrt{t}D(\sqrt{t})$ is strange with respect to $\frac{t}{(\log t)^{C-1}}$.
Furthermore, combining the product rule \eqref{eq:strange_product_rule}, the derivatives of $\Li(\sqrt{t})$ in \eqref{eq:diff_t_Li_sqrt}
and the derivatives of $D(\sqrt{t})$ in \eqref{eq:lem:D_sqrt_strange_asympt} imply that
\begin{align*}
\big(\Li(\sqrt{t})D(\sqrt{t})\big)^{(n)}
=\,&
\sum_{k=0}^{n}\binom{n}{k}\big(D(\sqrt{t})\big)^{(k)}\big(\Li(\sqrt{t})\big)^{(n-k)} \\
=\,&
\frac{P_{a,1}(\log\log \sqrt{t} )}{t^{n-1}(\log \sqrt{t})^{b+1}}
\sum_{k=1}^{n}\binom{n}{k}
\frac{\Gamma(\frac{3}{2})}{\Gamma(\frac{3}{2}-k)}\frac{\Gamma(\frac{3}{2})}{\Gamma(\frac{3}{2}-(n-k))}\\
&+
\frac{\Gamma(\frac{3}{2})}{\Gamma(\frac{3}{2}-n))} \frac{P_{a,0}(\log\log \sqrt{t} )}{t^{n-1}(\log \sqrt{t})^{b+1}}
+
O\bigg(\frac{1}{t^{\frac{n-1}{2}}(\log \sqrt{t})^{b+1}}\bigg).
\end{align*}
Using identity \eqref{eq:identity_from_sqrt}, we get
\begin{align*}
\sum_{k=1}^{n}\binom{n}{k}
\frac{\Gamma(\frac{3}{2})}{\Gamma(\frac{3}{2}-k)}\frac{\Gamma(\frac{3}{2})}{\Gamma(\frac{3}{2}-(n-k))}
&=
\one_{\{0,1\}}(n) -\frac{\Gamma(\frac{3}{2})}{\Gamma(\frac{3}{2}-n)}.
\end{align*}
Combining the last two equations completes the proof of \eqref{eq:D_Li_sqrt_strange_derivatives}.
The proof of \eqref{eq:D_Li_sqrt_strange_derivatives_lambda} is almost identical to the proof  of \eqref{eq:D_Li_sqrt_strange_derivatives} and we thus omit it.
\end{proof}

Thus Lemma~\ref{lem:D_Li_sqrt_strange} shows that the first term in Lemma~\ref{lem:explcit_Ar_as_A{r-1}} and in Lemma~\ref{lem:explcit_Ar_as_A{r-1}_lambda} is strange.
The next step is to show that the integrals are strange as well. 
For this wee need some combinations of $\Li$ and the pseudo-derivatives of $D$.
\begin{lemma}
\label{lem:D_sqrt_strange_derivative}
\textcolor{black}{Let $D$ satisfy the assumptions at the beginning of Section~\textnormal{\ref{sec:some-required-strange-functions}}}. Then
\begin{itemize}
\item $D^{(1)}(\sqrt{t})$ is strange with respect to $(\log \sqrt{t})^{-C}$, and
\item $D^{(j)}(\sqrt{t})$ is strange with respect to $t^{\frac{1-j}{2}}(\log \sqrt{t})^{-C-1}$ for $j\geq 2$.
\item Also, we have for $n,j\geq 1$
\begin{align}
\big(D^{(j)}(\sqrt{t})\big)^{(n)}
\ll
\frac{(\log\log \sqrt{t} )^{a-1}}{t^{n+\frac{j-1}{2}}(\log \sqrt{t})^{b+1}}.
\label{eq:lem:D_sqrt_strange_derivative_asympt}
\end{align}
\end{itemize}
\end{lemma}

\begin{proof}
Since $D$ be strange with respect to $t(\log t )^{-C}$,
we get from Definition~\ref{def:strange_function} that $D^{(j)}(t)$ is strange with respect to 
$(t(\log t)^{-C})^{(j)}$ with $j\in\N_0$ arbitrary.
Thus 
\begin{itemize}
\item $D^{(1)}(t)$ is strange with respect to $(\log t)^{-C}$, and 
\item $D^{(j)}(t)$ is strange with respect to $t^{1-j}(\log t)^{-C-1}$ for $j\geq 2$.
\end{itemize}
The lemma now follows from the chain rule with $g(t)=\left(t(\log t )^{-C}\right)^{(j)}$ and $\varphi(t) = \sqrt{t}$.
Observe that for $k\in\N$ and $j\in\N$
\begin{align*}
g^{(k)}(\varphi(t))
\ll
t^{\frac{1-k-j}{2}}(\log t)^{-C-1}.
\end{align*}
The remaining computations are (almost) the same as in the proof of Lemma~\ref{lem:D_sqrt_strange} and we thus omit them.
\end{proof}
\begin{lemma}
\label{lem:products_with_sqrt_strange}
\textcolor{black}{Let $D$ satisfy the assumptions at the beginning of Section~\textnormal{\ref{sec:some-required-strange-functions}}}. Then
\begin{itemize}
\item 
$t^{-\frac{1}{2}}\Li(\sqrt{t})D^{(1)}(\sqrt{t})$ is strange with respect to  $(\log t)^{-C}$,
\item 
$t^{-\frac{j}{2}}\Li(\sqrt{t})D^{(j)}(\sqrt{t})$ is strange with respect to  $t^{1-j}(\log t)^{-C-1}$ for all $j\geq 2$,
\item 
$t^{-\frac{j+1}{2}}\Li(\sqrt{t})D^{(j-1)}(\sqrt{t})$ is strange with respect to  $t^{1-j}(\log t)^{-C-1}$ for all $j\geq 2$.
%	\item 
%	$t^{-\frac{j}{2}-1}\Li(\sqrt{t})D^{(j)}(\sqrt{t})$ is strange with respect to  $t^{-j}(\log t)^{-C-1}$ for all $j\in\N$.
\end{itemize}
Furthermore, we have for $n\geq1$ and $j\geq 2$
\begin{align}
\left(t^{-\frac{1}{2}}\Li(\sqrt{t})D^{(1)}(\sqrt{t})\right)^{(n)}
&\ll
\frac{(\log\log \sqrt{t} )^{a-1}}{t^{n}(\log\sqrt{t})^{b+2}},
\label{eq:Li_D_sqrt/sqrt_upper_bound}\\
\left(t^{-\frac{j}{2}}\Li(\sqrt{t})D^{(j)}(\sqrt{t})\right)^{(n)}
&\ll
\frac{(\log\log t)^{a-1}}{t^{j+n}(\log \sqrt{t})^{b+2}},\\
\left(t^{-\frac{j+1}{2}}\Li(\sqrt{t})D^{(j-1)}(\sqrt{t})\right)^{(n)}
&\ll
\frac{(\log\log t)^{a-1}}{t^{j+n}(\log \sqrt{t})^{b+2}}.
\end{align}
\end{lemma}

\begin{proof}
The calculations are almost identical for all three functions, so we will only give the proof for $t^{-1/2}\Li(\sqrt{t})D^{(1)}(\sqrt{t})$.
We use the product rule. Observe that
\begin{align*}
(t^{-1/2}\Li(\sqrt{t}))^{(k)} 
&\ll
\begin{cases}
\frac{1}{\log t}& \text{if }k=0,\\
\frac{1}{t^{k}(\log t)^{2}}& \text{if }k\geq 1.
\end{cases}
\ \text{ and } \
\big((\log\sqrt{t})^{-C}\big)^{(\ell)} 
&\ll
\begin{cases}
(\log t)^{-C}& \text{if }\ell=0,\\
t^{-\ell}(\log\sqrt{t})^{-C-1}& \text{if }\ell\geq 1.
\end{cases}
%\label{eq:diff_t_Li_sqrt3}
\end{align*}
Thus
\begin{align}
(t^{-1/2}\Li(\sqrt{t}))^{(k)} ((\log\sqrt{t})^{-C})^{(\ell)} 
\leq 
\begin{cases}
(\log t)^{-C-1}& \text{if }k=\ell=0,\\
t^{-k-\ell}(\log\sqrt{t})^{-C-2}& \text{otherwise.}
\end{cases}
\end{align}
Thus condition \eqref{eq:lem:product_rule_strange_assumption_omn_derivatives} is fulfilled
and $t^{-1/2}\Li(\sqrt{t})D^{(1)}(\sqrt{t})$ is therefore strange.
Furthermore, \eqref{eq:strange_product_rule} and \eqref{eq:lem:D_sqrt_strange_derivative_asympt} give for $n\geq 1$
\begin{align*}
\bigg(\frac{\Li(\sqrt{t})D^{(1)}(\sqrt{t})}{\sqrt{t}}\bigg)^{(n)}
=
\sum_{k=0}^{n} \binom{n}{k} (t^{-1/2}\Li(\sqrt{t}))^{(k)} (D^{(1)}(\sqrt{t}))^{(n-k)}
\ll
\frac{(\log\log \sqrt{t} )^a}{t^{n}(\log\sqrt{t})^{b+2}}.
\end{align*}
The second point follows with a similar computation.
\end{proof}

\begin{lemma}
\label{lem:int_pi_D_is_strange}
\textcolor{black}{Let $D$ satisfy the assumptions at the beginning of Section~\textnormal{\ref{sec:some-required-strange-functions}}}.
Then the function
\begin{align*}
g(t)
:=
\int_2^{\sqrt t } \pi(u;q,a) \frac{tD^{(1)}(t/u)}{u^2} du
\end{align*}	
is strange with respect to $\frac{t}{(\log t)^{C-1}}$. 
Furthermore, suppose that the polynomials $P_{a,n}(x)$ in \eqref{eq:assumption_diff_D_first} have the form 
$P_{a,n}(x) = \sum_{j=0}^{a-1} c_{j,n} x^j$ for some $c_{j,n}\in\R$.
Then
\begin{align}
g(t)
&=
c_{a-1,1}\frac{t(\log\log t)^a}{(\log t )^{b}}
+
O\left(\frac{t(\log\log t)^{a-1}}{(\log t )^{b}}\right)
\label{eq:int_Pi_D_strange},\\
g^{(1)}(t)
&=
c_{a-1,1}\frac{(\log\log t)^a}{(\log t)^{b}}
+
O\left(\frac{(\log\log t)^{a-1}}{(\log t)^{b}}\right),
\label{eq:int_Pi_D_strange2}
\end{align}
and for all $n\geq 2$
\begin{align}
g^{(n)}(t)
&=
\big(nc_{a-1,n}+c_{a-1,n+1}+2^{b-1}c_{a-1,1}(-1)^n(n-2)!\big)\frac{(\log\log t)^a}{t^{n-1}(\log t)^{b+1}}
+
O\left(\frac{(\log\log t)^{a-1}}{t^{n-1}(\log t)^{b+1}}\right).
\label{eq:int_Pi_D_strange_derivative}
\end{align}
\end{lemma}

\begin{proof}
First we define $h_0(t):=0$ and for $n\geq 1$ 
\begin{align}
h_{n}(t)
&=
% \pi(\sqrt{t};a,q) 
\frac{\Li(\sqrt{t})}{\varphi(q)}
\bigg(\frac{(n-1)D^{(n-1)}(\sqrt{t})}{2t^{\frac{n+1}{2}}}+\frac{D^{(n)}(\sqrt{t})}{2t^{\frac{n}{2}}}\bigg)
+h_{n-1}^{(1)}(t),
\label{eq:int_Li_D1_diff_for_induction1}
\end{align}
where $h_{n-1}^{(1)}(t)$ is the first pseudo-derivative of $h_{n-1}(t)$.
Then $h_{1}(t)$ is strange with respect to $(\log t)^{-C}$ and $h_{n}(t)$ is strange with respect to $t^{1-n}(\log t)^{-C-1}$ for all $n\geq 2$. 
This follows immediately with Lemma~\ref{lem:products_with_sqrt_strange} and induction.

We now claim that the $n$th pseudo-derivative of $g$ exists and has the form
\begin{align}
g^{(n)}(t)
&=\int_2^{\sqrt t } \pi(u;a,q)\bigg( \frac{nD^{(n)}(t/u)}{u^{n+1}}+\frac{tD^{(n+1)}(t/u)}{u^{n+2}} \bigg) du + h_n(t).
\label{eq:int_Li_D1_diff_for_induction0}
\end{align}
We prove \eqref{eq:int_Li_D1_diff_for_induction0} by induction over $n$.
It is fulfilled for $n=0$ by the definition of $g$ and since $h_{0}(t)=0$.
We thus assume that \eqref{eq:int_Li_D1_diff_for_induction0} holds for some $n\in\N_0$ and show that it holds also for $n+1$.
First, we show that $g^{(n)}(t)$ is once pseudo-differentiable. 
Thus we have to differentiate $g^{(n)}(t)$ and show that we can write 
\begin{align*}
(g^{(n)})'(t) &= g^{(n+1)}(t) + \widetilde{R}_{n+1}(t) ,
\end{align*}
where $g^{(n+1)}(t)$ is a differentiable function and 
\begin{align*}
\widetilde{R}_{1}(t) 
\ll \frac{1}{(\log t )^{C-1}}
\ \text{ and } \
\widetilde{R}_{n}(t) \ll \frac{1}{t^{n}(\log t )^{C}} \quad \text{for }n\geq 2.
\end{align*}
Differentiating $g^{(n)}(t)$ gives
\begin{align}
(g^{(n)})'(t)
=\,&
\int_2^{\sqrt t } \pi(u;a,q)
\bigg( \frac{n(D^{(n)})'(t/u)}{u^{n+2}}+\frac{t(D^{(n+1)})'(t/u)}{u^{n+3}} +\frac{D^{(n+1)}(t/u)}{u^{n+2}}\bigg) du\nonumber\\
&+
\pi(\sqrt{t};a,q)\bigg( \frac{nD^{(n)}(\sqrt{t})}{2(\sqrt{t})^{n+2}}+\frac{tD^{(n+1)}(\sqrt{t})}{2(\sqrt{t})^{n+3}} \bigg)
+(h_n)'(t).
\label{eq:int_Li_D1_diff_for_induction_diff}
\end{align}
We first have a look at the second line in \eqref{eq:int_Li_D1_diff_for_induction_diff}.
If $n=0$ then $(h_0)'(t)=0$. Thus $(h_0)'(t)$ is differentiable and no error term is needed.
If $n\geq 1$ then we use that $h_n(t)$ is pseudo-differentiable.
In other words, $h_n$ is a differentiable function and we can write
\begin{align}
(h_n)'(t)
=
(h_n)^{(1)}(t) +O\left(t^{-n}(\log t)^{-C-1}\right).
\end{align}
Thus the error term has in both cases the required order.
We now look at the term involving $D^{(n)}$.
Using the Siegel-Walfisz theorem, we can rewrite this term as
\begin{align*}
\pi(\sqrt{t};a,q)  \frac{nD^{(n)}(\sqrt{t})}{2(\sqrt{t})^{n+2}}  
=\,&
\frac{\Li(\sqrt{t})}{\varphi(q)}\frac{nD^{(n)}(\sqrt{t})}{2(\sqrt{t})^{n+2}} 
+
O\bigg(\frac{\sqrt{t}(\log t)^{-C-2}\, nD^{(n)}(\sqrt{t})}{2(\sqrt{t})^{n+2}} \bigg).
\end{align*}
Observe that \eqref{eq:assumption_diff_D_first} and \eqref{eq:assumption_diff_D_third} imply
\begin{align*}
\frac{\sqrt{t}(\log t)^{-C-2}\, nD^{(n)}(\sqrt{t})}{2(\sqrt{t})^{n+2}}
&\ll
\sqrt{t}(\log t)^{-C-1}\left( \frac{1}{t^{\frac{n+2}{2}} t^{\frac{n-1}{2}}}\right)
\ll
t^{-n} (\log t)^{-C-1}.
\end{align*}
Thus the error term has also the correct order. 
Similarly for the term involving $D^{(n+1)}$. 
Combining this with the definition of $h_n$, we get
\begin{align}
(g^{(n)})'(t)
=\,&
\int_2^{\sqrt t } \pi(u;a,q)
\bigg( \frac{n(D^{(n)})'(t/u)}{u^{n+2}}+\frac{t(D^{(n+1)})'(t/u)}{u^{n+3}} +\frac{D^{(n+1)}(t/u)}{u^{n+2}}\bigg) du\nonumber\\
&+
h_{n+1}(t) +
O\left(t^{-n}(\log t)^{-C-1}\right).
\label{eq:int_Li_D1_diff_for_induction_diff3}
\end{align}
Since  $D$ is strange with respect to $t(\log t )^{-C}$, we can write for all $n\in\N_0$
\begin{align*}
(D^{(n)})'(t) &= D^{(n+1)}(t) + R_{n+1}(t),
\end{align*}	
where $D^{(n+1)}(t)$ is the pseudo-derivative of $D^{(n)}(t)$  and  
\begin{align*}
R_1(t) \ll \frac{1}{(\log t)^{C}}
\ \text{ and } \
R_n(t) \ll \frac{1}{t^{n-1}(\log t)^{C+1}} \text{ for }n\geq 2.
\end{align*}
We get with these results that
\begin{align}
\int_2^{\sqrt t } &\pi(u;a,q)
\bigg( \frac{n(D^{(n)})'(t/u)}{u^{n+2}}+\frac{t(D^{(n+1)})'(t/u)}{u^{n+3}} +\frac{D^{(n+1)}(t/u)}{u^{n+2}}\bigg) du\nonumber\\
=&
\int_2^{\sqrt t } \pi(u;a,q)
\bigg(\frac{(n+1)D^{(n+1)}(t/u)}{u^{n+2}}+\frac{tD^{(n+2)}(t/u)}{u^{n+3}} \bigg) du\nonumber\\
&+
\int_2^{\sqrt t } \pi(u;a,q)
\bigg( \frac{nR_{n+1}(t/u)}{u^{n+2}}+\frac{tR_{n+2}(t/u)}{u^{n+3}}\bigg) du.
\end{align}
The integral with $D^{(n+1)}$ and $D^{(n+2)}$ is a differentiable function and is the remaining term we need for \eqref{eq:int_Li_D1_diff_for_induction0}.
We now show that we can add the integral with $R_{n+1}$ and $R_{n+2}$ to the error term $\widetilde{R}_{n+1}(t)$.
Using that $\frac{1}{2}\log t \leq \log(t/u)\leq \log t$ for $1\leq u\leq \sqrt{t}$, we get for $n\geq 1$
\begin{align*}
\int_2^{\sqrt t } \pi(u;a,q)
\left( \frac{nR_{n+1}(t/u)}{u^{n+2}}\right) du
\ll\,&
\int_2^{\sqrt t } \frac{u}{\log u}
\bigg( \frac{n}{u^{n+2}\left(\frac{t}{u}\right)^{n} (\log(t/u))^{C+1}}\bigg)\, du\\
\ll\,&
\frac{n}{t^n(\log t)^{C+1}}
\int_2^{\sqrt t } \frac{du}{u \log u}
\ll
\frac{\log\log t}{t^n(\log t)^{C+1}}
\ll
\frac{1}{t^n(\log t)^{C}}.
\end{align*}
Similarly for the integral with $R_{n+2}$ and also for $n=0$.
This completes the proof of \eqref{eq:int_Li_D1_diff_for_induction0} and that $g(t)$ is a strange function.

It remains to show that \eqref{eq:int_Pi_D_strange}, \eqref{eq:int_Pi_D_strange2} and \eqref{eq:int_Pi_D_strange_derivative}.
For this, we require for $n\geq 1$ the integral 
\begin{align}
I_n
:=
\int_2^{\sqrt t } \pi(u;a,q) \frac{D^{(n)}(t/u)}{u^{n+1}} du.
\end{align}
%
%We now claim that these integrals $I_n$ are giving the leading terms in \eqref{eq:int_Pi_D_strange}, \eqref{eq:int_Pi_D_strange2} and \eqref{eq:int_Pi_D_strange_derivative}.
We first look at the case $n=1$.
We get with the prime number theorem that
\begin{align*}
I_1
&=
\frac{1}{\varphi(q)}\int_2^{\sqrt t } \frac{ u}{\log u} \frac{D^{(1)}(t/u)}{u^{2}}  \,du
+
O\bigg(\int_2^{\sqrt t } \frac{u}{(\log u)^2} \frac{D^{(1)}(t/u)}{u^{2}} du\bigg).
%	&=
%	\frac{1}{\varphi(q)}\int_2^{\sqrt t } \Li(u) \frac{D^{(1)}(t/u)}{u^{n+1}}  du
%	+
%	O\left(\int_2^{\sqrt t } \frac{u}{(\log u)^2} \frac{D^{(1)}(t/u)}{u^{n+1}} du\right)\\
\end{align*}
Inserting the expression \eqref{eq:assumption_diff_D_first} into $D^{(1)}$ gives
\begin{align*}
\int_2^{\sqrt t } \frac{u}{(\log u)^2} \frac{D^{(1)}(t/u)}{u^{2}} \,du
\ll 
\int_2^{\sqrt t }  \frac{\big(\log\log(t/u)\big)^{a-1}}{u^{}(\log u)^2 \log^b(t/u)} \,du
&\ll 
\frac{(\log\log t)^{a-1}}{(\log t)^b}\int_2^{\sqrt t }  \frac{du}{u(\log u)^2}\\
&\ll 
\frac{(\log\log t)^{a-1}}{(\log t)^b}.
\end{align*}
Further, since $P_{a,1}(x) = \sum_{j=0}^{a-1} c_{j,1} x^j$, Lemma~\ref{lem:int_loglog_a_log_b} gives
\begin{align*}
\int_2^{\sqrt t } \frac{D^{(1)}(t/u)}{u \log u}  \,du
&=
\sum_{j=0}^{a-1} c_{j,1} \int_2^{\sqrt t } \frac{\big(\log\log(t/u)\big)^{j}}{u \log(u)\log^b(t/u)}  \,du
+
O\bigg(\int_2^{\sqrt t } \frac{du}{u \log(u)\log^b(t/u)}  \bigg)\\
&=
c_{a-1,1} \int_2^{\sqrt t } \frac{\big(\log\log(t/u)\big)^{a-1}}{u \log(u)\log^b(t/u)}  \,du
+
O\bigg(\int_2^{\sqrt t } \frac{\big(\log\log(t/u)\big)^{a-2}}{u \log(u)\log^b(t/u)}  \,du\bigg)\\
&=
c_{a-1,1}\frac{(\log\log t)^{a} }{(\log t)^b} + O\bigg(\frac{\big(\log\log t\big)^{a-1}}{(\log t)^b}\bigg).
\end{align*}
This implies that 
\begin{align}
I_1
= 
\frac{c_{a-1,1}}{\varphi(q)}\frac{(\log\log t)^{a} }{(\log t)^b} 
+ O\bigg(\frac{\big(\log\log t\big)^{a-1}}{(\log t)^b}\bigg).
\label{eq:I1_first_int_strange_lemma}
\end{align}
The computations for $I_n$ for $n\geq2$ are almost the same as those for $n=1$. 
The main difference is that we have $(\log t)^{b+1}$ instead of $(\log t)^b$ in the denominator.
Thus, we get for $n\geq 2$
\begin{align}
I_n
= 
\frac{c_{a-1,n}}{t^{n-1}\varphi(q)}\frac{(\log\log t)^{a} }{(\log t)^{b+1}} 
+ O\bigg(\frac{\big(\log\log t\big)^{a-1}}{t^{n-1}(\log t)^{b+1}}\bigg).
\end{align}
Now, $g(t)=tI_1$ and \eqref{eq:I1_first_int_strange_lemma} give immediately \eqref{eq:int_Pi_D_strange}.
Further, $g^{(1)}(t)$ involves $I_1$ and $I_2$, but $I_2$ is of lower order in this case.
Also, the assumption on $D^{(1)}$ in \eqref{eq:assumption_diff_D_first} gives
\begin{align}
h_1(t) 
=
\frac{\Li(\sqrt{t})}{\varphi(q)}\frac{D^{(1)}(\sqrt{t})}{2}t^{-\frac{1}{2}}
\ll(\frac{(\log\log t)^{a-1}}{(\log t)^{b}}.
\end{align}
Hence $I_1$ again gives the leading term in \eqref{eq:int_Pi_D_strange2} as well.
Finally, $g^{(n)}(t)$ for $n\geq 2$ involves $I_n$ and $I_{n+1}$ and both together have the same order of magnitude.
Thus both contribute to the leading term \eqref{eq:int_Pi_D_strange_derivative}.
Furthermore, we have to look at the term 
\begin{align}
\frac{\Li(\sqrt{t})}{\varphi(q)}
\bigg(\frac{(n-1)D^{(n-1)}(\sqrt{t})}{2t^{\frac{n+1}{2}}}+\frac{D^{(n)}(\sqrt{t})}{2t^{\frac{n}{2}}}\bigg)
\label{eq:ugly_hn_term_that_also_contributes_for_n=2}
\end{align}
in the definition of $h_n$ in \eqref{eq:int_Li_D1_diff_for_induction1}.
The assumption on $D$ in \eqref{eq:assumption_diff_D_first} and \eqref{eq:assumption_diff_D_third}  show that this term has for $n=2$ the same order of magnitude as $I_2$ and $I_{3}$.
More precisely, we have 
\begin{align}
%\frac{\Li(\sqrt{t})}{\varphi(q)}
%\left(\frac{D^{(1)}(\sqrt{t})}{2t^{\frac{3}{2}}}+\frac{D^{(2)}(\sqrt{t})}{2t}\right)
%=
\frac{\Li(\sqrt{t}) D^{(1)}(\sqrt{t})}{2\varphi(q)t^{\frac{3}{2}}}
%+
%O\left(\frac{(\log\log t)^{a-1}}{t(\log t)^{b+1}}\right)
=
\frac{2^{b-1}c_{a-1,1}(\log\log t)^{a-1}}{\varphi(q)t(\log t)^{b}}
+
O\left(\frac{(\log\log t)^{a-1}}{t(\log t)^{b+1}}\right).
\end{align}
On the other hand, we get for $n\geq 3$ that the term in \eqref{eq:ugly_hn_term_that_also_contributes_for_n=2} is of lower order.
Thus it follows from Lemma~\ref{lem:products_with_sqrt_strange} and a simple induction argument that $h_2(t)$ and its pseudo-derivatives contribute to the leading term for all $n\geq2$.
On the other hand, we get that the additional terms in \eqref{eq:ugly_hn_term_that_also_contributes_for_n=2} and their pseudo-derivatives do not contribute to the leading terms for $n\geq 3$.
Combining everything gives \eqref{eq:int_Pi_D_strange_derivative} and completes the proof.
%Furthermore, we have to look at the term 
%
%
%So $I_1$ again gives the leading term in \eqref{eq:int_Pi_D_strange2} as well.
%Finally, $g^{(n)}(t)$ for $n\geq 2$ involves $I_n$ and $I_{n+1}$ and both give together the leading term in \eqref{eq:int_Pi_D_strange_derivative}.
%It remains to show that the $h_n$ does not contribute to the leading term.
%However, this follows immediately from Lemma~\ref{lem:products_with_sqrt_strange} and a simple induction argument.
\end{proof}

\begin{lemma}
\label{lem:int_log_D_is_strange2}
\textcolor{black}{Let $D$ satisfy the assumptions at the beginning of Section~\textnormal{\ref{sec:some-required-strange-functions}}}.
Then the function
\begin{align*}
f(t)
:=
\int_2^{\sqrt t }  \frac{tD(u)}{u^2\log(t/u)} du
\end{align*}	
is strange with respect to $\frac{t}{(\log t)^{C-1}}$. 
Furthermore, suppose that $b=1$ in the definition of $D$ and that the polynomial $P_{a,0}(x)$ in \eqref{eq:assumption_diff_D_zero} has the form 
$P_{a,0}(x) = \sum_{j=0}^{a-1} c_{j,0} x^j$ for some $c_{j,0}\in\R$.
Then
\begin{align}
f(t)
&=
\frac{c_{a-1,0}}{a}\,\frac{t(\log\log t)^a}{\log t}
+
O\left(\frac{t(\log\log t)^{a-1}}{\log t}\right)
\label{eq:int_Pi_D_strange_f},\\
f^{(1)}(t)
&=
\frac{c_{a-1,0}}{a}\frac{(\log\log t)^a}{\log t }
+
O\left(\frac{(\log\log t)^{a-1}}{\log t }\right),
\label{eq:int_Pi_D_strange2_f}\\
\intertext{and 	for all $n\geq 2$}
f^{(n)}(t)
&=
\frac{c_{a-1,0}(n-2)!(-1)^{n+1}}{a}\frac{(\log\log t)^a}{t^{n-1}(\log t)^{2}}
+
O\left(\frac{(\log\log t)^{a-1}}{t^{n-1}(\log t)^{2}}\right).
\label{eq:int_Pi_D_strange_derivative_f}
\end{align}
\end{lemma}
We have given the asymptotics in Lemma~\ref{lem:int_log_D_is_strange2} only for the case $b=1$.
The reason is that we do not need the other cases and that the expressions in \eqref{eq:int_Pi_D_strange_f}, \eqref{eq:int_Pi_D_strange2_f} and \eqref{eq:int_Pi_D_strange_derivative_f} have a different form for $b\neq 1$.
These can be derived directly with Lemma~\ref{lem:int_loglog_a_log_b} if needed.

%
%The proof Lemma~\ref{lem:int_log_D_is_strange} is almost the same as for Lemma~\ref{lem:int_pi_D_is_strange} and we this omit it.
%
\begin{proof}
The proof of Lemma~\ref{lem:int_log_D_is_strange2} is almost the same as that of Lemma~\ref{lem:int_pi_D_is_strange}.
We thus give only a brief overview.
Observe that 
\begin{align*}
f(t)
=
\int_2^{\sqrt t }  D(u)\frac{t\Li^{(1)}(t/u)}{u^2}\,du.
\end{align*}	
Thus $f$ has therefore a similar form to the function $g$ in Lemma~\ref{lem:int_pi_D_is_strange}.
The same argument as in the proof of Lemma~\ref{lem:int_pi_D_is_strange} shows that
$f$ is strange and
\begin{align*}
f^{(n)}(t)
=
\int_2^{\sqrt t }  D(u)\bigg(\frac{n\Li^{(n)}(t/u)}{u^{n+1}} + \frac{t\Li^{(n+1)}(t/u)}{u^{n+2}}\bigg)\,du + h_n(t).
\end{align*}	
with $h_0(t) = 0$ and for $n\geq 1$
\begin{align}
h_n(t)
&=
% \pi(\sqrt{t};a,q) 
D^{(0)}(\sqrt{t})
\bigg(\frac{(n-1)\Li^{(n-1)}(\sqrt{t})}{2t^{\frac{n+1}{2}}}+\frac{\Li^{(n)}(\sqrt{t})}{2t^{\frac{n}{2}}}\bigg)
+h_{n-1}^{(1)}(t).
\end{align}
It remains to deduce the asymptotic behaviour in \eqref{eq:int_Pi_D_strange_f}, \eqref{eq:int_Pi_D_strange2_f} and \eqref{eq:int_Pi_D_strange_derivative_f}.
Thus, we assume from now that $b=1$. 
As in the proof of Lemma~\ref{lem:int_pi_D_is_strange}, we set for $n\in\N$
\begin{align}
I_n
:=
\int_{2}^{\sqrt{t}} 
D(u)\frac{\Li^{(n)}(t/u)}{u^{n+1}}\,du.
\end{align}
Since $\Li'(t)= (\log t)^{-1}$, Lemma~\ref{lem:derivatives_loglog^a_log^b} gives for $n\geq 2$
%$\Li'(t)= t(\log t)^{-2}$
\begin{align}
\Li^{(n)}(t) 
=
\frac{(n-2)!(-1)^{n+1}}{t^{n-1}(\log t)^{2}} 
+  O\left(\frac{1}{t^{n-1}(\log t)^{3}}\right).
\end{align}
Using the asymptotics of $D$ in \eqref{eq:assumption_diff_D_third}, we get
\begin{align}
I_1
=
\sum_{j=0}^{a-1} c_{j,0}\int_{2}^{\sqrt{t}} \frac{(\log\log u)^j}{u \log(t/u)}\,\frac{du}{\log u}
+
O\bigg(\int_{2}^{\sqrt{t}} \frac{(\log\log u)^{a-1}}{u(\log u)^2 \log(t/u)}\,du\bigg).
\end{align}
We now have 
\begin{align*}
\int_{2}^{\sqrt{t}} \frac{(\log\log u)^{a-1}}{u(\log u)^2 \log(t/u)}\,du
\ll
\frac{(\log\log t)^{a-1}}{\log t }\,du
\int_{2}^{\sqrt{t}} \frac{du}{u(\log u)^2}
\ll
\frac{(\log\log t)^{a-1}}{\log t }.
\end{align*}
Further,  $\log (t/u) = (\log t) (1-\frac{\log u}{\log t})$ and using the geometric series, we get 
\begin{align}
\int_{2}^{\sqrt{t}} \frac{(\log\log u)^j}{u(\log u)\log (t/u)} \,du
&=
\frac{1}{\log t} \int_{2}^{\sqrt{t}} \frac{(\log\log u)^j}{u\log u} \,du
+
O\bigg(\frac{1}{(\log t)^2} \int_{2}^{\sqrt{t}} \frac{(\log\log u)^j}{u} \,du\bigg)\nonumber\\
%&=
%\left[\frac{(\log\log u)^{j+1}}{j+1}\right]_{u=2}^{\sqrt{t}}
%+
%O\left(\frac{(\log\log t)^j}{\log t}\right)\nonumber\\
&=
\frac{(\log\log t)^{j+1}}{(j+1)\log t}
+
O\left(\frac{(\log\log t)^j}{\log t}\right).
\label{eq:we_need_a_good_name2}
\end{align}
This implies that
\begin{align}
I_1
=
\frac{c_{a-1,0}}{a}
\frac{(\log\log t)^{a}}{\log t}
+
O\left(\frac{(\log\log t)^{a-1}}{\log t}\right).
\end{align}
Similarly, we get for $n\geq 2$ that
\begin{align}
I_n
=
\frac{c_{a-1,0}}{a}
\frac{(n-2)!(-1)^{n+1}(\log\log t)^{a}}{t^{n-1}(\log t)^{2}} 
+
O\left(\frac{(\log\log t)^{a-1}}{\log t}\right).
\end{align}
Further, we have %$h_1(t)=\frac{D^{(0)}(\sqrt{t})}{2\sqrt{t}\log(\sqrt{t})}$ and 
\begin{align}
h_1(t)
&=
\frac{D^{(0)}(\sqrt{t})}{2\sqrt{t}\log\sqrt{t}}
=
O\left(\frac{(\log\log t)^{a-1}}{(\log t)^2}\right)
\ \text{ and } \\
h_2(t)
&=
\frac{D^{(0)}(\sqrt{t})}{t^{\frac{3}{2}}\log t}+ O\bigg(\frac{(\log\log t)^{a-1}}{t(\log t)^2}\bigg)
\ll
\frac{(\log\log t)^{a-1}}{(\log t)^2}
%+h_{n-1}^{(1)}(t).
\end{align}
Since $f(t)=tI_1$ and $f'(t) =I_1+tI_2+h_1(t)$, we immediately get \eqref{eq:int_Pi_D_strange_f} and \eqref{eq:int_Pi_D_strange2_f}.
It remains to look at \eqref{eq:int_Pi_D_strange_derivative_f}.
It follows from the above equation that $h_2(t)$ does not contribute to the leading term in the case $n=2$.
It is straightforward to see that the same is true for all other $n$ and $h_n(t)$.
Thus, the leading coefficient comes from $I_n$ and $I_{n+1}$ for $n\geq 2$.
Combining the above computations completes the proof.
%
%As in Lemma~\ref{lem:int_pi_D_is_strange}, $h_2$ and its pseudo-derivatives contribute to the leading term,
%but all other $h_n$ do not. 
%Thus we have to determine the leading term of the pseudo-derivatives of $h_2$.
%Observe that we have for $j\geq 0$
%\begin{align}
%\left(\frac{1}{t^{\frac{3}{2}}\log(t)}\right)^{(j)}
%=
%\frac{\Gamma(3/2+j)}{\Gamma(\frac{3}{2})}\frac{(-1)^j }{t^{j+\frac{3}{2}}\log(t)} +O\left(\frac{1}{t^{j+\frac{3}{2}}(\log t)^2}\right)
%\end{align}
%and thus the product rule and the properties of $D$ imply that we have for $j\geq 1$
%\begin{align*}
%(h_2(t))^{(j)}
%&=
%\frac{\Gamma(3/2+j)}{\Gamma(\frac{3}{2})}\frac{(-1)^j D^{(0)}(\sqrt{t})}{t^{j+\frac{3}{2}}\log(t)} 
%+
%\frac{\Gamma(1/2+j)}{\Gamma(\frac{3}{2})}\frac{(-1)^{j-1} D^{(1)}(\sqrt{t})}{2t^{j+1}\log(t)} 
%+
%O\left(\frac{(\log\log t)^{a-1}}{t^{j+2}(\log t)^2}\right)\\
%&=
%\frac{\Gamma(3/2+j)}{\Gamma(\frac{3}{2})}\frac{(-1)^j c_{a-1,0}(\log\log t)}{t^{j+1}\log^2(t)}  + O\left(\frac{t(\log\log t)^{a-1}}{(\log t)^{2}}\right),
%\end{align*}
%
\end{proof}

\begin{lemma}
\label{lem:int_log_D_is_strange2_lambda}
\textcolor{black}{Let $D$ satisfy the assumptions at the beginning of Section~\textnormal{\ref{sec:some-required-strange-functions}}}.
Then the function
\begin{align*}
\widetilde{g}(t)
:=
\int_{2}^{\sqrt{t}} \psi(u;q,a) \frac{t D^{(1)}(t/u)}{u^2}\,du	
\end{align*}	
is strange with respect to $\frac{t}{(\log t)^{C-1}}$. 
Furthermore, suppose that assume $b\leq 0$ and the polynomial $P_{a,n}(x)$ in \eqref{eq:assumption_diff_D_zero} has the form 
$P_{a,n}(x) = \sum_{j=0}^{a-1} c_{j,n} x^j$ for some $c_{j,n}\in\R$.
We then have for $n=0$ and $n=1$ that
\begin{align}
\widetilde{g}^{(n)}(t)
&=
t^{1-n}
\frac{c_{a-1,1}(1-2^{b-1})}{\varphi(q)(1-b)}\frac{(\log\log t)^{a-1}}{(\log t)^{b-1}} 
%+
%t^{n-1}\frac{\widetilde{P}_{a,n}(\log\log t)}{(\log t)^{b-1}} 
+
O\left(\frac{t^{1-n}(\log\log t)^{a-2}}{(\log t)^{b-1}} \right),
%+
%O\left(\frac{t(\log\log t)^{a-1}}{(\log t )^{b}}\right)
\label{eq:int_Pi_D_strange_g_lambda}\\
%\widetilde{g}^{(1)}(t)
%&=
%%\frac{c_{a-1,0}}{a}\frac{(\log\log t)^a}{(\log t)^{b}}
%+
%O\left(\frac{(\log\log t)^{a-1}}{(\log t)^{b}}\right)
%\label{eq:int_Pi_D_strange2_g_lambda}\\
\intertext{and 	for all $n\geq 2$}
\widetilde{g}^{(n)}(t)
=\,&
\frac{1}{\varphi(q)}\left(\frac{(nc_{a-1,n} + c_{a-1,n+1})(2^{b}-1)}{b}+2^{b-1}c_{a-1,1}(-1)^n(n-2)!\right)\frac{(\log\log t)^{a-1}}{t^{n-1}(\log t)^{b}}\nonumber\\
&+
O\left(\frac{(\log\log t)^{a-2}}{t^{n-1}(\log t)^{b}}\right).
\label{eq:int_Pi_D_strange_derivative_f9}
\end{align}
\end{lemma}
In the case $b=0$, the term $\frac{2^b -1}{b}$ has to be interpreted as $\log 2$.

%
%The proof Lemma~\ref{lem:int_log_D_is_strange} is almost the same as for Lemma~\ref{lem:int_pi_D_is_strange} and we this omit it.
%
\begin{proof}
The proof Lemma~\ref{lem:int_log_D_is_strange2_lambda} is similar to the proof of Lemma~\ref{lem:int_pi_D_is_strange}.
We thus give only a brief overview.
For this, we define $\widetilde{h}_0(t):=0$ and for $n\geq 1$ 
\begin{align}
\widetilde{h}_{n}(t)
&=
\frac{\sqrt{t}}{\varphi(q)}
\bigg(\frac{(n-1)D^{(n-1)}(\sqrt{t})}{2t^{\frac{n+1}{2}}}+\frac{D^{(n)}(\sqrt{t})}{2t^{\frac{n}{2}}}\bigg)
+\widetilde{h}_{n-1}^{(1)}(t),
\label{eq:int_Li_D1_diff_for_induction1_lambda}
\end{align}
where $\widetilde{h}_{n-1}^{(1)}(t)$ is the first pseudo-derivative of $\widetilde{h}_{n-1}(t)$.
Then the $n$th pseudo-derivative of $\widetilde{g}$ exists and has the form
\begin{align}
\widetilde{g}^{(n)}(t)
&=
\int_2^{\sqrt t } \psi(u;a,q)\bigg( \frac{nD^{(n)}(t/u)}{u^{n+1}}+\frac{tD^{(n+1)}(t/u)}{u^{n+2}} \bigg) du + \widetilde{h}_n(t).
\label{eq:int_Li_D1_diff_for_induction0_lambda}
\end{align}
Using Lemma~\ref{lem:int_loglog_a_log_b}, we get for $n=0$ and $n=1$ that
%\begin{align*}
%\int_2^{\sqrt t } 
%\psi(u;a,q) \frac{D^{(n)}(t/u)}{u^{n+1}} du 	
%=\,&
%%\int_2^{\sqrt t } 
%%u\left( \frac{P_{a,n}(\log\log(t/u))}{(t/u)^{n-1}\log^{b}(t/u)u^{n+1}}
%%\right) du 
%%+O\left(\int_2^{\sqrt t } 
%%u\left( \frac{(\log\log(t/u))^{a-1}}{(t/u)^{n-1}\log^{b}(t/u)u^{n+1}}
%%\right) du \right)\\
%%=\,&
%%t^{n-1}\int_2^{\sqrt t } 
%% \frac{P_{a,n}(\log\log(t/u))}{u\log^{b}(t/u)}\, du 
%%+O\left(t^{n-1}\int_2^{\sqrt t } 
%%\frac{(\log\log(t/u))^{a-1}}{u\log^{b}(t/u)\log^{C}(u)}\,du \right)\\
%%=\,&
%%t^{n-1}\int_2^{\sqrt t } 
%%\frac{P_{a,n}(\log\log(t/u))}{u\log^{b}(t/u)}\,du 
%%+
%%O\left(\frac{t^{n-1}(\log\log t)^{a-1}}{(\log t)^{b}}\int_2^{\sqrt t } 
%%\frac{1}{u\log^{C}(u)}\,du \right)
%%=\,&
%t^{n-1}\int_2^{\sqrt t } 
%\frac{P_{a,n}(\log\log(t/u))}{u\log^{b}(t/u)}\,du 
%+
%O\left(\frac{t^{n-1}(\log\log t)^{a-1}}{(\log t)^{b}} \right)\\
%%=\,&
%%t^{n-1}
%%c_{a-1,b,0}\frac{(\log\log t)^{a-1} + \widetilde{P}_{a,n}(\log\log t)}{ 2^{}(\log t)^{b-1}} 
%%+
%%O\left(\frac{t^{n-1}(\log\log t)^{a-1}}{(\log t)^{b}} \right)\\
%=\,&
%t^{n-1}
%\frac{c_{a-1,n}(1-2^{b-1})}{1-b}\frac{(\log\log t)^{a-1} + \widetilde{P}_{a,n}(\log\log t)}{(\log t)^{b-1}} 
%+
%O\left(\frac{t^{n-1}(\log\log t)^{a-1}}{(\log t)^{b}} \right),\\
%=\,&
%t^{n-1}
%\frac{c_{a-1,n}(1-2^{b-1})}{1-b}\frac{(\log\log t)^{a-1}}{(\log t)^{b-1}} \\
%&+
%t^{n-1}\frac{\widetilde{P}_{a,n}(\log\log t)}{(\log t)^{b-1}} 
%+
%O\left(\frac{t^{n-1}(\log\log t)^{a-1}}{(\log t)^{b}} \right),
%\end{align*}
%
\begin{align*}
\int_2^{\sqrt t } 
\psi(u;a,q) \frac{D^{(n)}(t/u)}{u^{n+1}} du 	
&=
t^{n-1}\int_2^{\sqrt t } 
\frac{P_{a,n}(\log\log(t/u))}{u\log^{b}(t/u)}\,du 
+
O\left(\frac{t^{n-1}(\log\log t)^{a-2}}{(\log t)^{b}} \right)\\
%=\,&
%t^{n-1}
%c_{a-1,b,0}\frac{(\log\log t)^{a-1} + \widetilde{P}_{a,n}(\log\log t)}{ 2^{}(\log t)^{b-1}} 
%+
%O\left(\frac{t^{n-1}(\log\log t)^{a-1}}{(\log t)^{b}} \right)\\
%	=\,&
%	t^{n-1}
%	\frac{c_{a-1,n}(1-2^{b-1})}{1-b}\frac{(\log\log t)^{a-1} + \widetilde{P}_{a,n}(\log\log t)}{(\log t)^{b-1}} 
%	+
%	O\left(\frac{t^{n-1}(\log\log t)^{a-1}}{(\log t)^{b}} \right),\\
&=
t^{1-n}
\frac{c_{a-1,n}(1-2^{b-1})}{1-b}\frac{(\log\log t)^{a-1}}{(\log t)^{b-1}} \nonumber \\
&\quad+
t^{1-n}\frac{\widetilde{P}_{a,n}(\log\log t)}{(\log t)^{b-1}} 
+
O\left(\frac{t^{n-1}(\log\log t)^{a-2}}{(\log t)^{b}} \right),
\end{align*}

where $\widetilde{P}_{a,n}$ is a polynomial of degree $a-2$.
Similarly, we get for $n\geq 2$ that
\begin{align*}
\int_2^{\sqrt t } 
\psi(u;a,q) \frac{D^{(n)}(t/u)}{u^{n+1}} du 	
=\,&
t^{1-n}
\frac{c_{a-1,n}(1-2^{b})}{-b}\frac{(\log\log t)^{a-1}}{(\log t)^{b-1}} \\	
&+
t^{1-n}\frac{\widetilde{P}_{a,n}(\log\log t)}{(\log t)^{b}} 
+
O\left(\frac{t^{1-n}(\log\log t)^{a-1}}{(\log t)^{b+1}} \right).
\end{align*}
Further, we have
\begin{align}	
\widetilde{h}_{2}(t)
&=
\frac{D^{(1)}(\sqrt{t})}{2{\varphi(q)}t}
+
O\left(\frac{(\log\log t)^{a-1}}{t(\log t)^{b+1}}\right)
=
\frac{c_{a-1,1}(\log\log t)^{a-1}}{2{\varphi(q)}t(\log t)^{b}}
+
O\left(\frac{(\log\log t)^{a-1}}{t(\log t)^{b+1}}\right).
\end{align}

%\begin{align}	
%\widetilde{h}_{2}(t)
%&=
%\frac{1}{\varphi(q)}
%\left(\frac{D^{(1)}(\sqrt{t})}{2t}+\frac{D^{(2)}(\sqrt{t})}{2t^{t}\right)
%+\widetilde{h}_{n-1}^{(1)}(t),
%\end{align}

Combining these two equations completes the proof.
\end{proof}

\begin{lemma}
\label{lem:int_log_D_is_strange3_lambda}
\textcolor{black}{Let $D$ satisfy the assumptions at the beginning of Section~\textnormal{\ref{sec:some-required-strange-functions}}} and assume $b\leq 0$.
Then the function
\begin{align*}
\widetilde{f}(t)
:=
\int_{2}^{\sqrt{t}} \frac{t D(u)}{u^2} \,du.
\end{align*}	
is strange with respect to $\frac{t}{(\log t)^{C-1}}$. 
Furthermore, suppose that the polynomial $P_{a,n}(x)$ in \eqref{eq:assumption_diff_D_zero} has the form 
$P_{a,n}(x) = \sum_{j=0}^{a-1} c_{j,n} x^j$ for some $c_{j,n}\in\R$.
Then
\begin{align}
\widetilde{f}(t)
&=
c_{a-1,0} 2^{b-1}\frac{t(\log\log t)^{a-1}}{(1-b)(\log t)^{b-1}} \,du
+
O\left(\frac{(\log\log t)^{a-2}}{(\log t)^{b-1}}\right)
%+
%O\left(\frac{t(\log\log t)^{a-1}}{(\log t )^{b}}\right)
\label{eq:int_Pi_D_strange_f_lambda3},\\
\widetilde{f}^{(1)}(t)
&=
c_{a-1,0} 2^{b-1} \frac{t(\log\log t)^{a-1}}{(1-b)(\log t)^{b-1}} \,du
+
O\left(\frac{(\log\log t)^{a-2}}{(\log t)^{b-1}}\right),
\label{eq:int_Pi_D_strange3_f_lambda4}\\
\intertext{and 	for all $n\geq 2$}
\widetilde{f}^{(n)}(t)
&=
(c_{a,0}+c_{a,1})
\frac{\Gamma(-\frac{1}{2}) }{\Gamma(\frac{1}{2}-n)}\frac{2n-3}{2n-1}\frac{(\log\log t)^{a-1}}{2t^{n-1}(\log\sqrt{t})^b} 
\nonumber\\
&\quad +
c_{a,1} (n-2)!(-1)^{n-1}(n-2)
\frac{\Gamma(-\frac{1}{2}) }{\Gamma(\frac{1}{2}-n)}\frac{(\log\log t)^{a-1}}{2t^{n-1}(\log\sqrt{t})^b}.
\label{eq:int_Pi_D_strange_derivative_f3}
\end{align}
\end{lemma}
The proof of this lemma shows that the equations \eqref{eq:int_Pi_D_strange3_f_lambda4} and \eqref{eq:int_Pi_D_strange_derivative_f3} also hold for $b> 0$, but \eqref{eq:int_Pi_D_strange_f_lambda3} does not hold for $b>0$. Since we need only the case $b\leq 0$, we will not state the asymptotic of $ \widetilde{f}(t)$ for $b> 0$, but it can easily deduced from the proof if needed.
%
%The proof Lemma~\ref{lem:int_log_D_is_strange} is almost the same as for Lemma~\ref{lem:int_pi_D_is_strange} and we this omit it.
%
\begin{proof}
The proof Lemma~\ref{lem:int_log_D_is_strange2_lambda} is very similar to the proof of Lemma~\ref{lem:int_pi_D_is_strange}.
We thus give only a brief overview.
%
%Observe that we can write $\widetilde{f}(t) = t s(t)$ with $s(t)=\int_{2}^{\sqrt{t}} \frac{D(u)}{u^2} \,du$.
The product rule for strange functions shows that the $n$th pseudo-derivative of $\widetilde{f}(t)$ is for $n\geq 1$
\begin{align}
\widetilde{f}^{(n)}(t)
=
t s^{(n)}(t) + s^{(n-1)}(t)
\ \text{ with }\
s(t)
=
\int_{2}^{\sqrt{t}} \frac{D(u)}{u^2} \,du.
\label{eq:hopefully_the:last_ugly_eq}
\end{align}
We now get with the properties of $D$ in \eqref{eq:assumption_diff_D_zero} that
\begin{align*}
s(t)
&=
\int_{2}^{\sqrt{t}} \frac{P_{a,0}(\log\log u)}{u(\log u)^b} \,du
+
O\bigg(\int_{2}^{\sqrt{t}} \frac{(\log\log u)^{a-1}}{u(\log u)^{b+1}} \,du\bigg).
%\\
%&=1
%\int_{2}^{\sqrt{t}} \frac{P_{a,0}(\log\log u)}{u(\log u)^b} \,du
%+
%O\left((\log\log t)^{a-1}\int_{2}^{\sqrt{t}} \frac{1}{u(\log u)^{b+1}} \,du\right)\\
\end{align*}
We have 
\begin{align}
\int_{2}^{\sqrt{t}} \frac{(\log\log u)^{a-1}}{u(\log u)^{b+1}} \,du
\ll
(\log\log t)^{a-1}\int_{2}^{\sqrt{t}} \frac{du}{u(\log u)^{b+1}}
\ll
\frac{(\log\log t)^{a-1}}{u(\log t)^{b}}.
\end{align}
Thus, we get with partial integration that
\begin{align*}
s(t)
&=
c_{a-1,0}\int_{2}^{\sqrt{t}} \frac{(\log\log u)^{a-1}}{u(\log u)^b} \,du
+
O\bigg(\frac{(\log\log t)^{a-2}}{(\log t)^{b-1}}\bigg)\\
&=
c_{a-1,0} \frac{(\log\log t)^{a-1}}{(1-b)(\log \sqrt{t})^{b-1}} \,du
+
O\bigg(\frac{(\log\log t)^{a-2}}{(\log t)^{b-1}}\bigg).
\end{align*}
This completes the proof of \eqref{eq:int_Pi_D_strange_f_lambda3} since $\widetilde{f}(t)= ts(t)$.
Furthermore, we have $s'(t)= \frac{1}{2}t^{-\frac{3}{2}}D(\sqrt{t})$.
Inserting \eqref{eq:assumption_diff_D_first}  gives
\begin{align*}
s'(t)
=
c_{a-1,0}\frac{(\log\log t)^{a-1}}{2t(\log \sqrt{t})^{b}}
+
O\bigg(\frac{(\log\log t)^{a-2}}{(\log t)^{b}}\bigg).
\end{align*}
Since $\widetilde{f}^{(1)}(t)= ts'(t)+s(t)$, we immediately get \eqref{eq:int_Pi_D_strange3_f_lambda4}.
It remains to show \eqref{eq:int_Pi_D_strange_derivative_f3}.
Observe that we have for $m\geq 0$
\begin{align}
(t^{-\frac{3}{2}})^{(m)}
=
\frac{\Gamma(-\frac{1}{2}) }{\Gamma(-\frac{1}{2}-m)}t^{-m-\frac{3}{2}}.
\end{align}
Using the product rule and Lemma~\ref{lem:D_sqrt_strange}, we get for $n\geq 2$
\begin{align*}
s^{(n)}(t)
=\,&
\left(\frac{1}{2}t^{-\frac{3}{2}}D(\sqrt{t})\right)^{(n-1)}
=
\frac{1}{2}\sum_{m=0}^{n-1} \binom{n-1}{m} \left(D(\sqrt{t})\right)^{(m)} (t^{-\frac{3}{2}})^{(n-1-m)}\\
=\,&
c_{a,0}\frac{\Gamma(-\frac{1}{2}) }{\Gamma(\frac{1}{2}-n)}\frac{(\log\log t)^{a-1}}{2t^{n}(\log\sqrt{t})^b} 
+ 
O\left(\frac{(\log\log t)^{a-2}}{2t^{n}(\log\sqrt{t})^b}\right)
\\
&
+
c_{a,1}\frac{(\log\log t)^{a-1}}{2t^{n}(\log\sqrt{t})^b} \bigg(
\sum_{m=1}^{n-1} \binom{n-1}{m} \frac{\Gamma(\frac{3}{2})}{\Gamma(\frac{3}{2}-m)} \frac{\Gamma(-\frac{1}{2}) }{\Gamma(\frac{1}{2}-n+m)}  \bigg)\\
=\,&
c_{a,0}\frac{\Gamma(-\frac{1}{2}) }{\Gamma(\frac{1}{2}-n)}\frac{(\log\log t)^{a-1}}{2t^{n}(\log\sqrt{t})^b} 
+ 
O\bigg(\frac{(\log\log t)^{a-2}}{2t^{n}(\log\sqrt{t})^b}\bigg)
\\
&
+
c_{a,1}\frac{(\log\log t)^{a-1}}{2t^{n}(\log\sqrt{t})^b} 
\bigg( \left(-1\right)^{n-1}(n-1)!+\frac{ \Gamma(-\frac{1}{2})}{\Gamma(\frac{1}{2}-n)} \bigg).
\end{align*}
Inserting this into \eqref{eq:hopefully_the:last_ugly_eq} completes the proof.
\end{proof}

\subsection{Finalizing the proofs of Theorem~\ref{thm:Ar_is_strange} and~\ref{thm:Ar_is_strange_lambda}}
\label{sec:proof_Ar_complete}

In this subsection we complete the proof of the Theorems~\ref{thm:Ar_is_strange} and~\ref{thm:Ar_is_strange_lambda}.
As mentioned before, we do this by induction over $r$.
Since both proofs are very similar, we give the proof of Theorem~\ref{thm:Ar_is_strange} in detail and only highlight the necessary adjustments for the proof of Theorem~\ref{thm:Ar_is_strange_lambda}.
%
%Recall, Theorem~\ref{thm:Ar_is_strange} states that  $A_{r}^{(0)}(t,q)$ is strange and fulfils
%\begin{align}
%	A_{r}(t;q,\ell)
%	=
%	A_{r}^{(0)}(t;q) +O(t(\log t)^{-C}),
%	\label{eq:As_aympt_leading_only_q_and_O:uniformP}
%\end{align}
%where $A_{r}^{(0)}(t,q) $ is independent of $\ell$ and the implicit constant in $O(\cdot)$ can be chosen independent of $q$ and $\ell$.	
%Furthermore, Theorem~\ref{thm:Ar_is_strange} states we have for $n=0$ and $n=1$
%\begin{align}
%	A_{r}^{(n)}(t;q) 
%	&= 
%	\frac{r}{\varphi(q)}\frac{t^{1-n}(\log\log t)^{r-1}}{\log t } + O\left(\frac{t^{1-n}(\log\log q)(\log\log t)^{r-2}}{\varphi(q) \log t}\right),
%	\label{eq:find_a_cool_name_0_1P}\\
%\intertext{and for $n\geq 2$}
%	A^{(n)}_r(t;q)  
%	&=
%	(-1)^{n+1}\frac{r(n-2)!}{\varphi(q)} 
%	\frac{(\log\log t)^{r-1}}{t^{n-1}(\log t)^2} 
%	+  
%	O\left(\frac{(\log\log q)(\log\log t )^{r-2}}{\varphi(q)t^{n-1}(\log t)^2}\right).
%	\label{eq:find_a_cool_name2P}
%\end{align}

The first step in the proof is to show that both theorems hold for $r=1$.
The prime number theorem and the Siegel-Walfisz theorem imply
\begin{align}
A_1(t;q,\ell) 
&=
\pi(t;q,\ell)
=
\frac {\Li(t)}{\varphi(q)} + O(t(\log t)^{-C}),
\label{eq:proof_Ar_strange_case_r_1}\\
\widetilde{A}_1(t;q,\ell) 
&=
\psi(t;q,\ell)
=
\frac {t}{\varphi(q)} + O(t(\log t)^{-C}),
\label{eq:proof_Ar_strange_case_r_1_lambda}
\end{align}
where $C>1$ can be chosen arbitrarily.
Since $\Li(t)$ and $t$ are smooth functions, \eqref{eq:proof_Ar_strange_case_r_1} and \eqref{eq:proof_Ar_strange_case_r_1_lambda} immediately imply that $A_1(r;q,\ell)$ and $\widetilde{A}_1(t;q,\ell) $ are strange with respect to $t(\log t)^{-C}$. Further, $A_1^{(0)}(t;q,\ell) = \frac{\Li(t)}{\varphi(q)}$ and $\widetilde{A}_1^{(0)}(t;q,\ell) = \frac{t}{\varphi(q)}$ and both are therefore  independent of $\ell$.
Also, the implicit constant in the error terms in \eqref{eq:proof_Ar_strange_case_r_1} and \eqref{eq:proof_Ar_strange_case_r_1_lambda} can be chosen uniformly in $q$ and $\ell$ for $q\leq (\log t)^A$, see begin of Section~\ref{sec:pseudo-differentiable-functions-and-non-principal-arcs}. 
Finally, we have $\Li'(t) = (\log t)^{-1}$ and Lemma~\ref{lem:derivatives_loglog^a_log^b} therefore shows that the (pseudo) derivatives of $\Li(t)$ fulfil \eqref{eq:find_a_cool_name_0_1} and \eqref{eq:find_a_cool_name2}.
Also the function $t/\varphi(q)$ clearly fulfills \eqref{eq:find_a_cool_name_0_1_lambda} and \eqref{eq:find_a_cool_name2_lambda}.
Thus both theorems hold for $r=1$.

Next, we perform the induction step for the proof of Theorem~\ref{thm:Ar_is_strange}.
Recall, Theorem~\ref{thm:Ar_is_strange} states that  $A_{r}^{(0)}(t,q)$ is strange and fulfils
\begin{align}
	A_{r}(t;q,\ell)
	=
	A_{r}^{(0)}(t;q) +O(t(\log t)^{-C}),
	\label{eq:As_aympt_leading_only_q_and_O:uniformP}
\end{align}
where $A_{r}^{(0)}(t,q) $ is independent of $\ell$ and the implicit constant in the $O(\cdot)$ term can be chosen independently of $q$ and $\ell$.	
Furthermore, Theorem~\ref{thm:Ar_is_strange} states we have for $n=0$ and $n=1$
\begin{align}
	A_{r}^{(n)}(t;q) 
	&= 
	\frac{r}{\varphi(q)}\frac{t^{1-n}(\log\log t)^{r-1}}{\log t } + O\left(\frac{t^{1-n}(\log\log q)(\log\log t)^{r-2}}{\varphi(q) \log t}\right),
	\label{eq:find_a_cool_name_0_1P}\\
	\intertext{and for $n\geq 2$}
	A^{(n)}_r(t;q)  
	&=
	(-1)^{n+1}\frac{r(n-2)!}{\varphi(q)} 
	\frac{(\log\log t)^{r-1}}{t^{n-1}(\log t)^2} 
	+  
	O\left(\frac{(\log\log q)(\log\log t )^{r-2}}{\varphi(q)t^{n-1}(\log t)^2}\right).
	\label{eq:find_a_cool_name2P}
\end{align}
Suppose now that Theorem~\ref{thm:Ar_is_strange} holds for all $s$ with $s\leq r-1$ and $r\geq2$. 
We now show it also holds for $r$.
Recall, Lemma~\ref{lem:explcit_Ar_as_A{r-1}} gives
\begin{align*}
A_{r}(t;q,\ell)
=&
\frac{\Li(\sqrt{t})A^{(0)}_{r-1}(\sqrt{t})}{\varphi(q)}
-
\sum_{p|q} A_{r-1}^{(0)}(t/p;q)
-
\frac{1}{\varphi(q)}
\sum_{\substack{\mathcal{J}(\textbf{p}_{r-1}) \leq\sqrt{t}\\ (\mathcal{J}(\textbf{p}_{r-1}),q)>1}}
\Li\left(\frac{t}{\mathcal{J}(\textbf{p}_{r-1})}\right)   \nonumber\\
&+
\int_{2}^{\sqrt{t}} \pi(u) \frac{tA_{r-1}^{(1)}(t/u;q)}{u^2}\,du
+
\frac{1}{\varphi(q)}
\int_{2}^{\sqrt{t}} \frac{t A_{r-1}(u)}{u^2 \log(t/u)} \,du
+
O\left(\frac{t}{(\log t)^{C}}\right).
\end{align*}
Now, Lemma~\ref{lem:int_pi_D_is_strange}, the induction hypothesis for $r-1$ and using \eqref{eq:find_a_cool_name_0_1P} for $r-1$ show that
\begin{align}
\int_{2}^{\sqrt{t}} \pi(u) \frac{tA_{r-1}^{(1)}(t/u;q)}{u^2}\,du
=
\frac{(r-1)}{\varphi(q)}\frac{t(\log\log t)^a}{\log t}
+
O\left(\frac{t(\log\log q)(\log\log t)^{a-1}}{\varphi(q) \log t }\right)
\label{eq:final_step_Ar1}
\end{align}
and that this integral is a strange function.
Further, Lemma~\ref{lem:int_log_D_is_strange2}, the induction hypothesis for $r-1$ and using \eqref{eq:find_a_cool_name_0_1P} for $r-1$ show that
\begin{align}
\frac{1}{\varphi(q)}
\int_{2}^{\sqrt{t}} \frac{t A_{r-1}(u)}{u^2 \log(t/u)} \,du
=
\frac{1}{\varphi(q)}\frac{t(\log\log t)^a}{\log t}
+
O\left(\frac{t(\log\log q)(\log\log t)^{a-1}}{\varphi(q)\log t}\right)
\label{eq:final_step_Ar2}
\end{align}
and that this integral is also a strange function.
Combining \eqref{eq:final_step_Ar1} and \eqref{eq:final_step_Ar2} gives \eqref{eq:find_a_cool_name_0_1P} for $n=0$.
Similarly, Lemma~\ref{lem:int_pi_D_is_strange} and Lemma~\ref{lem:int_log_D_is_strange2} show that the pseudo-derivatives 
of these two integrals give \eqref{eq:find_a_cool_name_0_1P} for $n=1$, and \eqref{eq:find_a_cool_name2P} for $n\geq2$.
It thus remains to show that the remaining terms are strange and of lower order.

Now, Lemma~\ref{lem:D_Li_sqrt_strange} shows that 
\begin{align}
\frac{\Li(\sqrt{t})A^{(0)}_{r-1}(\sqrt{t})}{\varphi(q)}
\ll
\frac{t(\log\log t)^{r-2}}{(\log t)^{2}}
\label{eq:final_step_Ar3}
\end{align}
and that this function is strange. 
Thus this term contributes only to the error term in \eqref{eq:find_a_cool_name_0_1P} for $n=0$.
Further, Lemma~\ref{lem:D_Li_sqrt_strange} also shows that the pseudo-derivatives of this term
only contribute to the error term in \eqref{eq:find_a_cool_name_0_1P} for  $n=1$ and \eqref{eq:find_a_cool_name2P} for $n\geq2$.

Next we look at the sum $\sum_{p|q} A_{r-1}^{(0)}(t/p;q)$.
For this, suppose that $f(t)$ is strange with respect to $g(t)$.
Then the chain rule (or directly Definition~\ref{def:strange_function}) implies that $f(t/p)$ is strange with respect to $g(t/p)$ and 
$
(f(t/p))^{(n)} = f^{(n)}(t/p)p^{-n}
$.
Since $q\leq (\log t)^A$, it is straightforward to see that 
$A_{r-1}^{(0)}(t/p;q)$ is strange with respect to $(t/p)(\log t )^{-C-1}$ for all $p\leq q$
and that the implicit constants occurring in all $O(\cdot)$ terms in the pseudo-derivatives can be chosen independently of $p$ and $q$.  
This then implies that
\begin{align}
\sum_{p|q} A_{r-1}^{(0)}(t/p;q)
\label{eq:sum_Ar_q}
\end{align}
is strange with respect to $t(\log t )^{-C}$.
Further, 
\begin{align*}
\sum_{p|q} A_{r-1}^{(0)}(t/p;q)
\ll
\sum_{p|q} 
\frac{(t/p)(\log\log (t/p))^{r-2}}{\log (t/p)}
&\ll
\frac{t(\log\log t)^{r-2}}{\log t}
\sum_{p|q} \frac{1}{p}\\
&\ll
\frac{t(\log\log q)(\log\log t)^{r-2}}{\log t}.
\end{align*}
The computation for the pseudo-derivatives is similar. Thus this term only contributes to the error terms.
It remains to look at the sum
\begin{align*}
S(t):=
\frac{1}{\varphi(q)}
\sum_{\substack{\mathcal{J}(\textbf{p}_{r-1}) \leq\sqrt{t}\\ (\mathcal{J}(\textbf{p}_{r-1}),q)>1}}
\Li\left(\frac{t}{\mathcal{J}(\textbf{p}_{r-1})}\right).  
\end{align*}
In order to have $(\mathcal{J}(\textbf{p}_{r-1}),q)>1$, at least one $p_j$ must divide $q$.
Apart from that, there is no further restriction on the other $p_j$ (except that their product be less than $\sqrt{t})$.
%In particular, there is no requirement that they are less than $q$. 
%
We write $\textbf{p}_{r-1} = (\textbf{p}_{s}, \widetilde{\textbf{p}}_{r-1-s})$ and use the notation $\textbf{p}_{s}|q$ if $p_j|q$ for all $1\leq j\leq s$.
Using the inclusion-exclusion principle, we obtain
\begin{align}
S(t)
=
\frac{1}{\varphi(q)}
\sum_{s=1}^{r-1} (-1)^{s+1}\binom{r-1}{s} S_s(t)
\end{align}
with
\begin{align}
S_s(t)
:=
\sum_{\substack{\textbf{p}_{s}|q}}
\sum_{\substack{\mathcal{J}(\widetilde{\textbf{p}}_{r-1-s}) \leq\frac{\sqrt{t}}{\mathcal{J}(\textbf{p}_{s})}
}}
\Li\left(\frac{t}{\mathcal{J}(\textbf{p}_{s})\mathcal{J}(\widetilde{\textbf{p}}_{r-1-s})}\right).
\end{align}
%where we write $\textbf{p}_{r-1} = (\textbf{p}_{s}, \widetilde{\textbf{p}}_{r-1-s})$ and use the notation $\textbf{p}_{s}|q$ if $p_j|q$ for all $1\leq j\leq s$.
Consequently it is sufficient to show each $S_s(t)$ is strange and of lower order.
Observe that we have for $s=r-1$ 
\begin{align*}
S_{r-1}(t)
=
\sum_{\substack{\textbf{p}_{r-1}|q}}
\Li\left(\frac{t}{\mathcal{J}(\textbf{p}_{r-1})}\right).
\end{align*}
Thus we can use exactly the same argument for $S_{r-1}(t)$ as for the sum in \eqref{eq:sum_Ar_q}.
Hence this term is strange and only contributes to the error terms.
Further, we use Abel's summation formula for $s<r-1$ on the inner sum of $S_s(t)$.
This computation is almost the same as in \eqref{eq:almost_done} and gives
\begin{align}
\sum_{\substack{\mathcal{J}(\widetilde{\textbf{p}}_{r-1-s}) \leq\frac{\sqrt{t}}{\mathcal{J}(\textbf{p}_{s})}
}}
&\Li\left(\frac{t}{\mathcal{J}(\textbf{p}_{s})\mathcal{J}(\widetilde{\textbf{p}}_{r-1-s})}\right)\nonumber\\
&=
\Li(\sqrt{t}) A_{r-1-s}\left(\frac{\sqrt{t}}{\mathcal{J}(\textbf{p}_{s})}\right)
+
\frac{1}{\mathcal{J}(\textbf{p}_{s})}
\int_{2}^{\frac{\sqrt{t}}{\mathcal{J}(\textbf{p}_{s})}} 
\frac{t A_{r-1-s}(u)}{u^2 \log(t/(u\mathcal{J}(\textbf{p}_{s})))} \,du.
\label{eq:almost_done_really}
\end{align}
A minor adjustment of Lemma~\ref{lem:int_pi_D_is_strange} and Lemma~\ref{lem:int_log_D_is_strange2} shows that 
each term in \eqref{eq:almost_done_really} is strange and that similar bounds hold.
Combining everything shows that the sum $S_t$ and its pseudo-derivatives are of lower order.
This completes the proof of Theorem~\ref{thm:Ar_is_strange}.

The induction step for the proof of Theorem~\ref{thm:Ar_is_strange_lambda} is almost identical.
Instead of  Lemma~\ref{lem:explcit_Ar_as_A{r-1}}, we use  Lemma~\ref{lem:explcit_Ar_as_A{r-1}_lambda}  and get
\begin{align}
\widetilde{A}_{r}(t;q,\ell)
=&
\frac{\sqrt{t}\widetilde{A}^{(0)}_{r-1}(\sqrt{t})}{\varphi(q)}
-
\sum_{\substack{n_{r}\leq \sqrt{t}\\(n_{r},q)>1}} \Lambda(n_{r}) \widetilde{A}_{r-1}^{(0)}(t/n_{r};q)
-
\frac{t}{\varphi(q)}
\sum_{\substack{\mathcal{J}(\textbf{n}_{r-1}) \leq\sqrt{t}\\ (\mathcal{J}(\textbf{n}_{r-1}),q)>1}}
\frac{\Lambda(\textbf{n}_{r-1})}{\mathcal{J}(\textbf{n}_{r-1})}   \nonumber\\
&+
\int_{2}^{\sqrt{t}} \psi(u) \frac{t\widetilde{A}_{r-1}^{(1)}(t/u;q)}{u^2}\,du
+
\frac{1}{\varphi(q)}
\int_{2}^{\sqrt{t}} \frac{t \widetilde{A}_{r-1}(u)}{u^2} \,du
+
O\bigg(\frac{t}{(\log t)^{C}}\bigg), 	
\label{eq:final_step_Ar3_lambds}
\end{align}
The induction hypothesis and the
Lemmas~\ref{lem:int_log_D_is_strange2_lambda} and~\ref{lem:int_log_D_is_strange3_lambda} with $b=-(r-1)$, $a=0$ and $c_{a-1,0} = c_{a-1,1} =\frac{1}{(r-1)!}$
give
\begin{align}
\int_{2}^{\sqrt{t}} \psi(u) \frac{t\widetilde{A}_{r-1}^{(1)}(t/u;q)}{u^2}\,du
&=
t^{}
\frac{(1-2^{-r})(\log t)^{r}}{r!\varphi(q)}
%+
%t^{n-1}\frac{\widetilde{P}_{a,n}(\log\log t)}{(\log t)^{b-1}} 
+
O(t(\log t)^{r-1}),
\\
\frac{1}{\varphi(q)}
\int_{2}^{\sqrt{t}} \frac{t \widetilde{A}_{r-1}(u)}{u^2} \,du
&= 
t \frac{2^{-r}(\log t)^{t}}{r!\varphi(q)} 
+
O(t(\log t)^{r-1}).
\end{align}
Combining both equations gives \eqref{eq:As_aympt_leading_only_q_and_O:uniform_lambda}.
Similarly, Lemmas~\ref{lem:int_log_D_is_strange2_lambda} and~\ref{lem:int_log_D_is_strange3_lambda} show that the pseudo-derivatives  of these two integrals give \eqref{eq:find_a_cool_name_0_1_lambda} for  $n=1$ and \eqref{eq:find_a_cool_name2_lambda} for $n\geq2$.
It remains to show that the contribution of the remaining terms is of lower order.
Since this computation is very similar to the computations for $A_r$, we omit the details.
This completes the proof of Theorem~\ref{thm:Ar_is_strange_lambda}.

\section{Non-principal major arcs}
\label{sec:upper-bound-for-the-real-part-of-phimathbbpr-and-philambdar}
\begin{color}{black}
In this section, we use the Theorems~\ref{thm:Ar_is_strange} and~\ref{thm:Ar_is_strange_lambda} to establish an upper bound for the real part of $\Phi_{\Pri^r}$ and of $\Phi_{\Lambda^{*r}}$ on the non-principal major arcs $\mathfrak{M}(q,a)$ with $2\leq q\leq Q$, see Lemma~\ref{lem:culminationnonprincipal}.
\end{color}

We use in this section the same vector notation as in Section~\ref{sec:sum_r_primes}, see also \eqref{eq:vector_notation_bold_p}.
\begin{lemma} 
\label{lem:main_sum_exp_gamma}
Let $\gamma = \gamma_1+i\gamma_2$ with $\gamma_1 >0$ and $\gamma_2 \ll \gamma_1 (\log(1/\gamma_1))^A$ for some $A>0$ as $\gamma\to 0$. 
Moreover let $q,\ell\in\N$ with $q \ll (\log(1/\gamma_1))^A$ and $(\ell,q)=1$, and define
\begin{align}
U_r(\gamma, \ell, q) 
&:= 
\sum_{\mathcal{J}(\textbf{p}_r)\equiv \ell \bmod q} \exp(-\gamma\mathcal{J}(\textbf{p}_r) ), 
\label{eq:def_U_gamma_ell_q} \\
\widetilde{U}_r(\gamma, \ell, q) 
&:= 
\sum_{\mathcal{J}(\textbf{n}_r)\equiv \ell \bmod q} \Lambda(\textbf{n}_r) \exp(-\gamma\mathcal{J}(\textbf{n}_r) ). 
\label{eq:def_U_tilde_gamma_ell_q_lambda}     
\end{align}
Then one has that
\begin{align*}
U_r(\gamma, \ell, q)
&=
\frac{r}{\varphi(q)}\frac{(\log\log(1/\gamma_1))^{r-1}}{\gamma \log(1/\gamma_1)}
+
O\left(\frac{(\log\log q)(\log\log(1/\gamma_1))^{r-2}}{\varphi(q)\gamma_1\log (1/\gamma_1)}\right),\\
\widetilde{U}_r(\gamma, \ell, q) 
&= 
\frac{(\log(1/\gamma_1))^{r-1}}{(r-1)!\varphi(q)\gamma^{}}
+ 
O\left(\frac{(\log\log q)(\log(1/\gamma_1))^{r-2}}{\gamma_1}\right),  
\end{align*}
where $\varphi(q)$ is the Euler totient function.
\end{lemma}
\begin{proof}
We begin with $U_r(\gamma, \ell, q)$.
First, we apply Abel's summation to $U_r(\gamma, \ell, q)$.
For this define 
\begin{align}
a_n 
= 
\operatorname{card}
\{\textbf{p}_r\in\Pri^r \,:\, n= \mathcal{J}(\textbf{p}_r) \ \text{ and } \ n \equiv \ell \modu q\}.
\end{align}
Then $\sum_{n\leq t} a_n =A_r(t;q,\ell) $ with $A_r(t;q,\ell)$ as in \eqref{eq:Ar_with_vector}.
Moreover, set $f(n) = \exp(-\gamma n)$. 
Then \eqref{eq:find_a_cool_name_0_1} implies that $A_r(t;q,\ell)f(t)\to 0$ as $t\to\infty$ and thus 
\begin{align}
U_r(\gamma, \ell, q) 
=
\sum_{n=1}^\infty a_n f(n)
&= 
A_r(t;q,\ell)f(t)\big|_{t=2}^{\infty} - \int_2^\infty A_r(t;q,\ell) f'(t)dt \nonumber\\
&=  
\gamma \int_2^\infty A_r(t;q,\ell) \exp(-\gamma t)dt.
\label{eq:U_with_A}
\end{align}
%We used in the last equality that $A(x) = 0$ for $x\leq 2$ and $A(x)\leq 2x$ otherwise.
Inserting that $A_r(t;q,\ell)$ is strange by Theorem~\ref{thm:Ar_is_strange}, we see that
\begin{align*}
U_r(\gamma, \ell, q)
=
\gamma \int_2^\infty A_r^{(0)}(t;q) \exp(-\gamma t)dt
+ O\left(\gamma \int_2^\infty t (\log t)^{-C}\exp(-\gamma t)dt \right)
\end{align*}
where $C\geq C_0>0$ can be chosen arbitrarily large.
Here we choose $C> 2A+2$.
Lemma~\ref{lem:dirk'slemma2}, the assumption on $\gamma_2$ and the fact that $\varphi(q)\leq q\ll (\log(1/\gamma_1))^A$ imply that
\begin{align*}
\left|\gamma \int_2^\infty t (\log t)^{-C}\exp(-\gamma t)\,dt\right|
&\leq
|\gamma| \int_2^\infty t (\log t)^{-C}\exp(-\gamma_1 t)\,dt
\ll
\frac{|\gamma|}{\gamma_1^2 (\log(1/\gamma_1))^C}\\
&\ll
\frac{2\gamma_1 (\log(1/\gamma_1))^A}{\gamma_1^2 (\log(1/\gamma_1))^{2A+2}}
\ll
\frac{1}{\varphi(q) \gamma_1 (\log(1/\gamma_1))^2}.
\end{align*}
This then implies that
\begin{align*}
U_r(\gamma, \ell, q)
=
\gamma \int_2^\infty A_r^{(0)}(t;q) \exp(-\gamma t)dt
+O\left(\frac{1}{\varphi(q)\gamma_1 (\log(1/\gamma_1))^2}\right).
\end{align*}
Partial integration, inserting that $A_r^{(0)}(t;q)$ is strange and using Lemma~\ref{lem:dirk'slemma2} give
\begin{align*}
\gamma \int_2^\infty A_r^{(0)}(t;q) \exp(-\gamma t)dt
&=
\int_2^\infty (A_r^{(0)}(t;q))' \exp(-\gamma t)dt +O(1)\\
&=
\int_2^\infty A_r^{(1)}(t;q) \exp(-\gamma t)dt 
+
O\left(\int_2^\infty (\log t)^{-C}\exp(-\gamma t)dt \right)\\
&=
\int_2^\infty A_r^{(1)}(t;q) \exp(-\gamma t)dt 
+
O\left(\frac{1}{\gamma_1\log^C(1/\gamma_1)}  \right).
\end{align*}		
Further, \eqref{eq:find_a_cool_name_0_1} and Lemma~\ref{lem:dirk'slemma2} give
\begin{align*}
\int_2^\infty A_r^{(1)}(t;q) \exp(-\gamma t)dt
&=
\int_2^\infty \left(
\frac{r}{\varphi(q)}\frac{(\log\log t)^{r-1}}{\log t } + O\left(\frac{(\log\log q)(\log\log t)^{r-2}}{\varphi(q)\log t}\right)\right)e^{-\gamma t}dt\\
&=
\frac{r}{\varphi(q)}\frac{(\log\log(1/\gamma_1))^{r-1}}{\gamma \log(1/\gamma_1)}
+
O\left(\frac{(\log\log q)(\log\log(1/\gamma_1))^{r-2}}{\varphi(q)\gamma_1\log (1/\gamma_1)}\right).
\end{align*}
Combining everything finally completes the proof.	

The computations for $\widetilde{U}_r(\gamma, \ell, q)$ are very similar.
We use here 
\begin{align}
\widetilde{a}_n 
= 
\operatorname{card}
\{\textbf{n}_r\in\N^r \,:\, n= \mathcal{J}(\textbf{n}_r) \ \text{ and } \ n \equiv \ell \modu q\}
\end{align}
and get using the same arguments as above, we get
\begin{align*}
\widetilde{U}_r(\gamma, \ell, q)
&=
\sum_{n=1}^\infty \widetilde{a}_n f(n)
%	=
%	\gamma \int_2^\infty \widetilde{A}_r(t;q,\ell) \exp(-\gamma t)dt
%	&=
%	\gamma \int_2^\infty \widetilde{A}_r^{(0)}(t;q) \exp(-\gamma t)dt
%	+ O\left(\frac{1}{\varphi(q)\gamma_1 (\log(1/\gamma_1))^2}\right)\\
=
\int_2^\infty \widetilde{A}_r^{(1)}(t;q) \exp(-\gamma t)dt
+ O\left(\frac{1}{\varphi(q)\gamma_1 (\log(1/\gamma_1))^2}\right).
\end{align*}
Furthermore, we get with \eqref{eq:find_a_cool_name_0_1_lambda} and Lemma~\ref{lem:dirk'slemma2} that
\begin{align*}
\int_2^\infty A_r^{(1)}(t;q) \exp(-\gamma t)dt
&=
\int_2^\infty
\left(\frac{(\log t)^{r-1}}{(r-1)!\varphi(q)} + O\left(\frac{(\log\log q)(\log t)^{r-2}}{\varphi(q)}\right)\right)  \exp(-\gamma t)\,dt\\
&=
\frac{(\log(1/\gamma_1))^{r-1}}{(r-1)!\varphi(q)\gamma^{}}
+ 
O\left(\frac{(\log\log q)(\log(1/\gamma_1))^{r-2}}{\gamma_1}\right).    
\end{align*}
This completes the proof.
\end{proof}

Now we can employ Lemma~\ref{lem:main_sum_exp_gamma} to write an asymptotic expression for $\Phi_{\Pri_2}(\rho \ee(\alpha))$.
For this, we require the Ramanujan sum 
\begin{align}
S^{*}(q,a)
:=
\sum_{\substack{1 \le \ell \le q\\ (\ell,q)=1}} \ee\left(\frac{a\ell}{q}\right),% \quad \textnormal{and} \quad q_j=q/(q,j)
\label{eq:def_S*}
\end{align}
which is a key step in the bound we are looking for  the non-principal major arcs.
We have

\begin{lemma}
\label{lem:Vaugah_constanot_major}
Let $a\in\Z$, $q\in\N$ with $(a,q)=1$ and  $q\leq (\log X)^A$.
Further, we write $q_j:=q/(q,j)$ and  $a_j:=aj/(q,j)$ for all $j\in\N$.
We then have
\begin{align}
\sum_{j\leq\sqrt{X}}\frac{S^{*}(q_j,a_j)}{j^2\varphi(q_j)}
=
\zeta(2)\frac{\prod_{p|q}(-p)}{q^2}
+O(X^{-1/2}).
\label{eq:Phi_on_major_pain_less}
\end{align}
\end{lemma}
\begin{proof}
It is well known that $S^{*}(q_j,a_j)= (-1)^j\mu(q_j)$.
The remaining steps are straightforward and we thus omit them.
\end{proof}

\begin{lemma}\label{lemma5.1}
Let $\alpha\in\R$ and $A>0$ be given.
Further, let $a\in\Z$, $q\in\N$ with 
\begin{align*}
(a,q)=1, \quad q\leq (\log X)^A \quad \text{ and } \quad \left|\alpha - \frac{a}{q}\right|\leq q^{-1}X^{-1}(\log X)^A.
\end{align*}
Then there exists $X_0(A)$ such that we have for all $X>X_0(A)$ 
\begin{align}
\Phi_{\Pri_r}(\rho \ee(\alpha))
&=
\frac{\zeta(2)rX(\log\log X)^{r-1}}{(1-2\pi i (\alpha-\frac{a}{q}) X)\log X}
\frac{\prod_{p|q}(-p)}{q^2}
+O\left(\frac{X(\log\log q)(\log\log X)^{r-2}}{\log X}
\right).
\label{eq:Phi_on_major_pain}
\end{align}
\end{lemma}

\begin{proof}
We define $\gamma = \frac{1}{X}- 2\pi i \beta$ with $\beta=\alpha-\frac{a}{q}$.
Then \eqref{eq:Phionminor} with $J=\sqrt{X}$ implies that 	
\begin{align}\label{firstexpression}
\Phi_{\Pri_r}\left(\rho \ee\left(\alpha\right)\right)
&=
%	\Phi_{\Pri_r}\left(\rho e\left(\frac{a}{q}+\beta\right)\right)
%	=
\sum_{j=1}^{\sqrt{X}} \frac{1}{j} \sum_{\textbf{p}_r\in\Pri^r} e^{-j\mathcal{J}(\textbf{p}_r)/X} \ee(j \mathcal{J}(\textbf{p}_r) \alpha)
+O(X^{1/2+\epsilon})
\nonumber\\
&=
\sum_{j=1}^{\sqrt{X}}\frac{1}{j} \sum_{\textbf{p}_r\in\Pri^r} e^{-j\mathcal{J}(\textbf{p}_r)\gamma}\ee\left(\frac{ja \mathcal{J}(\textbf{p}_r)}{q}\right)+O(X^{1/2+\epsilon}).
%&=
%\sum_{j=1}^{\sqrt{X}}\frac{1}{j} \sum_{\textbf{p}_r\in\Pri^r} e\left(\frac{ja_j \mathcal{J}(\textbf{p}_r)}{q_j}\right)\exp(-j\mathcal{J}(\textbf{p}_r)\gamma)+O(X^{1/2}).
\end{align}
We now replace $a$ and $q$ by $a_j$ and $q_j$ with $a_j$, $q_j$ as in Lemma~\ref{lem:Vaugah_constanot_major} and split the sum into the case where all $p_i$ are coprime to $q_j$ and the case where at least one $p_i$ divides $q_j$.
In formulae
\begin{align}
S_1
&:=
\sum_{j=1}^{\sqrt{X}}\frac{1}{j} \sum_{(\mathcal{J}(\textbf{p}_r),q_j)=1} 
\ee\left(\frac{a_j \mathcal{J}(\textbf{p}_r)}{q_j}\right)
\exp(-j\mathcal{J}(\textbf{p}_r)\gamma),\\
S_2
&:=
\sum_{j=1}^{\sqrt{X}}\frac{1}{j} \sum_{(\mathcal{J}(\textbf{p}_r),q_j)>1} 
\ee\left(\frac{a_j \mathcal{J}(\textbf{p}_r)}{q_j}\right)
\exp(-j\mathcal{J}(\textbf{p}_r)\gamma).
\end{align}
We will see that the main contribution comes from $S_1$.
We begin by giving an upper bound for $S_2$. 
For this we write $\textbf{p}_r=(\textbf{p}_{r-1},p_r)$. 
Using that $S_2$ is symmetric in all $p_j$, and then applying Lemma~\ref{lem:main_sum_exp_gamma} to the sum over $\textbf{p}_{r-1}$ with $\gamma_1 = \frac{jp_r}{X}$, we get
%\begin{align}
%S_2
%&\ll
%\sum_{j=1}^{\sqrt{X}}\frac{1}{j} \sum_{p_r|q_j}\sum_{\textbf{p}_{r-1}\in\Pri^{r-1}} \exp\left(-\frac{jp_r}{X}\mathcal{J}(\textbf{p}_{r-1})\right)
%\ll
%\sum_{j=1}^{\sqrt{X}}\frac{1}{j} \sum_{p_r|q_j}
%\frac{X}{jp_r}
%\frac{(\log\log(X/(jp_r)))^{r-2}}{\log(X/(jp_r))}\nonumber\\
%&\ll
%\frac{X(\log\log X)^{r-2}}{\log X}
%\sum_{j=1}^{\sqrt{X}}\frac{1}{j^2} \sum_{p_r|q}
%\frac{1}{p_r }
%\ll
%\frac{X(\log\log X)^{r-2}}{\log X}\sum_{p_r|q}
%\frac{1}{p_r }.
%\end{align}
\begin{align*}
S_2
&\ll
\sum_{j=1}^{\sqrt{X}}\frac{1}{j} \sum_{p_r|q_j}\sum_{\textbf{p}_{r-1}\in\Pri^{r-1}} \exp\left(-\frac{jp_r}{X}\mathcal{J}(\textbf{p}_{r-1})\right)
\ll
\sum_{j=1}^{\sqrt{X}}\frac{1}{j} \sum_{p_r|q_j}
\frac{X}{jp_r}
\frac{(\log\log(X/(jp_r)))^{r-2}}{\log(X/(jp_r))}.
\end{align*}
Since $q\ll (\log X)^A$, we get that $\frac{1}{4}\log X\leq \log(X/(jp_r))\leq \log X$ and thus
\begin{align*}
S_2
&\ll
\frac{X(\log\log X)^{r-2}}{\log X}
\sum_{j=1}^{\sqrt{X}}\frac{1}{j^2} \sum_{p_r|q}
\frac{1}{p_r }
\ll
\frac{X(\log\log X)^{r-2}}{\log X}\sum_{p_r|q}
\frac{1}{p_r}.
\end{align*}
Denote by $m$ the number of distinct prime factors of $q$. Then Mertens' theorem implies
\begin{align*}
\sum_{\substack{p_r|q}}\frac{1}{p_r}
\leq 
\sum_{p\leq m} \frac{1}{p}
= 
\log\log m + M + o(1)
=
O(\log\log m).
\end{align*}
Since $q_j$ has at most $O(\log q_j)$ primes factors and $q_j \ll (\log X)^A$, we deduce that
\begin{align}
S_2
&\ll
\frac{X(\log\log X)^{r-2}\log\log\log\log X}{\log X}.
\label{eq:upper_bound_S2}
\end{align}
We look next at the inner sum of $S_1$ and combine all terms with $\mathcal{J}(\textbf{p}_r) = \ell \modu q_j$.
Since all $p_i$ are coprime to $q_j$, we only need to use those $\ell$ coprime to $q_j$.
Using Lemma~\ref{lem:main_sum_exp_gamma} with $\gamma_1= j/X$, we get that
\begin{align*}
\sum_{(\mathcal{J}(\textbf{p}_r),q_j)=1} 
&\ee\left(\frac{a_j \mathcal{J}(\textbf{p}_r)}{q_j}\right)\exp(-j\mathcal{J}(\textbf{p}_r)\gamma)
=
\sum_{\substack{\ell = 1\\ (\ell,q_j)=1}}^{q_j} 
\ee\left(\frac{a_j\ell}{q_j}\right) 
\sum_{\mathcal{J}(\textbf{p}_r) \equiv \ell \modu q_j} \exp(-j\mathcal{J}(\textbf{p}_r)\gamma)\\
=\,&
\sum_{\substack{\ell = 1\\ (\ell,q_j)=1}}^{q_j} 
\ee\left(\frac{a_j\ell}{q_j}\right) 
\left(\frac{r}{\varphi(q)}\frac{(\log\log(X/j))^{r-1}}{j\gamma \log(X/j)}
+
O\left(\frac{X(\log\log q)(\log\log(X/j))^{r-2}}{j\varphi(q)\log (X/j)}\right)\right)\\
=\,&
\sum_{\substack{\ell = 1\\ (\ell,q_j)=1}}^{q_j} 
\ee\left(\frac{a_j\ell}{q_j}\right) 
\frac{r}{\varphi(q)}\frac{(\log\log(X/j))^{r-1}}{j\gamma \log(X/j)}
+
O\left(\frac{X(\log\log q)(\log\log X)^{r-2}}{j\log X}\right).
\end{align*}
We used on the  last line that the number of summands is at most $\varphi(q)$ 
and that $\frac{1}{2}\log X\leq \log(X/j)\leq \log X$.
Inserting this into $S_1$, we get
\begin{align}
S_1
&=
\sum_{j=1}^{\sqrt{X}}\frac{1}{j^2}
\sum_{\substack{\ell = 1\\ (\ell,q_j)=1}}^{q_j} 
\ee\left(\frac{a_j\ell}{q_j}\right) 
\frac{r}{\varphi(q)}\frac{(\log\log(X/j))^{r-1}}{\gamma \log(X/j)}
+
O\bigg(\frac{X(\log\log q)(\log\log X)^{r-2}}{\log X}\sum_{j=1}^{\sqrt{X}}\frac{1}{j^2}
\bigg)\nonumber\\
&=
\sum_{j=1}^{\sqrt{X}}\frac{1}{j^2}
\sum_{\substack{\ell = 1\\ (\ell,q_j)=1}}^{q_j} 
\ee\left(\frac{a_j\ell}{q_j}\right) 
\frac{r}{\varphi(q)}\frac{(\log\log(X/j))^{r-1}}{\gamma \log(X/j)}
+
O\left(\frac{X(\log\log q)(\log\log X)^{r-2}}{\log X}
\right).
\label{eq:needs_a_good_name} 
\end{align}
Furthermore, using that $\log j\leq \frac{1}{2} \log X$, we see that
\begin{align*}
\frac{1}{\log(X/j)}
&=
\frac{1}{\log X -\log j}
=
\frac{1}{\log X} \sum_{k=0}^{\infty} \left(\frac{\log j}{\log X}\right)^k
=
\frac{1}{\log X} +O\left( \frac{\log j}{(\log X)^2}  \right),
\end{align*}
as well as
\begin{align*}
\log\log(X/j)
=
\log\log X +O\left(  \frac{\log j}{\log X}  \right).
\end{align*}
Inserting these two identities into \eqref{eq:needs_a_good_name} and combining this with \eqref{eq:upper_bound_S2} 
and using the definition of $S^{*}(q_j,a_j)$ gives
\begin{align*}
\Phi_{\Pri_r}(\rho \ee(\alpha))
=
\frac{r(\log\log X)^{r-1}}{\gamma\log X}\sum_{j\leq\sqrt{X}}\frac{S^{*}(q_j,a_j)}{j^2\varphi(q_j)}+O\left(\frac{X(\log\log q)(\log\log X)^{r-2}}{\log X}
\right).
%	\label{eq:Phi_on_major_pain}
\end{align*}
Inserting that $\gamma = \frac{1}{X}- 2\pi i (\alpha-\frac{a}{q})$ and Lemma~\ref{lem:Vaugah_constanot_major} completes the proof.
\end{proof}
\begin{lemma}
\label{lemma5.1_lambda}
Let $\alpha\in\R$ and $A>0$ be given.
Further, let $a\in\Z$, $q\in\N$ with 
\begin{align*}
(a,q)=1, \quad q\leq (\log X)^A \quad \text{ and } \quad \left|\alpha - \frac{a}{q}\right|\leq q^{-1}X^{-1}(\log X)^A.
\end{align*}
Then there exists $X_0(A)$ such that we have for all $X>X_0(A)$ 
\begin{align}
\Phi_{\Lambda^{*r}}(\rho \ee(\alpha))
&=
\frac{\zeta(2)X(\log X)^{r-1}}{(r-1)!(1-2\pi i (\alpha-\frac{a}{q}) X)}\frac{\prod_{p|q}(-p)}{q^2}+O\left(X(\log\log q)(\log(X))^{r-2}
\right).
\label{eq:Phi_on_major_pain_lambda}
\end{align}
\end{lemma}

\begin{proof}
The proof of this lemma is similar to the proof of Lemma~\ref{lemma5.1_lambda}.
We thus give only an overview.

We define $\gamma = \frac{1}{X}- 2\pi i \beta$ with $\beta=\alpha-\frac{a}{q}$.
Then \eqref{eq:Phionminor_Lambda} with $J=\sqrt{X}$ implies that 	
\begin{align}\label{firstexpression_lambda}
\Phi_{\Lambda^{*r}}\left(\rho \ee\left(\alpha\right)\right)
&=
\Phi_{\Lambda^{*r}}\left(\rho \ee\left(\frac{a}{q}+\beta\right)\right)
=
S_1+S_2+O(X^{1/2+\epsilon})
\end{align}
with
\begin{align}
S_1
&:=
\sum_{j=1}^{\sqrt{X}}\frac{1}{j} \sum_{(\mathcal{J}(\textbf{n}_r),q_j)=1} \Lambda(\textbf{n}_r)
\ee\left(\frac{a_j \mathcal{J}(\textbf{n}_r)}{q_j}\right)
\exp(-j\mathcal{J}(\textbf{n}_r)\gamma),\\
S_2
&:=
\sum_{j=1}^{\sqrt{X}}\frac{1}{j} \sum_{(\mathcal{J}(\textbf{n}_r),q_j)>1} \Lambda(\textbf{n}_r)
\ee\left(\frac{a_j \mathcal{J}(\textbf{n}_r)}{q_j}\right)
\exp(-j\mathcal{J}(\textbf{n}_r)\gamma),
\end{align}
and $a_j$, $q_j$ as in Lemma~\ref{lem:Vaugah_constanot_major}.
We begin by giving an upper bound for $S_2$. 
For this we write $\textbf{n}_r=(\textbf{n}_{r-1},n_r)$. 
Using that $S_2$ is symmetric in all $n_j$ and writing $n_r=p^m$, we get
%\begin{align}
%S_2
%&\ll
%\sum_{j=1}^{\sqrt{X}}\frac{1}{j} \sum_{p_r|q_j}\sum_{\textbf{p}_{r-1}\in\Pri^{r-1}} \exp\left(-\frac{jp_r}{X}\mathcal{J}(\textbf{p}_{r-1})\right)
%\ll
%\sum_{j=1}^{\sqrt{X}}\frac{1}{j} \sum_{p_r|q_j}
%\frac{X}{jp_r}
%\frac{(\log\log(X/(jp_r)))^{r-2}}{\log(X/(jp_r))}\nonumber\\
%&\ll
%\frac{X(\log\log X)^{r-2}}{\log X}
%\sum_{j=1}^{\sqrt{X}}\frac{1}{j^2} \sum_{p_r|q}
%\frac{1}{p_r }
%\ll
%\frac{X(\log\log X)^{r-2}}{\log X}\sum_{p_r|q}
%\frac{1}{p_r }.
%\end{align}
\begin{align*}
S_2
%&\ll
%\sum_{j=1}^{\sqrt{X}}\frac{1}{j} \sum_{(n_r,q_j)>1}\Lambda(n_r)\sum_{\textbf{n}_{r-1}\in\N^{r-1}} \Lambda(\textbf{n}_{r-1}) \exp\left(-\frac{jn_r}{X}\mathcal{J}(\textbf{n}_{r-1})\right)\\
&\ll
\sum_{j=1}^{\sqrt{X}}\frac{1}{j} \sum_{p|q_j}(\log p)\sum_{m=1}^{\infty}\sum_{\textbf{n}_{r-1}\in\N^{r-1}} \Lambda(\textbf{n}_{r-1}) \exp\left(-\frac{jp^m}{X}\mathcal{J}(\textbf{n}_{r-1})\right).
\end{align*}
Lemma~\ref{lem:main_sum_exp_gamma} shows that there exist a constant $C_0>1$
such that if $\frac{X}{jp^m}\geq C_0$ then 
\begin{align*}
\sum_{\textbf{n}_{r-1}\in\N^{r-1}} \Lambda(\textbf{n}_{r-1}) \exp\left(-\frac{jp^m}{X}\mathcal{J}(\textbf{n}_{r-1})\right)
\leq 
2X\frac{(\log(X/(jp^m)))^{r-1}}{(r-1)!\varphi(q)jp^m}.
\end{align*}
We now write $S_2=S_{2,1}+S_{2,2}$, where $S_{2,1}$ corresponds the sum over all summands with  $\frac{X}{jp^m}\geq C_0$.
Using that $0\leq \log(X/(jp^m)) \leq \log X$ if $\frac{X}{jp^m}\geq C_0$, we get
\begin{align*}
S_{2,1}
&\ll
\sum_{j=1}^{\sqrt{X}}\frac{1}{j} \sum_{p|q}(\log p)\sum_{m\leq \log X}
X\frac{(\log(X/(jp^m)))^{r-1}}{jp^m}
\leq 
X (\log X)^{r-1} \sum_{j=1}^{\infty}\frac{1}{j^2} \sum_{p|q}\sum_{m=1}^\infty\frac{(\log p)}{p^m}\\
&\ll
X (\log X)^{r-1} \sum_{p|q}\frac{\log p}{p}
\ll 
X (\log X)^{r-1}  (\log q)^2
\ll 
X (\log X)^{r-1}  (\log\log X)^2.
\end{align*}
We look next as $S_{2,2}$. 
For a given $p$, we have to sum over all $m$ such that $\frac{jp^m}{X}\geq \frac{1}{C_0}$.
The smallest value that $\frac{jp^m}{X}$ can take is of course $\frac{1}{C_0}$.
The second smallest value is at least $\frac{p}{C_0}$, the third smallest value is at least $\frac{p^2}{C_0}$, and so on.
Thus
\begin{align*}
S_{2,2}
&\ll
\sum_{j=1}^{\sqrt{X}}\frac{1}{j} \sum_{p|q_j}(\log p)\sum_{m'=1}^{\infty}\sum_{\textbf{n}_{r-1}\in\N^{r-1}} \Lambda(\textbf{n}_{r-1}) \exp\bigg(-\frac{p^{m'}}{C_0}\mathcal{J}(\textbf{n}_{r-1})\bigg)\\
&\ll
\sum_{j=1}^{\sqrt{X}}\frac{1}{j}
\sum_{p|q_j}(\log p)
\sum_{m'=1}^{\infty}\sum_{\textbf{n}_{r-1}\in\N^{r-1}} \exp\left(-\frac{m'}{2C_0}\mathcal{J}(\textbf{n}_{r-1})\right)\\
&\ll
\bigg(\sum_{j=1}^{\sqrt{X}}\frac{1}{j}\bigg) 
\bigg(\sum_{p|q_j}\log p\bigg)
\sum_{\textbf{n}_{r-1}\in\N^{r-1}} \exp\left(-\frac{1}{2C_0}\mathcal{J}(\textbf{n}_{r-1})\right)\\
&\ll (\log X) (\log q)^2 \ll  (\log X) (\log \log X)^2.
\end{align*}
This implies that $S_2 \ll (\log X) (\log q)^2 \ll  (\log X) (\log \log X)^2$, which is of lower order.

We look next at the inner sum of $S_1$ and combine all terms with $\mathcal{J}(\textbf{p}_r) = \ell \modu q_j$.
Since all $p_i$ are coprime to $q_j$, we only need to use those $\ell$ coprime to $q_j$.
Using Lemma~\ref{lem:main_sum_exp_gamma} with $\gamma_1= j/X$, 
we get with a similar argument as in Lemma~\ref{lemma5.1}
\begin{align*}
\sum_{(\mathcal{J}(\textbf{n}_r),q_j)=1} \Lambda(\textbf{n}_r)
&\ee\left(\frac{a_j \mathcal{J}(\textbf{n}_r)}{q_j}\right)
\exp(-j\mathcal{J}(\textbf{n}_r)\gamma)\\
%=
%\sum_{\substack{\ell = 1\\ (\ell,q_j)=1}}^{q_j} 
%e\left(\frac{a_j\ell}{q_j}\right) 
%\sum_{\mathcal{J}(\textbf{n}_r)= \ell \modu q_j}
%\Lambda(\textbf{n}_r)\exp(-j\mathcal{J}(\textbf{n}_r)\gamma)\\
&=
\sum_{\substack{\ell = 1\\ (\ell,q_j)=1}}^{q_j} 
\ee\left(\frac{a_j\ell}{q_j}\right) 
\left(\frac{(\log(X/j))^{r-1}}{(r-1)!\gamma\varphi(q)j} \right)
+
O\left(\frac{X(\log\log q)(\log(X/j))^{r-2}}{j}\right).
\end{align*}

We used on the  last line that  $\frac{1}{2}\log X\leq \log(X/j)\leq \log X$.
Inserting this into $S_1$, we get
\begin{align}
S_1
&=
\sum_{j=1}^{\sqrt{X}}\frac{1}{j^2}
\sum_{\substack{\ell = 1\\ (\ell,q_j)=1}}^{q_j} 
\ee\left(\frac{a_j\ell}{q_j}\right) 
\left(\frac{(\log(X/j))^{r-1}}{(r-1)!\gamma\varphi(q)} \right)
+
O\left(X(\log\log q)(\log X)^{r-2}\right).
\label{eq:needs_a_good_name_lambda} 
\end{align}
Furthermore, we have
\[
(\log(X/j))^{r-1} = (\log X )^{r-1} + O\left((\log j )^{r-1}(\log X )^{r-2}\right).
\]
Inserting this identity into \eqref{eq:needs_a_good_name_lambda} and combining it with the estimate for $S_2$ 
and using the definition of $S^{*}(q_j,a_j)$ we get
\begin{align*}
\Phi_{\Lambda^{*r}}(\rho \ee(\alpha))
=
\frac{(\log X)^{r-1}}{(r-1)!\gamma}\sum_{j\leq\sqrt{X}}\frac{S^{*}(q_j,a_j)}{j^2\varphi(q_j)}
+
O\left(X(\log\log q)(\log X))^{r-2}
\right).
%	\label{eq:Phi_on_major_pain}
\end{align*}
Inserting Lemma~\ref{lem:Vaugah_constanot_major} completes the proof.
\end{proof}
\begin{corollary}
\label{cor:lemma5.1}
We have for $\alpha\in\mathfrak{M}(1,0)$ with $\mathfrak{M}(a,q)$ as in \eqref{eq:def_M(q,a)} that 
\begin{align*}
\Phi_{\Pri_r}(\rho \ee(\alpha))
&=
\frac{\zeta(2)rX(\log\log X)^{r-1}}{(1-2\pi i \alpha X)\log X}+O\left(\frac{X(\log\log X)^{r-2}}{\log X}
\right),\\
\Phi_{\Lambda^{*r}}(\rho \ee(\alpha))
&=
\frac{\zeta(2)X(\log X)^{r-1}}{(r-1)!(1-2\pi i \alpha X)}+O\left(X(\log\log X)^{r-2}\right).
\end{align*}
\end{corollary}
\begin{proof}
We have $\alpha\in\mathfrak{M}(1,0)$ is equivalent to $|\alpha|\leq X^{-1}(\log X)^A$.
Thus we can use Lemma~\ref{lemma5.1} with $q=1$ and $a=0$. 
We have in this case $q_j =1$ for all $j$ and $S^{*}(1,0) =1$, which immediately implies the statement of the corollary.
\end{proof}	
The previous results from this section now allow us to conclude that the non-principal major arcs will not contribute to the main term and therefore will be absorbed in the error term.
\begin{lemma} 
\label{lem:culminationnonprincipal}
Let $\mathfrak{M}(q,a)$ be as in \eqref{eq:def_M(q,a)} with $2\leq q\leq (\log X)^A$ and $(a,q)=1$.
We then have for all $\alpha\in\mathfrak{M}(q,a)$ that
\begin{align*}
\Re(\Phi_{\Pri_r}(\rho \ee(\alpha))) 
\leq  
\frac{3}{4}\Phi_{\Pri_r}(\rho)
\ \text{ and } \
\Re(\Phi_{\Lambda^{*r}}(\rho \ee(\alpha))) 
\leq  
\frac{3}{4}\Phi_{\Lambda^{*r}}(\rho).
\end{align*}	
\end{lemma}
\begin{proof}
We first look at $\Phi_{\Pri_r}$
We know from Corollary~\ref{cor:lemma5.1} that 
\begin{align} 
\Phi_{\Pri_r}(\rho)
&=
\frac{\zeta(2)rX(\log\log X)^{r-1}}{\log X}+O\left(\frac{X	(\log\log X)^{r-2}}{\log X}
\right).
\label{eq:lem:culminationnonprincipal1}
\end{align}
Further, we get for $q\geq 2$
\begin{align*}
\bigg|\zeta(2)\frac{\prod_{p|q}(-p)}{q^2} \bigg|
\leq 
\frac{\zeta(2)}{q}
\leq 
\frac{\zeta(2)}{2}.
\end{align*}
Since $\alpha\in\mathfrak{M}(q,a)$ with $q\geq 2$, Lemmas~\ref{lemma5.1} and \ref{lem:Vaugah_constanot_major} imply 
\begin{align}
\Re\left(\Phi_{\Pri_r}(\rho \ee(\alpha))\right)
\leq 
\frac{1}{2}\frac{\zeta(2)rX(\log\log X)^{r-1}}{\log X}+O\left(\frac{X(\log\log q)(\log\log X)^{r-2}}{\log X}
\right).
\end{align}
Combining this estimate with \eqref{eq:lem:culminationnonprincipal1} completes the proof for $\Phi_{\Pri_r}$. 
The proof for $\Phi_{\Lambda^{*r}}$ is (almost) identical and we thus omit it. 
\end{proof}

\section{Main terms and the principal major arcs} \label{sec:main_terms_and_principal_arcs}

\subsection{The method of contour integration}
\label{sec:principalarcs}
In Corollary~\ref{cor:lemma5.1} we have determined the asymptotic behaviour of $\Phi_{\Pri_r}$ and 
$\Phi_{\Lambda^{*r}}$ in the principal arc.
However, this corollary only gives us the leading term and we need a more precise version to prove the main results which are Theorems~\ref{thm:main_Pri_r} and \ref{thm:main_Lambda_r}, 
in particular we also need the behaviour of the derivatives of $\Phi_{\Pri_r}$ and 
$\Phi_{\Lambda^{*r}}$.
We obtain this via contour integration.
We begin with $\Phi_{\Pri_r}$. We have
\begin{theorem}
\label{thm:Phi_asympt_principal}
Let $\rho = e^{-1/X}$ and $r\in\N$ . 
Then there exists a polynomial $P_r$ of degree $r-1$ and leading coefficient $r$ such that we have for all any $m \in \Z_{\ge 0}$ and as $\rho \to 1^{-}$,
\begin{align}
	\bigg(\rho \frac{d}{d\rho}\bigg)^m \Phi_{\mathbb{P}_r}(\rho) 
	&= 
	\frac{\zeta(2)\Gamma(m+1) X^{m+1}P_r(\log\log X)}{\log X} \bigg(1+O\bigg(\frac{1}{\log X}\bigg)\bigg),\\
	\Phi_{\mathbb{P}_r}^{(m)}(\rho) 
	&=
	\frac{\zeta(2)\Gamma(m+1) X^{m+1}P_r(\log\log X)}{\log X} \bigg(1+O\bigg(\frac{1}{\log X}\bigg)\bigg).
\end{align}
\end{theorem}
We can explicitly compute the polynomial $P_r$. 
The proof of Theorem~\ref{thm:Phi_asympt_principal} below gives
\begin{align}
	P_r(y)
	= \frac{1}{\pi} \sum_{n=0}^r \binom{r}{n} (-1)^{r-n} (\mathcal{D}(1))^{r-n}  \sum_{h=0}^n \sum_{k=0}^h \binom{n}{h}\binom{h}{k}(-1)^k \imag[(i \pi)^{n-h}] y^{h-k} \Gamma^{(k)}(1),
	\label{eq:P_r_formula}
\end{align}
%\begin{align}
%	\bigg(\rho \frac{d}{d\rho}\bigg)^m \Phi_{\mathbb{P}_r}(\rho) &= \frac{\zeta(2)}{\pi} \Gamma(m+1) \frac{X^{m+1}}{\log X} \sum_{n=0}^r \binom{r}{n} (-1)^{r-n} (\mathcal{D}(1))^{r-n} \nonumber \\
%	& \quad \times \sum_{h=0}^n \sum_{k=0}^h \binom{n}{h}\binom{h}{k}(-1)^k \imag[(i \pi)^{n-h}] (\log \log X)^{h-k} \Gamma^{(k)}(1) \bigg(1+O\bigg(\frac{1}{\log X}\bigg)\bigg),
%\end{align}
where
\begin{align}
	\mathcal{D}(s)
	:=
	 \log\big((s-1)\zeta(s)\big)-  \widehat{\mathcal{D}}(s) \quad \textnormal{with} \quad  \widehat{\mathcal{D}}(s) = \sum_{j \ge 2} \sum_{p \in \Pri} \frac{1}{p^{js}}.
	 \label{eq:def_D_principal}
\end{align}
%
%\begin{align}
%	\zeta_\mathcal{P}(s) 
%	= 
%	-\log(s-1) +\log\big((s-1)\zeta(s)\big) - \widehat{\mathcal{D}}(s) \quad \textnormal{where} \quad  \widehat{\mathcal{D}}(s) = \sum_{j \ge 2} \sum_{p \in \Pri} \frac{1}{p^{js}}.
%\end{align}
Here $\log$ denotes the principal branch of the logarithm and $\log((s-1)\zeta(s))$ is holomorphic in the standard zero-free region of $\zeta$.
Furthermore, for any $\delta > 0$ we have that $\widehat{\mathcal{D}}(s)$ converges absolutely and uniformly for $\real(s) \ge \frac{1}{2} + \delta$. 
This shows that the expression $\mathcal{D}(1)$ is well-defined.

First we prove an auxiliary lemma. 
\begin{lemma}
\label{lem:Phi_asympt_principal_aux_lemma}
Let $0<a\leq 1$ be given and $F:[0,a]\to\R$ be a smooth function.
We then have for all $h\in\N_0$ that
\begin{align}
\int_{0}^{a} X^{-u} (-\log u)^h F(u)\,du
=
\frac{F(0)}{\log X}\sum_{k=0}^{h} \binom{h}{k} (\log\log X)^{h-k}  (-1)^{k}\Gamma^{(k)}(1)
+
O\left(\frac{(\log\log X)^h}{(\log X)^2}\right),
\label{eq:aux_lem_main_term}
\end{align} 
where $\Gamma^{(n)}$ is the $n$th derivative of the Gamma function.
\end{lemma}
\begin{proof}
Expanding $F$ around $0$ and inserting it into the integral give
\begin{align*}
	\int_{0}^{a} X^{-u} (-\log u)^h F(u)\,du
	&=
	F(0)\int_{0}^{a} X^{-u} (-\log u)^h\,du + 	
	O\bigg(\int_{0}^{a} X^{-u} u(\log u)^h\,du\bigg).
%	\\
%	&=
%	\frac{F(0)}{\log X}\int_{0}^{a} e^{-t} \left(\log\left(\frac{t}{\log X}\right)\right)^k\,dt + 
%	O\left(\frac{(\log\log X)^k}{(\log X)^2}\int_{0}^{a} te^{-t} \left(\log t\right)^k\,dt\right)
\end{align*} 
We first look at the main term.
The variable substitution $t=u\log X$ gives
\begin{align*}
F(0)\int_{0}^{a} X^{-u} (-\log u)^h\,du 
&=
\frac{F(0)}{\log X}\int_{0}^{a\log X} e^{-t} \left(-\log\left(\frac{t}{\log X}\right)\right)^h\,dt\\
&=
\frac{F(0)}{\log X}\sum_{k=0}^{h} \binom{h}{k} (\log\log X)^{h-k} (-1)^{k}\int_{0}^{a\log X} e^{-t} (\log t)^{k}\,dt.
\end{align*} 
Further
\begin{align*}
\int_{0}^{a\log X} e^{-t} (\log t)^{k}\,dt
&=
\int_{0}^{\infty} e^{-t} (\log t)^{k}\,dt - \int_{a\log X}^\infty e^{-t} (\log t)^{k}\,dt=
\Gamma^{(k)}(1) + O(X^{-a/2}).
\end{align*}
Similarly, we get 
\begin{align*}
\int_{0}^{a} X^{-u} u(\log u)^h\,du 
&=
\frac{1}{(\log X)^2}\int_{0}^{a\log X} te^{-t} \left(\log\left(\frac{t}{\log X}\right)\right)^h\,dt\\
&\ll
\frac{(\log\log X)^h}{(\log X)^2}\int_{0}^{\infty} te^{-t} (\log t)^{h}\,dt
\ll \frac{(\log\log X)^h}{(\log X)^2}.
\end{align*} 
Combining everything completes the proof.
\end{proof}
We only used one derivative of $F$ in this proof. 
By using more derivatives of $F$, we can get more terms in the expansion \eqref{eq:aux_lem_main_term}.
With this we can actually get more terms in the expansion in the Theorem~\ref{thm:Phi_asympt_principal}.
%The expression we get is of course the same as the one we gave after the Theorem~\ref{thm:Ar_is_strange}.
%
\begin{proof}[Proof of Theorem~\textnormal{\ref{thm:Phi_asympt_principal}}]
As before, we use here also the vector notation of sums, see \eqref{eq:vector_notation_bold_p}.
Using the definition of $\Phi_{\mathbb{P}_r}(\rho)$ in \eqref{eq:def_Phi_Lambda} and that $\rho=\exp(-1/X)$, we get
\begin{align}
\Phi_{\mathbb{P}_r}(\rho) 
= 
\sum_{j=1}^\infty \frac{1}{j} \sum_{\textbf{p}_r} \rho^{j\mathcal{J}(\textbf{p}_r)}
= 
\sum_{j=1}^\infty \frac{1}{j} \sum_{\textbf{p}_r} e^{-j\mathcal{J}(\textbf{p}_r)/X}.
\end{align}
We then have
\begin{align}
	\bigg(\rho \frac{d}{d\rho} \bigg)^m \Phi_{\mathbb{P}_r}(\rho) 
	= 
	\sum_{j=1}^\infty j^{m-1} \sum_{\textbf{p}_r}\left(\mathcal{J}(\textbf{p}_r)\right)^m e^{-\mathcal{J}(\textbf{p}_r) j/X}.
\end{align}
The Cahen-Mellin formula allows us to write the above expression as the following complex integral
\begin{align} \label{auxintegral}
	\bigg(\rho \frac{d}{d\rho} \bigg)^m \Phi_{\mathbb{P}_r}(\rho) 
	&= 
	\frac{1}{2 \pi i} \int_{(c+m)} X^s \bigg(\prod_{k=1}^r \sum_{p_k} \frac{1}{p_k^{s-m}} \bigg) \bigg(\sum_{j=1}^\infty \frac{1}{j^{s+1-m}} \bigg) \Gamma(s)ds \nonumber \\
	&= 
	\frac{1}{2 \pi i} \int_{(c+m)} X^s (\zeta_{\mathcal{P}}(s-m))^r \zeta(s+1-m) \Gamma(s) ds\nonumber\\
	&= 
	\frac{X^{m}}{2 \pi i} \int_{(c)} X^{s} (\zeta_{\mathcal{P}}(s))^r \zeta(s+1) \Gamma(s+m) ds,
\end{align}
with $c > 1$ and where $\zeta_\mathcal{P}(s)$ is the prime zeta function. 
Now we split the singular part of the prime zeta function by writing
\begin{align}
	\zeta_\mathcal{P}(s) 
	= 
	-\log(s-1) +\log\big((s-1)\zeta(s)\big) - \widehat{\mathcal{D}}(s) \quad \textnormal{where} \quad  \widehat{\mathcal{D}}(s) = \sum_{j \ge 2} \sum_{p \in \Pri} \frac{1}{p^{js}}.
\end{align}
For any $\delta > 0$ we have that $\widehat{\mathcal{D}}(s)$ converges absolutely and uniformly for $\real(s) \ge \frac{1}{2} + \delta$. 
Furthermore, $\log((s-1)\zeta(s))$ is holomorphic in the standard zero-free region of $\zeta$
and $\log$ denotes the principal branch of the logarithm.
Using $\mathcal{D}(s)$ as in \eqref{eq:def_D_principal}, we make the replacement 
\begin{align} \label{decompsitionzetaR}
	(\zeta_{\mathcal{P}}(s))^r 
	= 
	\sum_{n=0}^r \binom{r}{n}(-\log (s-1))^n (\mathcal{D}(s))^{r-n}
\end{align}
and get
\begin{align} \label{auxintegral2}
	\bigg(\rho \frac{d}{d\rho} \bigg)^m \Phi_{\mathbb{P}_r}(\rho) 
	&= 
	\sum_{n=0}^r \binom{r}{n}\frac{X^{m}}{2 \pi i} \int_{(c)} X^{s} (-\log (s-1))^n (\mathcal{D}(s))^{r-n} \zeta(s+1) \Gamma(s+m) ds.
\end{align}
The task is therefore to calculate the contribution of the binomial components. 
A typical term in \eqref{decompsitionzetaR} will have the form
\begin{align}
	\Omega(m, X, \alpha, \beta)
	:= 
	\frac{1}{2 \pi i}\int_{(c)} X^s (-\log (s-1))^\alpha  (\mathcal{D}(s))^\beta \zeta(s+1) \Gamma(s+m) ds.
	\label{eq:def_int_Omega}
\end{align}
for some integers $\alpha$ and $\beta$ such that $\alpha+\beta=r$. 
First, note that the integrand is analytic in the standard zero-free region of $\zeta(s)$
except for the line $\{s= \sigma \, : \, \sigma \le 1\}$, where $\log$ is not defined.
Further, by Stirling's formula, we have in any fixed vertical strip that
$
\Gamma(\sigma+it)\ll e^{-(\frac{\pi}{2}-\epsilon)|t|}
$
as $|t|\to\infty$, where $\epsilon>0$ is arbitrary.
Furthermore, $\zeta(s)$ is at most polynomially growing in any vertical strip.
We now truncate the integral in \eqref{eq:def_int_Omega} at height $T$ with $T = \exp (\sqrt{\log X})$.
Using the above bounds, we get
\begin{align}
	\Omega(m, X, \alpha, \beta)
	= 
	\frac{1}{2 \pi i}\int_{c-iT}^{c+iT} X^s (-\log (1-s))^\alpha  (\mathcal{D}(s))^\beta \zeta(s+1) \Gamma(s+m) ds
	+
	O(X^{-N})
	\label{eq:def_int_Omega2}
\end{align}
where $N\geq 1$ is arbitrary.
Next we complete the remaining curve to the contour as in Figure~\ref{fig:integrationcontour}.
We denote this contour by $\Xi$.
This contour is the boundary of the domain obtained by intersecting the rectangle with corners $\tfrac{3}{4}\pm iT$, $c\pm iT$ 
with the standard zero-free region and adding a keyhole contour with height $\delta>0$ 
and radius $\varepsilon>0$ around the essential singularity at $s=1$. 
This keyhole contour runs clockwise along the top and the bottom of the branch cut of the logarithm located at 
$\{s= \sigma \, : \, \sigma \le 1\}$. 

\begin{figure}[H]
	\includegraphics[scale=0.30]{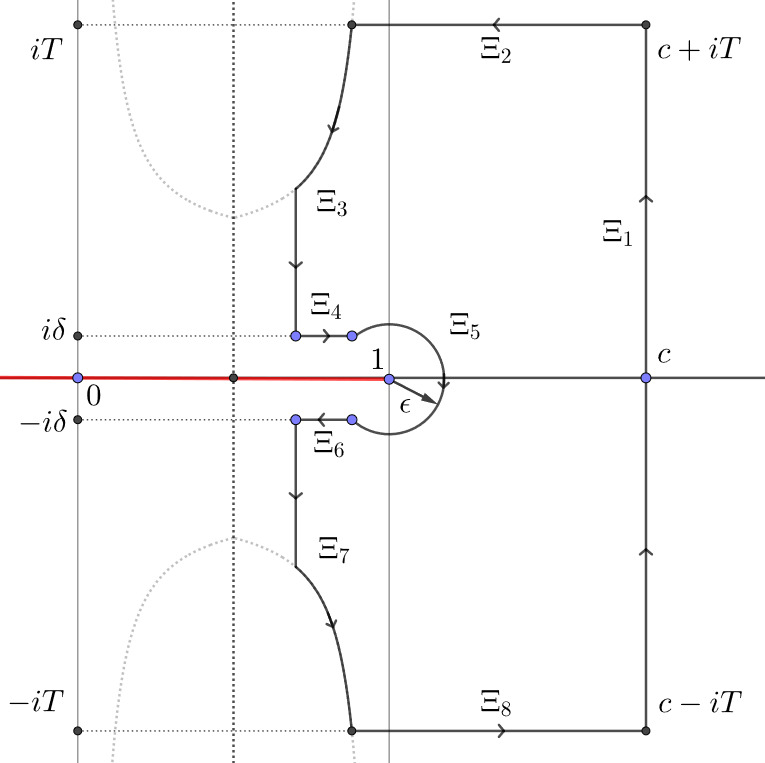}
	\caption{The contour of integration $\Xi$ alongside the zero-free region of $\zeta(s)$.}
	\label{fig:integrationcontour}
\end{figure}

By construction, the integrand is analytic in the interior of $\Xi$. 
Thus Cauchy's theorem implies
\begin{align} \label{cauchy}
	0 
	&= 
	\frac{1}{2 \pi i} \oint_{\Xi} X^s (-\log (s-1))^\alpha  (\mathcal{D}(s))^\beta \zeta(s+1) \Gamma(s+m) ds \nonumber \\
	&= 
	\frac{1}{2 \pi i}\bigg(\int_{\Xi_1} + \int_{\Xi_2} + \cdots + \int_{\Xi_8}\bigg) X^s (-\log (s-1))^\alpha  (\mathcal{D}(s))^\beta \zeta(s+1) \Gamma(s+m) ds.
\end{align}

We denote for $1\leq j\leq 8$
\begin{align}
	\Omega_j
	:=
	\frac{1}{2 \pi i} \int_{\Xi_j} X^s (-\log (s-1))^\alpha  (\mathcal{D}(s))^\beta \zeta(s+1) \Gamma(s+m) ds.
\end{align}
It is straightforward to see from the above bounds that
$
\Omega_2+\Omega_3+\Omega_7+\Omega_8
=
O(X^{3/4+\epsilon})
$.
It remains to look at $\Omega_4$, $\Omega_5$ and $\Omega_6$.
All factors in the integrands of these $\Omega_j$'s are holomorphic in a neighborhood of $\{s= \sigma \, : \, 3/4\leq \sigma \le 1\}$, 
except, of course, the logarithm $\log(s-1)$.
However, we have for all $\sigma< 1$
\begin{align}
\lim_{\delta\to 0} \log(1-\sigma+i\delta) = \log|\sigma-1| + i\pi, \quad \lim_{\delta\to 0} \log(1-\sigma-i\delta) = \log|\sigma-1| - i\pi.
\end{align}
In other words the logarithm $\log(s-1)$ has a continuous limit if one approaches the branch cut from above or below.
Also, the integral over $\Xi$ is independent of $\delta$ and $\epsilon$.
Thus we can let $\delta\to0$ and then $\epsilon\to0$.
It is straightforward to see that $\Omega_5\to0$ as $\epsilon\to 0$.
Further, we get with the parametrization $s=1-u$ and dominated convergence that
\begin{align*}
	\Omega_4
	&= 
	\frac{1}{2 \pi i}\int_{0}^{\frac{1}{4}} X^{1-u} (- \log u + i \pi)^\alpha (\mathcal{D}(1-u))^\beta \zeta(2 - u) \Gamma(m+1 - u) du,\\
	\Omega_6
	&= 
	-\frac{1}{2 \pi i}\int_{0}^{\frac{1}{4}} X^{1-u} (- \log u - i \pi)^\alpha (\mathcal{D}(1-u))^\beta \zeta(2 - u) \Gamma(m+1 - u) du.
\end{align*}
Since $\Omega(m, X, \alpha, \beta) =\Omega_4+\Omega_6+O(X^{3/4+\epsilon})$, we get
\begin{align*}
\Omega(m, X, \alpha, \beta)
=
\frac{1}{2 \pi i}\int_{0}^{\frac{1}{4}} X^{1-u} \vartheta(u,\alpha)  (\mathcal{D}(1-u))^\beta \zeta(2 - u) \Gamma(m+1 - u) du
+O(X^{3/4+\epsilon})
\end{align*}
with $\vartheta(u,\alpha) = (- \log u + i \pi)^\alpha - (- \log u - i \pi)^\alpha $.
%
%\begin{align*}
%(- \log u + i \pi)^\alpha - (- \log u - i \pi)^\alpha 
%=
%\sum_{0\leq m\leq \frac{\alpha-1}{2} } \binom{\alpha}{2m+1} (-1)^m\pi^{2m} (- \log u )^{\alpha-2m-1} 
%%2i\sum_{0\leq m\leq \frac{\alpha-1}{2} } \binom{\alpha}{2m+1} (-1)^m\pi^{2m} (- \log u )^{\alpha-2m-1} .
%\end{align*}
We now have 
\begin{align}
	\vartheta(u,\alpha) = 2i \sum_{h=0}^\alpha \binom{h}{\alpha}(-1)^{h}(\log u)^h \imag[(i\pi)^{\alpha-h}].
\end{align}
Thus 
\begin{align}
	\Omega(m, X, \alpha, \beta)
	&=
	2i \sum_{h=0}^\alpha \binom{h}{\alpha}(-1)^{h}(\log u)^h \imag[(i\pi)^{\alpha-h}] 	I(\alpha,\beta,h)
+O(X^{3/4+\epsilon})
\intertext{with }
I(\alpha,\beta,h) 
&=
\int_{0}^{\frac{1}{4}} X^{1-u} (- \log u )^{h} (\mathcal{D}(1-u))^\beta \zeta(2 - u) \Gamma(m+1 - u) du.
\label{eq:Omega_with_I(alpha,beta,h)}
\end{align}
Using Lemma~\ref{lem:Phi_asympt_principal_aux_lemma}, we get
\begin{align*}
I(\alpha,\beta,h) 
=
\frac{(\mathcal{D}(1))^\beta \zeta(2) \Gamma(m+1)}{\log X}\sum_{k=0}^{h} \binom{h}{k} (\log\log X)^{h-k}  (-1)^{k}\Gamma^{(k)}(1)
+
O\left(\frac{(\log\log X)^h}{(\log X)^2}\right).
\end{align*}
Inserting this expression for $I(\alpha,\beta,h)$ into
\eqref{eq:Omega_with_I(alpha,beta,h)} and then
substituting the resulting expression for $\Omega(m, X, \alpha, \beta)$ into \eqref{auxintegral} completes the proof.
\end{proof}

\begin{corollary}
\label{cor:P_r_spcial_cases}
For $r \in \{1,2,3,4\}$ one has
\begin{align}
	P_1(y)
	&= 
	1, \qquad
	P_2(y)
	= 
	2 \left( y+ M \right), \qquad
	P_3(y)
	= 
	3(  y^2+ 2My +M^2 - \zeta(2) ), \nonumber \\
	P_4(y)
	&= 
	4 (y^3 +3My^2+ 3({M^2} - \zeta (2))y + 2\zeta (3) - 3\zeta (2)M + {M^3}  ), \nonumber
\end{align}
%\begin{align}
%P_1(y)
%&= 
%1 \nonumber \\
%P_2(y)
%&= 
%2 \left( y+ M \right) \nonumber \\
%P_3(y)
%&= 
%3(  y^2+ 2My +M^2 - \zeta(2) ) \nonumber \\
%P_4(y)
%&= 
%4 (y^3 +3My^2+ 3({M^2} - \zeta (2))y + 2\zeta (3) - 3\zeta (2)M + {M^3}  ), \nonumber
%\end{align}
where $M = \gamma - \mathcal{D}(1) \approx 0.2614972 \ldots$ is the Meissel-Mertens constant.
\end{corollary}
%
%
%\begin{corollary}
%For $r \in \{1,2,3,4\}$ one has
%\begin{align}
%{\left( {\rho \frac{d}{{d\rho }}} \right)^m}{\Phi _{{\mathbb{P}_1}}}(\rho ) 
%&= 
%1 \times \zeta (2)\Gamma (m + 1)\frac{{{X^{m + 1}}}}{{\log X}} \nonumber \\
%{\left( {\rho \frac{d}{{d\rho }}} \right)^m}{\Phi _{{\mathbb{P}_2}}}(\rho ) 
%&= 
%2 \times \zeta (2)\Gamma (m + 1)\frac{{{X^{m + 1}}}}{{\log X}}\left( \log \log X+ M \right) \nonumber \\
%{\left( {\rho \frac{d}{{d\rho }}} \right)^m}{\Phi _{{\mathbb{P}_3}}}(\rho ) 
%&= 
%3 \times \zeta (2)\Gamma (m + 1)\frac{{{X^{m + 1}}}}{{\log X}}(  (\log \log X)^2+ 2M\log \log X +M^2 - \zeta (2) ) \nonumber \\
%{\left( {\rho \frac{d}{{d\rho }}} \right)^m}{\Phi _{{\mathbb{P}_4}}}(\rho ) 
%&= 
%4 \times \zeta (2)\Gamma (m + 1)\frac{{{X^{m + 1}}}}{{\log X}}( 2\zeta (3) - 3\zeta (2)M + {M^3} + 3({M^2} - \zeta (2))\log \log X \nonumber \\
%&\quad \quad \quad \quad \quad \quad \quad \quad \quad \quad \quad \quad \quad \quad \, + 3M{{(\log \log X)}^2} + {{(\log \log X)}^3} ), \nonumber
%\end{align}
%where $M = \gamma - \mathcal{D}(1) \approx 0.2614972 \ldots$ is the Meissel-Mertens constant.
%\end{corollary}

\begin{theorem}
\label{thm:Phi_asympt_principal_Lambda}
Let $\rho = e^{-1/X}$ and $r\in\N$ . 
Then there exists a polynomial $\widetilde{P}_r$ of degree $r-1$ and leading coefficient $r$ such that we have for any $m \in \Z_{\ge 0}$ 
and as $\rho \to 1^{-}$
\begin{align}
	\bigg(\rho \frac{d}{d\rho}\bigg)^m \Phi_{\Lambda^{*r}}(\rho) 
	&= 
	\zeta(2)\Gamma(m+1) X^{m+1}\widetilde{P}_r(\log X) \bigg(1+O\bigg(\frac{1}{\log X}\bigg)\bigg),
	\label{eq:Phi_asympt_principal_Lambda}\\
	\Phi_{\Lambda^{*r}}^{(n)}(\rho) 
	&=
	\zeta(2)\Gamma(m+1) X^{m+1}\widetilde{P}_r(\log X) \bigg(1+O\bigg(\frac{1}{\log X}\bigg)\bigg).
	\label{eq:Phi_asympt_principal_Lambda2}
\end{align}
\end{theorem}
As in Theorem~\ref{thm:Phi_asympt_principal}, we can give a more precise expression than in \eqref{eq:Phi_asympt_principal_Lambda} and \eqref{eq:Phi_asympt_principal_Lambda2}.
For this we need
\begin{align}
F_m(s)
:=
\zeta(s+1) \left(-\frac{\zeta'}{\zeta}(s)\right)^r \Gamma(s+m).
\end{align}
We then have
\begin{align}
\bigg(\rho \frac{d}{d\rho}\bigg)^m \Phi_{\Lambda^{*r}}(\rho) 
&= 
\res_{s=1} X^{m+s} F_m(s)
+
\sum_\varpi \res_{s=\varpi} X^{m+s} F_m(s)
+O(X^{m+\epsilon}),
\label{eq:Phi_Lambda-precise}
\end{align}
where $\sum_\varpi$ denotes the sum over the non-trivial zeros of $\zeta(s)$.

\begin{proof}
From the definition of $\Phi_{\Lambda^{*r}}$ and that $\rho=\exp(-1/X)$, we can write
\begin{align}
\Phi_{\Lambda^{*r}}(\rho) 
= 
\sum_{j=1}^\infty \sum_{\textbf{n}_r}\frac{\Lambda(\textbf{n}_r)}{j} \rho^{j\mathcal{J}(\textbf{n}_r)}
=
\sum_{j=1}^\infty \sum_{n=1}^\infty \frac{\Lambda^{*r}(n)}{j} \exp\bigg( - \frac{jn}{X}\bigg).
\end{align}
This implies 
\begin{align}
\bigg(\rho \frac{d}{d\rho}\bigg)^m \Phi_{\Lambda^{*r}}(\rho) 
=
\sum_{j=1}^\infty \sum_{n=1}^\infty j^{m-1}n^m \Lambda^{*r}(n) \exp\bigg( - \frac{jn}{X}\bigg).
\end{align}
Using the Cahen-Mellin formula again on $\rho = e^{-1/X}$ we can write the above as
\begin{align}
\Phi_{\Lambda^{*r}}(\rho)  
&= 
\frac{1}{2 \pi i} \int_{(c+m)} X^s \Big(\sum_{j=1}^\infty \frac{1}{j^{s+1-m}}\Big) \left(\sum_{n=1}^\infty \frac{\Lambda^{*r}(n)}{n^{s-m}} \right) \Gamma(s)  ds\nonumber\\
&= 
\frac{1}{2 \pi i} \int_{(c+m)} X^s \zeta(s+1-m) \left(-\frac{\zeta'}{\zeta}(s-m)\right)^r \Gamma(s)  ds\nonumber\\
&= 
\frac{1}{2 \pi i} \int_{(c)} X^s \zeta(s+1) \left(-\frac{\zeta'}{\zeta}(s)\right)^r \Gamma(s+m)  ds,
\label{eq:cahen-mellin_Lambda}
\end{align}
for $c>1$. 
We used in this computation that Dirichlet series $\zeta(s+1)$ and $(-\frac{\zeta'}{\zeta}(s))^r$ converge absolutely and uniformly for $\Re(s) \ge 1+\delta$ where $\delta$ is any positive number. 
The integrand in \eqref{eq:cahen-mellin_Lambda} is meromorphic on
the whole complex plane.
Further, $(-\frac{\zeta'}{\zeta}(s))^r$ has a pole of order $r$ at $s=1$ and and poles at the non-trivial zeros $\varpi$ of $\zeta$. 
Also, $\zeta(s+1)$ has a pole at $s=0$ and $\Gamma(s+m)$ has poles at all negative integers less than or equal to $m$.

We thus shift the contour to $\Re(s)=- \frac{1}{2}$.
By doing this we get
\begin{align}
	\Phi_{\Lambda^{*r}}(\rho)  
	&= 
	\res_{s=1} X^{m+s} F_m(s)
	+
	\sum_\varpi \res_{s=\varpi} X^{m+s} F_m(s)
	+
	\res_{s=0} X^{m+s} F_m(s)\nonumber\\
	&\quad+
	\frac{1}{2 \pi i} \int_{(-\frac{1}{2})} X^s \zeta(s+1) \left(-\frac{\zeta'}{\zeta}(s)\right)^r \Gamma(s+m)  ds.
	\label{eq:cahen-mellin_Lambda3}
\end{align}
A direct estimate then gives \eqref{eq:Phi_Lambda-precise}.
Furthermore, we have
\begin{align*}
	X^{s+m}
	=
	X^{m+1} e^{(s-1)\log X}
	=
	X^{m+1}\sum_{\ell=0}^{\infty}\frac{(\log X)^\ell}{k!}(s-1)^\ell.
\end{align*}
Inserting this expression into the computation of the residue at $s=1$ gives the main term in \eqref{eq:Phi_asympt_principal_Lambda}.
It remains to show that the sum over the non-trivial zeros is of lower order. However, this computation is straightforward and we thus omit the details.
\end{proof}

\begin{corollary}
	\label{cor:tilde_P_r_explicit}
We have 
\begin{align}
	\widetilde{P}_1(y) 
	&= 
	1, \quad
	\widetilde{P}_2(y) 
	= 
	y+\bigg(\frac{\zeta '(2)}{\zeta(2)}-3 \gamma   \bigg), \nonumber \\
	\widetilde{P}_3(y)
	&= 
	\frac{y^2}{2} + \bigg(\frac{\zeta '(2)}{\zeta(2)} -4 \gamma  \bigg) y+
	\bigg(6 \gamma _1 +\frac{1}{2} \frac{\zeta ''(2)}{\zeta(2)}-4 \gamma  \frac{\zeta '(2)}{\zeta(2)} +\frac{19}{2}
	\gamma ^2  + \frac{1}{2} \zeta(2)\bigg).\nonumber
\end{align}
\end{corollary}

\subsection{The asymptotic orders of magnitude}
\label{sec:the-asymptotic-orders-of-magnitude}
Later on we will need estimates for the auxiliary functions involved in the main theorems of Section~\ref{sec:introduction}. 
Recall that the expressions for $\Phi_{\Pri_r}(\rho)$ and $\Phi_{\Lambda^{*r}}(\rho)$ as integrals in \eqref{eq:integralarcs} are valid for any real $\rho <1$. 

Now let $x$ be a large real keeping in mind that we will choose $x = n$ for $\partition_{\Pri_r}$ and $\partition_{\Lambda^{*r}}$. 
We will choose $\rho$ and $\widehat{\rho}$ such that
\begin{align}
x= \rho\,\Phi'_{\Pri_r} (\rho)
\ \text{ and } \
x= \widetilde{\rho}\,\Phi'_{\Lambda^{*r}} (\widetilde{\rho})
\label{eq:def_saddle_point}
\end{align}

Since the coefficients of $\Phi_{\Pri_r}$ and $\Phi_{\Lambda^{*r}}$ are all positive,
it follows that the relationship between $x$ and $\rho$ is well-defined, injective and that $\rho \to 1^-$ and $\widetilde{\rho} \to 1^-$ as $x \to \infty$.
\begin{proposition} 
\label{propositionmagnitude}
Denote by $X= (\log(1/\rho))^{-1}$.
We set
\begin{align}
Q(x)
:=
\left(\frac{\log x + \log \log x - \log 2 - \log P_r(\log\log x-\log2) -\log\zeta(2)}{2\zeta(2) P_r\left(\log \log x - \log 2 \right) }\right)^{\frac{1}{2}}
\end{align}

One has as $x\to\infty$ that
\begin{align} \label{eq:magnitudes0}
X 
&= 
x^{1/2}Q(x)\bigg(1+ O\bigg(\frac{1}{\log x}\bigg)\bigg).
%&=
%x^{1/2}\left(\frac{\log x + \log \log x - \log 2 - \log P_r(\log\log x-\log2) -\log(\zeta(2))}{2\zeta(2) P_r\left(\log \log x - \log 2 \right) }\right)^{\frac{1}{2}}\bigg(1+ O\bigg(\frac{1}{\log x}\bigg)\bigg).
\end{align}
Furthermore,
\begin{align} \label{magnitudes1}
x \log \frac{1}{\rho(x)} 
&= 
x^{1/2}\big(Q(x)\big)^{-1}\bigg(1+ O\bigg(\frac{1}{\log x}\bigg)\bigg).
%&=
%x^{1/2}\left(\frac{2\zeta(2) P_r\left(\log \log x - \log 2 \right) }{\log x + \log \log x - \log 2 - \log P_r(\log\log x-\log2) -\log(\zeta(2))}\right)^{\frac{1}{2}}\bigg(1+ O\bigg(\frac{1}{\log x}\bigg)\bigg).
\end{align}
as $x \to \infty$. For the function $\Phi_{\Pri_r}$ and its derivatives, one has that for all $m\in\N_0$
\begin{align} \label{magnitudes3}
\bigg(\rho \frac{d}{d\rho}\bigg)^m \Phi_{\Pri_r}(\rho) 
&=
\Gamma(m+1)x^{\frac{m+1}{2}} \big(Q(x)\big)^{m-1}\bigg(1+ O\bigg(\frac{1}{\log x}\bigg)\bigg)
%\nonumber \\
%&= \Gamma(m+1)x^{\frac{m+1}{2}}\left(\frac{\log x + \log \log x - \log 2 - \log P_r(\log\log x-\log2) -\log(\zeta(2))}{2\zeta(2) P_r\left(\log \log x - \log 2 \right) }\right)^{\frac{m-1}{2}}\bigg(1+ O\bigg(\frac{1}{\log x}\bigg)\bigg).
\end{align}
as $x \to \infty$.
\end{proposition}	

\begin{proof}
Let us suppose that $x$ is sufficiently large in which case $\rho$ will be very close to $1$ and so $X$ is also large. 
We have shown in Theorem~\ref{thm:Phi_asympt_principal} that
\begin{align} 
\label{asympaux01}
x = 
\rho \frac{d}{d\rho} \Phi_{\mathbb{P}_r}(\rho) 
&= 
\frac{\zeta(2)X^{2}P_r(\log\log X)}{\log X} \bigg(1+O\bigg(\frac{1}{\log X}\bigg)\bigg).
\end{align}
Taking the logarithm of \eqref{asympaux01} implies that
\begin{align}
\log x &= 2 \log X  - \log \log X + \log P_r(\log\log X)+ \log\zeta(2)  + O\bigg(\frac{1}{\log X}\bigg).
\label{eq:asympaux02}
\end{align}
This implies that $\log x\ll \log X \ll \log x$. 
Furthermore, taking the logarithm of \eqref{eq:asympaux02} gives
\begin{align}
\log \log x &= \log \log X + \log 2 +O\left(\frac{\log \log X}{\log X}\right), 
\label{eq:loglog_x} \\
\log \log \log x &= \log \log \log X + O\left(\frac{1}{\log \log X }\right).
\label{eq:logloglog_x}
\end{align}
Solving \eqref{eq:loglog_x}  for $\log \log X$ and entering this into $P_r$ gives
\begin{align}
P_r(\log\log X)
&=
%P_r\left(\log \log x - \log 2 +O\left(\frac{\log \log x}{\log x}\right) \right)
%=
P_r\left(\log \log x - \log 2 \right)\left(1+O\left(\frac{1}{\log x}\right)\right).
\label{eq:P(loglog_x)}
\end{align}
Also, solving \eqref{eq:asympaux02} for $\log X$ and using \eqref{eq:loglog_x} and \eqref{eq:P(loglog_x)} gives
\begin{align*}
\log X 
&= 
%\frac{1}{2}\left( \log x  +\log \log X -\log P_r(\log\log X)- \log(\zeta(2))  + O\bigg(\frac{1}{\log x}\bigg)\right)\\
%&= 
\frac{1}{2}\left( \log x  +\log \log x -\log 2 -\log P_r(\log\log x-\log2)- \log\zeta(2)  + O\bigg(\frac{1}{\log x}\bigg)\right).
\end{align*}
We now plug these into \eqref{asympaux01} so that
\begin{align*}
x 
%&= 
%\frac{\zeta(2)X^2 P_r\left(\log \log x - \log 2 \right)\left(1+O\left(\frac{1}{\log x}\right)\right) }{\frac{1}{2}\log x + \frac{1}{2}\log \log x - \frac{1}{2}\log2 - \frac{1}{2}\log P_r(\log\log x-\log2) -\frac{\log(\zeta(2))}{2}+O(\frac{1}{ \log x})}\bigg(1+ O\bigg(\frac{1}{\log x}\bigg)\bigg)\\
&=
\frac{2\zeta(2)X^2 P_r\left(\log \log x - \log 2 \right) }{\log x + \log \log x - \log 2 - \log P_r(\log\log x-\log2) -\log\zeta(2)}\bigg(1+ O\bigg(\frac{1}{\log x}\bigg)\bigg).
\end{align*}
Solving for $X$ gives \eqref{eq:magnitudes0}.
For the proof of \eqref{magnitudes1}, we just have to combine $x \log \frac{1}{\rho} = xX^{-1}$ and \eqref{eq:magnitudes0}.
It remains to prove \eqref{magnitudes3}. Using Theorem~\ref{thm:Phi_asympt_principal} and \eqref{eq:magnitudes0}, we can  write
\begin{align*} %\label{asympaux03}
\bigg(\rho \frac{d}{d\rho}\bigg)^m \Phi_{\mathbb{P}_r}(\rho) 
&= 
\frac{\zeta(2)\Gamma(m+1) X^{m+1}P_r(\log\log X)}{\log X} \bigg(1+O\bigg(\frac{1}{\log X}\bigg)\bigg)\\
&=
\frac{\zeta(2)\Gamma(m+1) x^{\frac{m+1}{2}} \big(Q(x)\big)^{m+1}P_r(\log\log x- \log 2)}{\log X} \bigg(1+O\bigg(\frac{1}{\log X}\bigg)\bigg).
\end{align*}
Inserting the expression for $\log X$ and the definition of $Q(x)$ completes the proof.
\end{proof}

\begin{proposition} 
\label{propositionmagnitude_lambda}
Denote by $X= (\log(1/\widetilde{\rho}))^{-1}$.
One has as $x\to\infty$ that
\begin{align} \label{eq:magnitudes0_lambda}
X 
&= 
x^{\frac{1}{2}}\left(\zeta(2)\widetilde{P}_r\left(\frac{\log x}{2}\right)\right)^{-\frac{1}{2}}
\left(1+O\left(\frac{(\log\log x)^{r-1}}{\log x}\right)\right).
\end{align}
Furthermore,
\begin{align} \label{magnitudes1_lambda}
x \log \frac{1}{\widetilde{\rho}(x)} 
&= 
x^{\frac{1}{2}}\left(\zeta(2)\widetilde{P}_r\left(\frac{\log x}{2}\right)\right)^{\frac{1}{2}}
\left(1+O\left(\frac{(\log\log x)^{r-1}}{\log x}\right)\right)
\end{align}
as $x \to \infty$. For the function $\Phi_{\Lambda^{*r}}$ and its derivatives, one has that for all $m\in\N_0$
\begin{align} 
\label{magnitudes3_lambda}
\bigg(\widetilde{\rho} \frac{d}{d\widetilde{\rho}}\bigg)^m \Phi_{\Lambda^{*r}}(\widetilde{\rho}) 
&=
\Gamma(m+1) 
x^{\frac{m+1}{2}}\left(\zeta(2)\widetilde{P}_r\left(\frac{\log x}{2}\right)\right)^{-\frac{m-1}{2}}
\left(1+O\left(\frac{(\log\log x)^{r-1}}{\log x}\right)\right)
\end{align}
as $x \to \infty$.
\end{proposition}	

\begin{proof}
Let us suppose that $x$ is sufficiently large in which case $\rho$ will be very close to $1$ and so $X$ is also large. 
We have shown in Theorem~\ref{thm:Phi_asympt_principal_Lambda} that
\begin{align} 
\label{eq:asympaux01_lambda}
x 
= 
\widetilde{\rho} \frac{d}{d\widetilde{\rho}}\Phi_{\Lambda^{*r}}(\widetilde{\rho}) 
&= 
\zeta(2) X^{2}\widetilde{P}_r(\log X) \bigg(1+O\bigg(\frac{1}{\log X}\bigg)\bigg).
\end{align}
Taking the logarithm of \eqref{asympaux01} implies that
\begin{align}
\log x &= 2 \log X  + \log \widetilde{P}_r(\log X)+ \log\zeta(2)  + O\bigg(\frac{1}{\log X}\bigg).
\label{eq:asympaux02_lambda}
\end{align}
This implies that $\log x\sim 2\log X$. 
Solving \eqref{eq:asympaux02_lambda} for $\log X$ gives
\begin{align}
\log X 
%&= 
%\frac{1}{2} \log x  - \log \widetilde{P}_r(\log X)+ \log(\zeta(2))  - O\bigg(\frac{1}{\log x}\bigg)\\
%&= 
%\frac{1}{2} \log x  +O\bigg((\log\log x)^{r-1}\bigg).\\
&= 
\frac{1}{2} \log x  \left(1+O\left(\frac{(\log\log x)^{r-1}}{\log x}\right)\right).
\label{eq:asympaux0e_lambda}
\end{align}
This implies
\begin{align}
\widetilde{P}_r(\log X)
&=
\widetilde{P}_r\left(\frac{\log x}{2}\right)\left(1+O\left(\frac{(\log\log x)^{r-1}}{\log x}\right)\right).
\label{eq:P(log_x)_lambda}
\end{align}
We now plug these into \eqref{eq:asympaux01_lambda} and solving for $x$ gives
\begin{align*}
X
&=
x^{\frac{1}{2}}
\left(\zeta(2)\widetilde{P}_r\left(\frac{\log x}{2}\right)\right)^{-\frac{1}{2}}
\left(1+O\left(\frac{(\log\log x)^{r-1}}{\log x}\right)\right)
\end{align*}
and therefore \eqref{eq:magnitudes0_lambda}
For the proof of \eqref{magnitudes1_lambda}, we have just to combine $x \log \frac{1}{\widetilde{\rho}} = xX^{-1}$ and \eqref{eq:magnitudes0_lambda}.
It remains to prove \eqref{magnitudes3_lambda}. 
For this, we have only to plug in the above expression into Theorem~\ref{thm:Phi_asympt_principal_Lambda}.
\end{proof}

\subsection{The saddle point method and proofs of main theorems}\label{sec:the-saddle-point-method-and-proofs-of-main-theorems}

We are now in a position to conclude the proof of the main theorems of Section~\ref{sec:introduction}. 
\begin{theorem} 
\label{thmtheorem1over5}
Let  $\Psi_{\Pri_r}(z)$ and $\Psi_{\Lambda^{*r}}(z)$ be as in \eqref{eq:def_gen_Pr} and \eqref{eq:def_gen_Lambda_r}.
Further, let $\rho$ and $\widetilde{\rho}$ be the solutions of the equations 
\begin{align*}
n= \rho\,\Phi'_{\Pri_r} (\rho)
\ \text{ and } \
n= \widetilde{\rho}\,\Phi'_{\Lambda^{*r}} (\widetilde{\rho}).
\end{align*}
We then have 
\begin{align*}
\partition_{\Pri_r}(n)
= 
\frac{\rho^{-n} \Psi_{\Pri_r}(\rho)}{(2\pi\Phi_{\Pri_r,(2)}(\rho))^{\frac{1}{2}}}(1+O(n^{-\frac{1}{6}}))
\ \text{ and } \
\partition_{\Lambda^{*r}}(n)
= 
\frac{(\widetilde{\rho})^{-n} \Psi_{\Lambda^{*r}}(\widetilde{\rho})}{(2\pi\Phi_{\Lambda^{*r},(2)}(\widetilde{\rho}))^{\frac{1}{2}}}(1+O(n^{-\frac{1}{6}}))
\end{align*}
as $n\to\infty$ with $\Phi_{\Pri_r,(2)} (\rho)= (\rho \frac{d}{d\rho})^2\,\Phi_{\Pri_r}(\rho)$ and $\Phi_{\Lambda^{*r},(2)} (\rho)= (\rho \frac{d}{d\rho})^2\,\Phi_{\Lambda^{*r}}(\rho)$.
\end{theorem}
\begin{proof}
The proof of this theorem uses the saddle-point method, which is fairly standard and well-known.
We will therefore only give a brief overview and only for $\partition_{\Pri_r}(n)$. 
The reader can find a complete analogous proof with all the details in \cite[Section~7]{semiprimes}.

Recall that $\partition_{\Pri_r}(n)$ is given by the integral
\begin{align*}
\partition_{\Pri_r}(n)
=
\rho^{-n}\int_{-\frac{1}{2}}^{\frac{1}{2}}\Psi_{\Pri_r}(\rho \e(\alpha))\e(-n\alpha) d\alpha 
= 
\rho^{-n}\int_{-\frac{1}{2}}^{\frac{1}{2}}\exp\big(\Phi_{\Pri_r}(\rho \e(\alpha))\big)\e(-n\alpha)) d\alpha.    
\end{align*}
Recall, we have defined  $\mathfrak{M}(1,0)$ in \eqref{eq:def_majyor_M_and_minor_m} as 
$
\mathfrak{M}(1,0) = \{\alpha\in[-\tfrac{1}{2},\tfrac{1}{2}]; |\alpha|\leq X^{-1}(\log X)^{A}\},
$
with $A>8r$ arbitrary, but fixed.
The first step is to show that the integral over $[-\pi,\pi]\setminus\mathfrak{M}(1,0)$ does not contribute to the main term.
Combining Lemma~\ref{lem:cul_minor_arcs} and Lemma~\ref{lem:culminationnonprincipal} immediately get 
\begin{align}
\Re(\Phi_{\Pri_r}(\rho \ee(\alpha))) 
\leq  
\frac{3}{4}\Phi_{\Pri_r}(\rho) 
\end{align}
for $\alpha\in[-\pi,\pi]\setminus\mathfrak{M}(1,0)$. 
Inserting this it the integral immediately gives 
\begin{align*}
\bigg|\int_{[-\frac{1}{2}, \frac{1}{2}]\setminus \mathfrak{M}(1,0)}
\exp(\Phi_{\Pri_r}(\rho \e(\alpha)))\e(-n\alpha)d\alpha \bigg|
%\ll
%\int_{[-\frac{1}{2}, \frac{1}{2}]\setminus \mathfrak{M}(1,0)}
%\exp(\Re(\Phi_{\Pri_r}(\rho \e(\alpha))))d\alpha
\ll
n^{-B}\exp\left(\Phi_{\Pri_r}(\rho)\right),
\end{align*}
where $B\geq 1$ can be chosen arbitrarily large.
It thus remains to consider the integral over $\mathfrak{M}(1,0)$.
We now spilt $\mathfrak{M}(1,0)$ into the three regions
\begin{align*}
I_1 =\{\alpha\in\mathfrak{M}(1,0);\,|\alpha|\leq \eta \},\ 
I_2 =\{\alpha\in\mathfrak{M}(1,0);\,\eta <|\alpha|\leq \beta\} \ \text{ and } \
I_3 =\{\alpha\in\mathfrak{M}(1,0);\,\beta<|\alpha| \}
\end{align*}
with $\eta = X^{-17/12}$ and $\beta = (9\pi X)^{-1}$.
The next step is to show that the integrals over $I_2$ and $I_3$ are of lower order.
%We begin with $I_3$.
Using that $(1+r)^{-1}\leq 1- 2r$ for $0\leq r\leq 1/2$, we obtain for $\alpha\in I_3$ that
\begin{align*}
\Re\left(\frac{1}{1+ 2\pi i \alpha X}\right)
=
\frac{1}{1+ 4\pi^2 \alpha^2 X^2}
\leq 
\frac{1}{1+ 4\pi^2 \beta^2 X^2}
\leq 
1- 2\pi^2 \beta^2 X^2
=
\frac{79}{81}.
\end{align*}
Combining this with Corollary~\ref{cor:lemma5.1}, we get for $X$ large enough and all $\alpha\in I_3$ that
$
\Re\big(\Phi_{\Pri_r}(\rho \e(\alpha))\big) 
\leq 
\frac{80}{81} \Phi_{\Pri_r}(\rho)
$
As above, this shows that the integral over $I_3$ is also of lower order.
For our next task we look at the integral over $I_2$. 
Here we use the Taylor approximation of $\Phi_{\Pri_r}(\rho \e(\alpha))$ since it is more precise in the central region than Corollary~\ref{cor:lemma5.1}. 
We have
\begin{align}
\Phi_{\Pri_r}(\rho \e(\alpha)) 
&= 
\Phi_{\Pri_r}(\rho)
+2\pi i\alpha a_n
-2\pi^2\alpha^2 b_n +4\pi^3 R(\rho,\alpha)
\label{eq:taylortheorem2}
\end{align} 
with $a_n:= \rho\Phi_{\Pri_r}'(\rho)$,   $b_n:= \rho\Phi_{\Pri_r}'(\rho)+\rho^2\Phi_{\Pri_r}''(\rho)$ and
$
\max\{|\real R(\rho,\alpha)|, |\imag R(\rho,\alpha)|\}
\leq 
|\alpha|^3 c_n,
%	=
%	O\left(\alpha^3 n^2 \right).
$
where $c_n:=\rho \Phi_{\Pri_r}'(\rho)+ \rho^2 \Phi_{\Pri_r}''(\rho)+\rho^3 \Phi'''_{\Pri_r}(\rho)$. 
Theorem~\ref{thm:Phi_asympt_principal} implies that 
\begin{align*}
b_n\sim 2\frac{\zeta(2)r X^{3}(\log \log X )^{r-1}}{\log X}
\quad \text{and} \quad
c_n\sim 6\frac{\zeta(2)r X^{4}(\log \log X )^{r-1} }{\log X}.
\end{align*}
Let $0<\epsilon<\frac{1}{5}$ be arbitrary. Since $\alpha \in I_2$ (and  thus $ X^{-17/12}\leq |\alpha|\leq (9\pi X)^{-1}$), we get for $X$ large enough that
\begin{align*}
\real (\Phi_{\Pri_r}(\rho \e(\alpha))) 
&\leq
\Phi_{\Pri_r}(\rho) 
-2\pi^2\alpha^2 b_n +4\pi^3 |\alpha|^3 c_n\\
&\leq 
\Phi_{\Pri_r}(\rho) 
-
2\pi^2\alpha^2\frac{\zeta(2)r X^{3}(\log \log X)^{r-1}}{\log X}\big((2-\epsilon)- (12+\epsilon)\pi |\alpha|X   \big)\\
%&\leq 
%\Phi_{\Pri_r}(\rho) 
%-
%2\pi^2\alpha^2\frac{\zeta(2)r X^{3}(\log \log X)^{r-1}}{\log X}\left( 3- 26\pi |\alpha|X   \right)\\
&\leq 
\Phi_{\Pri_r}(\rho) 
- \frac{(6-10\epsilon)2\pi^2\alpha^2	}{9}\frac{\zeta(2)r X^{3}(\log \log X)^{r-1} }{\log X}\\
&\leq 
\Phi_{{\Pri_r}}(\rho) 
- \frac{8\pi^2 X^{1/6}}{9}\frac{\zeta(2)r (\log \log X)^{r-1} }{\log X}.
\end{align*}
Equation \eqref{eq:magnitudes0} implies that $X^{1/6}  \asymp (\frac{n \log n}{\log\log n})^{1/12}$.
This implies that for $n$ large
\begin{align*}
\left|\int_{I_2}
\exp(\Phi_{\Pri_r}(\rho \e(\alpha)))\e(-n\alpha)d\alpha \right|
\leq 
\exp (\Phi_{\Pri_r}(\rho)  - n^{-1/13})
\ll
n^{-B}\exp (\Phi_{\Pri_r}(\rho))
\end{align*}
for any constant $B\geq 1$. 
The last step is to show that the integral over $I_1$ gives the main term.
The main idea is to approximate the integrand by a Gaussian. 
These arguments are completely standard and straightforward and we thus omit the details. 
\end{proof}

\section{Conclusion and future work} \label{sec:conclusion}
A natural question to ask is how to find partitions with respect to powerful $r$-primes, i.e. $m=p_1^{\alpha_1} \cdots p_r^{\alpha_r}$.
This will necessitate three main ingredients. For the principal major arcs, we will need to deal with integrands of the form $\zeta_{\mathcal{P}} (\alpha_1(s-m)) \cdots \zeta_{\mathcal{P}} (\alpha_r(s-m))$ in \eqref{auxintegral}. In addition, for the minor arcs we will need to bound exponential sums of the form \eqref{eq:generic_exp_sum} with $f(n)$ being the characteristic function of prescribed powerful $r$-primes. Lastly, for the non-principal major arcs the technique of pseudo-differentiable and strange functions discussed in Section \ref{sec:pseudo-differentiable-functions-and-non-principal-arcs} will have to be enlarged to accommodate this new extension. This will be the subject of future research.

\section{Acknowledgments} \label{sec:ack}
DZ was supported by the Leverhulme Trust Research Project Grant RPG-2021-129, and wishes to acknowledge Madhuparna Das for fruitful discussions during the early stages of the manuscript. The authors are grateful to the referees for making valuable suggestions that
have greatly increased the clarity and value of the manuscript.

%\section*{Data availability} Data sharing not applicable to this article as no datasets were generated or
%analysed during the current study.
\section*{Statements and declarations}
\begin{itemize}
\item \emph{Data availability}: Data sharing not applicable to this article as no datasets were generated or
analysed during the current study.
\item \emph{Research involving Human Participants and/or Animals}: Not applicable.
\item \emph{Disclosure of potential conflicts of interest}: Not applicable.
\end{itemize}

\bibliographystyle{abbrv}
\bibliography{literature_triple}

\end{document}